\documentclass[preprint]{elsarticle}
\makeatletter
\def\ps@pprintTitle{%
 \let\@oddhead\@empty
 \let\@evenhead\@empty
 \def\@oddfoot{\centerline{\thepage}}%
 \let\@evenfoot\@oddfoot}
\makeatother
\usepackage{amsmath,amsthm,amssymb}
\usepackage{cancel}
\usepackage{comment}
\usepackage[hmargin={28mm,28mm},vmargin={30mm,35mm}]{geometry}
\usepackage[bold]{hhtensor}
\usepackage[ocgcolorlinks,allcolors={blue}]{hyperref}
\usepackage{mathabx}
\usepackage{mathtools}
\usepackage{paralist}
\usepackage{pgfplots}

\usepackage{caption}
\usepackage{subcaption}
\usepackage{xcolor}

\graphicspath{{fig/pdf/}}

\newtheorem{theorem}{Theorem}
\newtheorem{proposition}[theorem]{Proposition}
\newtheorem{lemma}[theorem]{Lemma}

\theoremstyle{remark}
\newtheorem{remark}[theorem]{Remark}
\theoremstyle{definition}

\newtheorem{definition}[theorem]{Definition}

\newcommand{\doi}[1]{DOI:~\href{https://dx.doi.org/#1}{#1}}

\newcommand{\st}{\,:\,}

\newcommand{\ud}[1]{\,\mathrm{d}#1}

\newcommand{\Real}{\mathbb{R}}
\newcommand{\Natural}{\mathbb{N}}

\newcommand{\norm}[2][]{\|#2\|_{#1}}
\newcommand{\seminorm}[2][]{|#2|_{#1}}

\newcommand{\tnorm}[2][]{\vvvert #2\vvvert_{#1}}

\newcommand{\GRAD}{{\vec{\nabla}}}
\newcommand{\DIV}{\vec{\nabla}{\cdot}}
\newcommand{\LAPL}{{\triangle}}

\DeclareMathOperator{\oDIV}{div}
\DeclareMathOperator{\card}{card}

\newcommand{\SCAL}{{\cdot}}

\newcommand{\GD}[1][]{\mathcal{D}_{#1}}
\newcommand{\XDz}[1][]{X_{\GD[#1],0}}
\newcommand{\PiD}[1][]{\Pi_{\GD[#1]}}
\newcommand{\grD}[1][]{\vec{\nabla}_{\GD[#1]}}
\newcommand{\grT}{\vec{\mathcal{G}}_T}

\newcommand{\grDt}[1][]{\widetilde{\vec{\nabla}}_{\GD[#1]}}
\newcommand{\grTt}{\widetilde{\vec{\mathcal{G}}}_T}

\newcommand{\CD}[1][]{C_{\GD[#1]}}
\newcommand{\SD}[1][]{S_{\GD[#1]}}
\newcommand{\WD}[1][]{W_{\GD[#1]}}

\newcommand{\Mh}{\mathcal{M}_h}
\newcommand{\Th}[1][h]{\mathcal{T}_{#1}}
\newcommand{\Sh}[1][h]{\mathfrak{T}_{#1}}
\newcommand{\Fh}[1][h]{\mathcal{F}_{#1}}
\newcommand{\Fhi}[1][h]{\mathcal{F}_{#1}^{\rm i}}
\newcommand{\Fhb}[1][h]{\mathcal{F}_{#1}^{\rm b}}
\newcommand{\PTF}[1][TF]{P_{#1}}

\newcommand{\normal}{\vec{n}}

\newcommand{\Poly}[1]{\mathbb{P}^{#1}}

\newcommand{\lproj}[2][T]{\pi_{#1}^{0,#2}}
\newcommand{\vlproj}[2][T]{\vec{\pi}_{#1}^{0,#2}}

\newcommand{\eproj}[2][T]{\pi_{#1}^{1,#2}}

\newcommand{\UT}[1][k,l]{\underline{U}_T^{#1}}
\newcommand{\IT}[1][k,l]{\underline{I}_T^{#1}}

\newcommand{\Uh}[1][k,l]{\underline{U}_h^{#1}}
\newcommand{\Ih}[1][k,l]{\underline{I}_h^{#1}}

\newcommand{\Uhz}[1][k,l]{\underline{U}_{h,0}^{#1}}

\newcommand{\uu}[1][h]{\underline{u}_{#1}}
\newcommand{\uv}[1][h]{\underline{v}_{#1}}
\newcommand{\uvm}[1][h]{\underline{v}_{#1}^0}
\newcommand{\uw}[1][h]{\underline{w}_{#1}}

\newcommand{\huv}[1][T]{\widehat{\underline{v}}_{#1}}

\newcommand{\bphi}{\vec{\phi}}

\newcommand{\rT}[1][k+1]{\mathrm{r}_T^{#1}}

\newcommand{\GT}[1][k]{\vec{\mathrm{G}}_T^{#1}}
\newcommand{\ST}[1][]{\vec{\mathrm{S}}_T^{#1}}
\newcommand{\LTF}[1][k+1]{\vec{\mathrm{L}}_{TF}^{#1}}

\newcommand{\Gh}[1][k]{\vec{\mathrm{G}}_h^{#1}}

\newcommand{\LTFt}[1][k]{\widetilde{\vec{\mathrm{L}}}_{TF}^{#1}}

\newcommand{\STt}[1][]{\widetilde{\vec{\mathrm{S}}}_T^{#1}}

\newcommand{\RangeST}[1][]{\vec{\mathfrak S}_T^{#1}}

\newcommand{\dT}[1][l]{\delta_T^{#1}}
\newcommand{\dGT}[1][k]{\vec{\delta}_{\nabla,T}^{#1}}
\newcommand{\dTF}[1][k]{\delta_{TF}^{#1}}

\newcommand{\RT}[1][k+1]{\mathbb{RT}^{#1}}

\newcommand{\control}{\mathbf{\Phi}}

\newcommand{\Ndk}[1][d,k]{N_{#1}}

\newcommand{\logLogSlopeTriangle}[5]
{

    \pgfplotsextra
    {
        \pgfkeysgetvalue{/pgfplots/xmin}{\xmin}
        \pgfkeysgetvalue{/pgfplots/xmax}{\xmax}
        \pgfkeysgetvalue{/pgfplots/ymin}{\ymin}
        \pgfkeysgetvalue{/pgfplots/ymax}{\ymax}

        \pgfmathsetmacro{\xArel}{#1}
        \pgfmathsetmacro{\yArel}{#3}
        \pgfmathsetmacro{\xBrel}{#1-#2}
        \pgfmathsetmacro{\yBrel}{\yArel}
        \pgfmathsetmacro{\xCrel}{\xArel}

        \pgfmathsetmacro{\lnxB}{\xmin*(1-(#1-#2))+\xmax*(#1-#2)} 
        \pgfmathsetmacro{\lnxA}{\xmin*(1-#1)+\xmax*#1} 
        \pgfmathsetmacro{\lnyA}{\ymin*(1-#3)+\ymax*#3} 
        \pgfmathsetmacro{\lnyC}{\lnyA+#4*(\lnxA-\lnxB)}
        \pgfmathsetmacro{\yCrel}{\lnyC-\ymin)/(\ymax-\ymin)} 

        \coordinate (A) at (rel axis cs:\xArel,\yArel);
        \coordinate (B) at (rel axis cs:\xBrel,\yBrel);
        \coordinate (C) at (rel axis cs:\xCrel,\yCrel);

        \draw[#5]   (A)-- node[pos=0.5,anchor=north] {\scriptsize{1}}
                    (B)-- 
                    (C)-- node[pos=0.,anchor=west] {\scriptsize{#4}} 
                    cycle;
    }
}
\newcommand{\INTP}  {\footnotesize{\texttt{I}}}
\newcommand{\vsI} {v^{\INTP}}
\newcommand{\vmfd}{\mathrm{v}}
\newcommand{\umfd}{\mathrm{u}}
\newcommand{\wmfd}{\mathrm{w}}

\newcommand{\Vhk}{V^{h,nc}_{k}}
\newcommand{\Vhkz}{V^{h,nc}_{k,0}}

\newcommand{\jump}[1]{\lbrack\!\lbrack\,#1\,\rbrack\!\rbrack} 
\newcommand{\vh}[1][h]{\mathfrak{v}_{#1}}
\newcommand{\uh}[1][h]{\mathfrak{u}_{#1}}

\def\lprop#1{\label{#1}\textbf{(#1)}}

\def\rprop#1{\hyperref[#1]{\textbf{(#1)}}}

\title{Discontinuous Skeletal Gradient Discretisation Methods on polytopal meshes\tnoteref{acknowledgements}}

\author[imag]{Daniele A. Di Pietro} %
\ead{daniele.di-pietro@umontpellier.fr} %
\author[monash]{J\'{e}r\^{o}me Droniou} %
\ead{jerome.droniou@monash.edu}
\author[lanl]{Gianmarco Manzini} %
\ead{gmanzini@lanl.gov}

\tnotetext[acknowledgements]{The first author was supported by the ANR grant HHOMM (ANR-15-CE40-0005).
The second author was supported by the ARC Discovery Projects funding scheme (project number DP170100605).
The third author was funded by the Laboratory Directed Research and Development program, under the auspices of the National Nuclear Security Administration of the U.S. Department of Energy by Los Alamos National Laboratory, operated by Los Alamos National Security LLC under contract DE-AC52-06NA25396.}

\address[imag]{Institut Montpelli\'{e}rain Alexander Grothendieck, CNRS, Univ. Montpellier (France)}
\address[monash]{School of Mathematical Sciences, Monash University, Melbourne (Australia)}
\address[lanl]{T-5 Applied Mathematics and Plasma Physics Group, Los Alamos National Laboratory, Los Alamos, New Mexico (USA)}

\begin{document}

\begin{abstract}
  In this work we develop arbitrary-order Discontinuous Skeletal Gradient Discretisations (DSGD) on general polytopal meshes.
  Discontinuous Skeletal refers to the fact that the globally coupled unknowns are broken polynomials on the mesh skeleton.
  The key ingredient is a high-order gradient reconstruction composed of two terms:
  \begin{inparaenum}[(i)]
  \item a consistent contribution obtained mimicking an integration by parts formula inside each element and
  \item a stabilising term for which sufficient design conditions are provided.
  \end{inparaenum}
  An example of stabilisation that satisfies the design conditions is proposed based on a local lifting of high-order residuals on a Raviart--Thomas--N\'ed\'elec subspace.
  We prove that the novel DSGDs satisfy coercivity, consistency, limit-conformity, and compactness requirements that ensure convergence for a variety of elliptic and parabolic problems.
  Links with Hybrid High-Order, non-conforming Mimetic Finite Difference and non-conforming Virtual Element methods are also studied.
  Numerical examples complete the exposition.
\end{abstract}


\maketitle


\section{Introduction}

The numerical resolution of (linear or non-linear) partial differential equations (PDEs) is nowadays ubiquitous in the engineering practice.
In this context, the design of convergent numerical schemes is a very active research topic.
The Gradient Discretisation Method (GDM) is a recently introduced framework which identifies key design properties to obtain convergent schemes for a variety of linear and non-linear elliptic and parabolic problems.
Several models of current use in fluid mechanics fall into the latter categories including, e.g., porous media flows governed by Darcy's law, phase change problems governed by the Stefan problem~\cite{Friedman:68}, as well as simplified models of the viscous terms in power-law fluids corresponding the Leray--Lions elliptic operators.
The latter also appear in the modelling of glacier motion \cite{Glowinski.Rappaz:03}, of incompressible turbulent flows in porous media \cite{Diaz.Thelin:94}, and in airfoil design \cite{Glowinski:84}.

A Gradient Discretisation (GD) is defined by a finite-dimensional space encoding the discrete unknowns, as well as two linear operators acting on the latter, and corresponding to reconstructions of scalar functions and of their gradient.
For a given PDE problem, convergent GDs are characterised by four properties, which can also serve as guidelines for the design of new schemes: \emph{coercivity}, which corresponds to a discrete Poincar\'e inequality; \emph{GD-consistency}, which expresses the ability of the scalar and gradient reconstructions to approximate functions in the space where the continuous problem is set; \emph{limit-conformity}, linking the two reconstructions through an approximate integration by parts formula; \emph{compactness}, corresponding to a discrete counterpart of the Rellich theorem.

In the recent monograph~\cite{gdm}, several classical discretisation methods have been interpreted in the GDM framework. These include: arbitrary-order conforming, nonconforming, and mixed Finite Elements (FE) on standard meshes; arbitrary-order discontinuous Galerkin (DG) schemes in their SIPG form \cite{Arnold:82} (see, in particular, \cite{Eymard.Guichard:17} on this point); various lowest-order Finite Volume methods on specific grids; lowest-order methods belonging to the Hybrid Mixed Mimetic family (see the unified presentation in~\cite{Droniou.Eymard.ea:10} of the methods originally proposed in~\cite{Brezzi.Lipnikov.ea:05,Droniou.Eymard:06,Eymard.Gallouet.ea:10}) as well as nodal Mimetic Finite Differences (MFD)~\cite{Brezzi-Lipnikov-Simoncini:2005} on arbitrary polyhedral meshes; see also~\cite{Beirao-da-Veiga.Lipnikov.ea:14}.

In this paper we present an important addition to the GDM framework: arbitrary-order Discontinuous Skeletal (DS) methods~\cite{Di-Pietro.Droniou.ea:15}, characterised by globally coupled unknowns that are broken polynomials on the mesh skeleton.
Specifically, the primary source of inspiration are the recently introduced Hybrid High-Order (HHO) methods for linear~\cite{Di-Pietro.Ern.ea:14,Di-Pietro.Ern:15} and non-linear~\cite{Di-Pietro.Droniou:16,Di-Pietro.Droniou:16*1} diffusion problems, and the high-order non-conforming MFD (ncMFD) method of~\cite{Lipnikov-Manzini:2014}; see also~\cite{Ayuso-de-Dios.Lipnikov.ea:16} for an interpretation in the Virtual Element framework and~\cite{Beirao-da-Veiga.Brezzi.ea:13} for an introduction to the latter.
We also cite here the Hybridizable Discontinuous Galerkin methods of~\cite{Cockburn.Gopalakrishnan.ea:09}, whose link with the former methods has been studied in~\cite{Cockburn.Di-Pietro.ea:15}; see also~\cite{Boffi.Di-Pietro:17} for a unified formulation.
Like DG methods, DS methods support arbitrary approximation orders on general polytopal meshes.
DS methods are, in addition, amenable to static condensation for linear(ised) problems, which can significantly reduce the number of unknowns in some configurations.
They also have better data locality, which can ease parallel implementations.
Moreover, lowest-order versions are often available that can be easily fitted into traditional Finite Volume simulators.
Finally, unlike DG methods, DS methods admit a Fortin operator in general meshes, a crucial property in the context of incompressible or quasi-incompressible problems in solid- and fluid-mechanics; see, e.g., \cite{Di-Pietro.Ern:15,Di-Pietro.Krell:17}.

Let a polynomial degree $k\ge 0$ be given.
The Discontinuous Skeletal Gradient Discretisations (DSGD) studied here hinge on face unknowns that ensure the global coupling and that correspond to broken polynomials of total degree up to $k$ on the mesh skeleton, as well as locally coupled element-based unknowns that correspond to broken polynomials of total degree up to $l\in\{k-1,k,k+1\}$ on the mesh itself.
The reconstruction of scalar functions is defined in a straightforward manner through the latter if $l\ge 0$, or by a suitable combination of face-based unknowns if $l=-1$.
The gradient reconstruction, on the other hand, requires a more careful design.
The seminal ideas to devise high-order gradient reconstructions on general meshes are already present, among others, in HHO methods (see, e.g.,~\cite[Eq. (13)]{Di-Pietro.Ern.ea:14} and~\cite[Eq. (4.3)]{Di-Pietro.Droniou:16}) as well as in ncMFD methods (see~\cite[Eq. (21)]{Lipnikov-Manzini:2014}).
These gradient reconstructions, however, are not suitable to define a convergent DSGD because they fail to satisfy the coercivity requirement.
In addition, when considering non-linear problems, the codomain of the gradient reconstruction has to be carefully selected in order for the GD-consistency requirement to be satisfied with optimal scaling in the meshsize for $k\ge 1$ (this point was already partially recognised in~\cite{Di-Pietro.Droniou:16}).
In the context of DG methods, a stable discrete gradient based on a variation of the method originally proposed in \cite{Castillo.Cockburn.ea:00} has been recently studied in~\cite{John.Neilan.ea:15}.

The main novelty of this work is the introduction of a gradient reconstruction that meets all the requirements to define a convergent GD, and which satisfies the limit-conformity property with an error that scales optimally in the meshsize.
This gradient reconstruction is composed of two terms: a consistent contribution closely inspired by~\cite[Eq. (4.3)]{Di-Pietro.Droniou:16} and a stabilisation term.
Two design conditions are identified for the stabilisation term:
\begin{inparaenum}[(i)]
\item local stability and boundedness with respect to a suitable boundary seminorm and
\item $L^2$-orthogonality to vector-valued polynomials of degree up to $k$.
\end{inparaenum}
When considering problems posed in a non-Hilbertian setting, an additional condition is added stipulating that the stabilisation is built on a piecewise polynomial space.
An example of stabilisation term that meets all of the above requirement is proposed based on a  Raviart--Thomas--N\'{e}d\'{e}lec space on a submesh.

The rest of the paper is organised as follows.
In Section~\ref{sec:GDM} we recall the basics of the GDM and give a few examples of linear and non-linear problems for which GDs are convergent under the coercivity, GD-consistency, limit-conformity, and compactness properties discussed above.
The construction of arbitrary-order DSGD is presented in Section~\ref{sec:DSGD}, the main results are stated in Section~\ref{sec:main.results}, and numerical examples are collected in Section~\ref{sec:numerical.examples}.
The links with HHO, ncMFD, and ncVEM schemes are studied in detail in Section~\ref{sec:links}.
\ref{sec:proofs} contains the proofs of the main results.
The material is organised so that multiple levels of reading are possible: readers mainly interested in the numerical recipe and results can primarily focus on Sections~\ref{sec:GDM}--\ref{sec:DSGD}; readers also interested in the relations with other polytopal methods can consult Section~\ref{sec:links}.


\section{The Gradient Discretisation Method}\label{sec:GDM}

We give here a brief presentation of the Gradient Discretisation Method (GDM) in the context
of homogeneous Dirichlet boundary conditions, and we refer to the monograph \cite{gdm} for more details
and other boundary conditions.

\subsection{Gradient Discretisations and Gradient Schemes}

Let $\Omega$ be a bounded polytopal domain in $\Real^d$, where $d\ge 1$ is the space dimension.
We consider elliptic or parabolic problems whose weak formulation is set in $W^{1,p}_0(\Omega)$, where $p\in(1,+\infty)$ denotes a Sobolev exponent which we assume fixed in what follows.

A \emph{Gradient Discretisation (GD)} is a triplet $\GD=(\XDz,\PiD,\grD)$ where:

\begin{compactenum}[(i)]
\item $\XDz$ is a finite dimensional vector space on $\Real$
  encoding the \emph{discrete unknowns}, and accounting for the
  homogeneous Dirichlet boundary condition;
  
  \medskip
\item $\PiD\,:\,\XDz\to L^p(\Omega)$ is a linear mapping that
  reconstructs scalar functions in $ L^p(\Omega)$ from the discrete unknowns in $\XDz$;

  \medskip
\item $\grD\,:\,\XDz\to L^p(\Omega)^d$ is a linear mapping that
  reconstructs the ``gradient'' of scalar functions in $ L^p(\Omega)^d$
  from the unknowns in $\XDz$. This reconstruction
  must be defined such that $\norm[L^p(\Omega)^d]{\grD\SCAL}$ is a
  norm on $\XDz$.
\end{compactenum}

In a nutshell, the GDM consists in selecting a GD and in replacing, in the weak
formulation of the PDE, the continuous space and operators by the discrete ones provided by the GD. The scheme
thus obtained is called a \emph{Gradient Scheme} (GS). To illustrate this procedure, consider the simple linear problem:
Find $u:\Omega\to\Real$ such that
\begin{equation}\label{lin.ell}
  \begin{alignedat}{2}
    -\DIV(\matr{\Lambda}\GRAD u)&=f &\qquad&\text{ in $\Omega$},\\
    u&=0 &\qquad&\text{ on $\partial\Omega$},
  \end{alignedat}
\end{equation}
with diffusion tensor $\matr{\Lambda}$ bounded and uniformly coercive,
and source term $f\in L^2(\Omega)$. The weak formulation of
\eqref{lin.ell} is
\begin{equation}\label{lin.ell.w}
\mbox{Find $u\in H^1_0(\Omega)$ such that, for all $v\in H^1_0(\Omega)$, }\int_\Omega \matr{\Lambda} \GRAD u
\SCAL\GRAD v=\int_\Omega fv.
\end{equation}
Given a gradient discretisation $\GD$, the gradient scheme for \eqref{lin.ell.w} is then
\begin{equation}\label{lin.ell.gs}
  \mbox{Find $u_{\GD}\in \XDz$ such that, for all $v_{\GD}\in \XDz$, }\int_\Omega \matr{\Lambda} \grD u_{\GD}
  \SCAL\grD v_{\GD}=\int_\Omega f\PiD v_{\GD}.
\end{equation}

The same procedure applies to non-linear problems.
Consider, e.g., the following generalisation of~\eqref{lin.ell} that corresponds to Leray--Lions operators:
Find $u:\Omega\to\Real$ such that
\begin{equation}\label{eq:leray.lions:strong}
  \begin{alignedat}{2}
    -\DIV\vec{\sigma}(\vec{x},u,\GRAD u) &= f &\qquad&\text{in $\Omega$},
    \\
    u &= 0 &\qquad&\text{on $\partial\Omega$},
  \end{alignedat}
\end{equation}
where the flux function $\vec{\sigma}:\Omega\times\Real\times\Real^d\to\Real^d$ satisfies the requirements detailed in~\cite[Eq. (2.85)]{gdm}.
A paradigmatic example of this class of problems is the $p$-Laplace equation which, for a fixed $p\in(1,+\infty)$, corresponds to the flux function
\begin{equation}\label{eq:flux.plap}
  \vec{\sigma}(\vec{x},u,\GRAD u) 
  = |\GRAD u|^{p-2}\GRAD u.
\end{equation}
Assuming $f\in L^{p'}(\Omega)$ with $p'\coloneq\frac{p}{p-1}$, Problem \eqref{eq:leray.lions:strong} admits the following weak formulation:
\begin{equation}\label{eq:plap:weak}
  \text{%
    Find $u\in W^{1,p}_0(\Omega)$ such that, for all $v\in W^{1,p}_0(\Omega)$,
    $\int_\Omega\vec{\sigma}(\vec{x},u,\GRAD u)\SCAL\GRAD v = \int_\Omega fv$.%
  }    
\end{equation}
Given a gradient discretisation $\GD$, the gradient scheme for \eqref{eq:plap:weak} is then
\begin{equation}\label{eq:plap:gs}
  \mbox{Find $u_{\GD}\in \XDz$ such that, for all $v_{\GD}\in \XDz$, }
  \int_\Omega \vec{\sigma}(\vec{x},\PiD u,\grD u_{\GD})\SCAL\grD v_{\GD}=\int_\Omega f\PiD v_{\GD}.
\end{equation}

\subsection{Convergent Gradient Schemes}\label{sec:convergent.GS}

The accuracy and convergence of GSs, for linear and non-linear problems, can be assessed by a few properties of the underlying GDs.
In many situations, and in all cases considered in this paper, GDs are obtained starting from a mesh of the domain.
We consider here \emph{polytopal meshes} corresponding to couples $\Mh\coloneq(\Th,\Fh)$, where $\Th$ is a finite collection of polytopal elements $T$
of maximum diameter equal to $h>0$, while $\Fh$ is a finite collection of hyperplanar faces $F$.
  It is assumed henceforth that the mesh $\Mh$ matches the weak geometrical requirements detailed in \cite[Definition 7.2]{gdm}; see also~\cite[Section 2]{Di-Pietro.Tittarelli:17}.
  Our focus is on the so-called $h$-convergence analysis, where we consider a sequence of refined meshes $(\Mh)_{h\in{\cal H}}$ whose sizes are collected in a countable set ${\cal H}\subset \Real_*^+$ having $0$ as its unique accumulation point.
  We further assume that the polytopal mesh sequences that we deal with are \emph{regular} in the sense of~\cite[Definition~3]{Di-Pietro.Tittarelli:17}, and we denote by $\varrho>0$ the corresponding regularity parameter.

The following properties allow us to single out sequences $(\GD[h])_{h\in{\cal H}}=(\XDz[h],\PiD[h],\grD[h])_{h\in{\cal H}}$ of GDs that lead to gradient schemes that converge, for both linear and non-linear problems:

\begin{enumerate}
\item[\lprop{GD1}] \emph{Coercivity.} Consider, for all $h\in{\cal H}$, the norm of the linear mapping $\PiD[h]$ defined by:
  $$
  \CD[h] \coloneq \max_{v\in\XDz[h]\setminus\{0\}}\frac{\norm[L^p(\Omega)]{\PiD[h] v}}{\norm[L^p(\Omega)^d]{\grD[h] v}}.
  $$
  Then, there exists a real number $C_{\rm P}>0$ such that $\CD[h]\leq C_{\rm P}$ for all $h\in{\cal H}$.
\item[\lprop{GD2}]  \emph{GD-Consistency.} For all $h\in{\cal H}$, let $\SD[h]\,:\, W^{1,p}_0(\Omega)\to[0,+\infty)$ be such that
  \begin{equation*}\label{eq:SD}
    \SD[h](\phi) \coloneq \min_{ v\in\XDz[h]}\left(
    \norm[L^p(\Omega)]{\PiD[h] v-\phi} + \norm[L^p(\Omega)^d]{\grD[h] v-\GRAD\phi}
    \right)
    \qquad\forall\phi\in W^{1,p}_0(\Omega).
  \end{equation*}
  Then, it holds that
  \begin{equation}\label{SD.to.0}
    \lim_{h\to 0}\SD[h](\phi) = 0 \qquad\forall\phi\in W^{1,p}_0(\Omega).
  \end{equation}
\item[\lprop{GD3}]  \emph{Limit-conformity.}  Let $p'\coloneq \frac{p}{p-1}$ denote the dual exponent of $p$, and set
$\vec{W}^{p'}(\oDIV;\Omega)\coloneq\{\vec{\psi}\in L^{p'}(\Omega)^d\,:\,\DIV\vec{\psi}\in L^{p'}(\Omega)\}$.
For all $h\in{\cal H}$, let $\WD[h]\,:\,\vec{W}^{p'}(\oDIV;\Omega)\to[0,+\infty)$ be such that, for all $\vec{\psi}\in \vec{W}^{p'}(\oDIV;\Omega)$,
  $$
  \WD[h](\vec{\psi}) \coloneq \sup_{ v\in\XDz[h]\setminus\{0\}}
  \frac{1}{\norm[L^p(\Omega)^d]{\grD[h] v}}
  \left|
  \int_{\Omega} \Big(\grD[h] v(\vec{x})\SCAL\vec{\psi}(\vec{x}) + \PiD[h] v(\vec{x})\DIV\vec{\psi}(\vec{x}) \Big) \ud{\vec{x}}
  \right|.
  $$
  Then, it holds that
  \begin{equation}\label{WD.to.0}
  \lim_{h\to 0}\WD[h](\vec{\psi}) = 0 \qquad\forall\vec{\psi}\in \vec{W}^{p'}(\oDIV;\Omega).
\end{equation}
\item[\lprop{GD4}]  \emph{Compactness.} For any $v_h\in \XDz[h]$ such that $(\norm[L^p(\Omega)^d]{\grD[h] v_h})_{h\in{\cal H}}$
  is bounded, the sequence $(\PiD[h]v_h)_{h\in{\cal H}}$ is relatively compact in $L^p(\Omega)$.
\end{enumerate}
A few comments are of order. Property \rprop{GD1} is linked to the stability of the method, and stipulates that the $L^p$-norm of the reconstruction of scalar functions is uniformly controlled by the $L^p$-norm of the reconstruction of their gradient. It readily implies the uniform Poincar\'e inequality $\norm[L^p(\Omega)]{\Pi_{\GD[h]} v_h}\leq C_{\rm P}\norm[L^p(\Omega)^d]{\grD[h] v_h}$ valid for all $h\in{\cal H}$ and all $v_h\in\XDz[h]$.

Properties \rprop{GD2} and \rprop{GD3} are linked to the consistency of the method.
More specifically, property \rprop{GD2} states that the reconstructions $\PiD[h]$ of scalar functions and $\grD[h]$ of their gradients are able to approximate functions that lie in the space $W^{1,p}_0(\Omega)$ where the continuous problem is set.
In the context of the FE convergence analysis, this property is an attribute of the underlying discrete space, and is usually called \emph{approximability}; see, e.g.,~\cite[Definition~2.14]{Ern.Guermond:04}.
Property \rprop{GD3}, on the other hand, establishes a link between $\PiD[h]$ and $\grD[h]$ in the form of a discrete integration by parts formula.
Its counterpart in the context of the FE convergence analysis for linear problems is \emph{asymptotic consistency}; see, e.g.,~\cite[Definition~2.15]{Ern.Guermond:04}. Notice, however, that the formulation in \rprop{GD3} is in a sense more general, as it is not linked to a specific underlying problem and is in particular readily applicable to non-linear problems (whereas \cite[Definition~2.15]{Ern.Guermond:04} is restricted to linear problems).

Finally, property \rprop{GD4} is a discrete Rellich compactness theorem, and can be regarded as the key ingredient to obtain strong convergence results by compactness techniques.

\begin{remark}[Limit-conformity or compactness implies coercivity]
Either one of \rprop{GD3} or \rprop{GD4} imply \rprop{GD1}, see~\cite[Lemmas 2.7 and 2.11]{gdm}. The coercivity is however kept as a separate property to highlight its importance.
\end{remark}

The above properties are sufficient to carry out a convergence analysis, either by error estimates (when the
model is amenable to these) or by compactness, for a variety of linear and non-linear elliptic or parabolic models.
An example of such convergence results for gradient discretisations of the Leray--Lions problem~\eqref{eq:plap:weak} is provided by Theorems~\ref{thm:convergence.example} and \ref{thm:error.est.example} below; see~\cite{gdm} for a comprehensive collection of convergence results for various linear and non-linear elliptic and parabolic problems.

\begin{theorem}[Convergence]\label{thm:convergence.example}
We assume that $\vec{\sigma}$ satisfies the classical properties of Leray--Lions operators
(see~\cite[Eqs. (2.85) and (2.87)]{gdm}).
  Let $(\GD[h])_{h\in{\cal H}}$ denote a sequence of GDs satisfying \rprop{GD1}--\rprop{GD4}.
  Then, for all $h\in{\cal H}$, there exists at least one $u_{\GD[h]}\in\XDz[h]$ solution to~\eqref{eq:plap:gs} and, along a subsequence as $h\to 0$,
  \begin{inparaenum}[(i)]
  \item $\PiD[h]u_{\GD[h]}$ converges strongly in $L^p(\Omega)$ to a solution $u$ of~\eqref{eq:plap:weak};
  \item $\grD[h]u_{\GD[h]}$ converges strongly in $L^p(\Omega)^d$ to $\GRAD u$.
  \end{inparaenum}
\end{theorem}
\begin{proof}
  This is a special case of~\cite[Theorem~2.45]{gdm}.
\end{proof}

\begin{theorem}[Error estimates]\label{thm:error.est.example}
Let $\GD$ be a gradient discretisation, and let $\vec{\sigma}$ be the Leray--Lions operator
corresponding to the $p$-Laplace equation (see \eqref{eq:flux.plap}). 
Then there exists a unique $u_{\GD}$ solution to \eqref{eq:plap:gs} and, if $u$ is the solution to \eqref{eq:plap:weak},
then there exists $C$ depending only on $p$, $f$ and an upper bound of $\CD$ such that
\begin{itemize}
\item If $1<p\le 2$,
\[
\norm[L^p(\Omega)]{u-\PiD u_{\GD}}+\norm[L^p(\Omega)]{\GRAD u-\grD u_{\GD}}
\le C \left[\SD(u)+\SD(u)^{p-1}+\WD(\vec{\sigma}(\GRAD u))\right].
\]
\item If $2\le p$,
\[
\norm[L^p(\Omega)]{u-\PiD u_{\GD}}+\norm[L^p(\Omega)]{\GRAD u-\grD u_{\GD}}
\le C \left[\SD(u)+\SD(u)^{\frac{1}{p-1}}+\WD(\vec{\sigma}(\GRAD u))^{\frac{1}{p-1}}\right].
\]
\end{itemize}
\end{theorem}

\begin{proof} These error estimates are simplified forms of the ones in
\cite[Theorem~2.39]{gdm}.
\end{proof}

\section{Discontinuous Skeletal Gradient Discretisations}\label{sec:DSGD}

In this section, we construct a family of Discontinuous Skeletal Gradient Discretisations (DSGD).
The notation is closely inspired by HHO methods; see, e.g.,~\cite{Di-Pietro.Tittarelli:17}.

\subsection{Local polynomial spaces and projectors}

Local polynomial spaces on mesh elements or faces and projectors thereon play a crucial role in the design and analysis of DSGD methods.

For any $X\subset\overline{\Omega}$, we denote by $(\cdot,\cdot)_X$ the standard $L^2(X)$- or $L^2(X)^d$-products.
This notation is used in place of integrals when dealing with quantities that are inherently $L^2$-based.
Let now $X$ be a mesh element or face.
For an integer $\ell\ge 0$, $\Poly{\ell}(X)$ denotes the space spanned by the restriction to $X$ of scalar-valued, $d$-variate (if $X$ is a mesh element) or $(d-1)$-variate (if $X$ is a face) polynomials of total degree $\ell$ or less, and conventionally set $\Poly{-1}(X)\coloneq\{0\}$.

Let again $X$ denote a mesh element or face.
The $L^2$-orthogonal projector $\lproj[X]{\ell}:L^1(X)\to\Poly{\ell}(X)$ is defined as follows:
For all $v\in L^1(X)$, $\lproj[X]{\ell}$ is the unique polynomial in $\Poly{\ell}(X)$ such that
\begin{equation}\label{eq:lproj}
  (\lproj[X]{\ell}v-v,w)_X=0\qquad\forall w\in\Poly{\ell}(X).
\end{equation}
In the vector case, the $L^2$-projector is defined component-wise and denoted by $\vlproj[X]{\ell}$.

For any mesh element $T\in\Th$, we also define the elliptic projector $\eproj{\ell}:W^{1,1}(T)\to\Poly{\ell}(T)$ as follows:
For all $v\in W^{1,1}(T)$, $\eproj{\ell}v$ is the unique polynomial in $\Poly{\ell}(T)$ that satisfies
\begin{equation*}\label{eq:eproj}
  \text{$(\GRAD(\eproj{\ell}v-v),\GRAD w)_T=0$ for all $w\in\Poly{\ell}(T)$ and $(\eproj{\ell}v-v,1)_T=0$.}
\end{equation*}

On regular polytopal mesh sequences, both $\lproj[T]{\ell}$ and $\eproj[T]{\ell}$ have optimal approximation properties in $\Poly{\ell}(T)$ (see Theorem~1.1, Theorem~1.2, and Lemma~3.1 in~\cite{Di-Pietro.Droniou:16*1}):
For any $\alpha\in\{0,1\}$ and $s\in\{\alpha,\ldots,\ell+1\}$, there exists a real number $C>0$ independent of $h$, but possibly depending only on $d$, $p$, $\varrho$, $\ell$, $\alpha$, and $s$, such that, for all $T\in\Th$, and all $v\in W^{s,p}(T)$,
\begin{subequations}\label{eq:approx.approx.trace}
  \begin{equation}\label{eq:approx}
    \seminorm[W^{r,p}(T)]{v - \pi_T^{\alpha,\ell} v }
    \le 
    C h_T^{s-r} 
    \seminorm[W^{s,p}(T)]{v}
    \qquad \forall r \in \{0,\ldots,s\},
  \end{equation}
  and, if $s\ge 1$,
  \begin{equation}\label{eq:approx.trace}
    h_T^{\frac1p}\seminorm[{W^{r,p}(\Fh[T])}]{v - \pi_T^{\alpha,\ell} v}
    \le 
    C h_T^{s-r} 
    \seminorm[W^{s,p}(T)]{v}
    \qquad \forall r \in \{0,\ldots,s-1\},
  \end{equation}
  where $W^{r,p}(\Fh[T])\coloneq\left\{ v\in L^p(\partial T)\st\text{$v_{|F}\in W^{r,p}(F)$ for all $F\in\Fh[T]$}\right\}$ and $h_T$ denotes the diameter of the element $T$.
\end{subequations}

\subsection{Computing gradient projections from projections of scalar functions}

We continue our discussion with a crucial remark concerning the computation of the $L^2$-orthogonal projection of the gradient from $L^2$-orthogonal projections of a scalar function and its traces.
This remark will inspire the choice of the discrete unknowns as well as the definition of the gradient reconstruction in DSGD methods.
In what follows, we work on a fixed mesh element $T\in\Th$, denote by $\Fh[T]$ the set of mesh faces that lie on the boundary of $T$ and, for all $F\in\Fh[T]$, by $\normal_{TF}$ the normal vector to $F$ pointing out of $T$.

Consider a function $v\in W^{1,1}(T)$.
We note the following integration by parts formula, valid for all $\vec{\phi}\in C^\infty(\overline{T})^d$:
\begin{equation}\label{eq:ipp}
  (\GRAD v,\vec{\phi})_T = -(v, \DIV\vec{\phi})_T + \sum_{F\in\Fh[T]}(v, \vec{\phi}\SCAL\normal_{TF})_F.
\end{equation}
Let now an integer $k\ge 0$ be fixed. Specialising~\eqref{eq:ipp} to $\vec{\phi}\in\Poly{k}(T)^d$, we obtain
\begin{equation}\label{eq:ipp.poly}
  (\vlproj[T]{k}\GRAD v,\vec{\phi})_T
  = -(\lproj[T]{k-1}v, \DIV\vec{\phi})_T
  + \sum_{F\in\Fh[T]}(\lproj[F]{k} v, \vec{\phi}\SCAL\normal_{TF})_F,
\end{equation}
where we have used~\eqref{eq:lproj} to insert $\vlproj[T]{k}$ into the left-hand side, and $\lproj[T]{k-1}$ and $\lproj[F]{k}$ into the right-hand side after  observing that $\DIV\vec{\phi}\in\Poly{k-1}(T)$ and, since we are considering planar faces, $\vec{\phi}_{|F}\SCAL\normal_{TF}\in\Poly{k}(F)$ for all $F\in\Fh[T]$.
The relation~\eqref{eq:ipp.poly} shows that computing the $L^2$-orthogonal projection of $\GRAD v$ on $\Poly{k}(T)^d$ \emph{does not require a full knowledge of the function $v$}.
All that is required is
\begin{enumerate}[(i)]
\item $\lproj[T]{k-1}v$, the $L^2$-projection of $v$ on $\Poly{k-1}(T)$.
  Other possible choices are $\lproj[T]{k}v$ or $\lproj[T]{k+1}v$ (in fact, any polynomial degree larger than or equal to $k-1$ will do);
\item for all $F\in\Fh[T]$, $\lproj[F]{k}v$, the $L^2$-projection on $\Poly{k}(F)$ of the trace of $v$ on $F$.
\end{enumerate}

\subsection{Space of discrete unknowns and reconstruction of scalar functions}

Inspired by the previous remark, for two given integers $k\ge 0$ and $l\in\{k-1,k,k+1\}$ we consider the following set of discrete unknowns:
$$
\Uh \coloneq \left(
\bigtimes_{T\in\Th}\Poly{l}(T)
\right)\times\left(
\bigtimes_{F\in\Fh}\Poly{k}(F)
\right).
$$
The choice $l=k-1$ can be traced back to the ncMFD of~\cite{Lipnikov-Manzini:2014}, the choice $l=k$ to the Hybrid High-Order method of~\cite{Di-Pietro.Ern.ea:14}, and the choice $l=k+1$ to the Hybridizable Discontinuous Galerkin method of~\cite[Remark~1.2.4]{Lehrenfeld:10}.
Notice that, for $k=0$ and $l=k-1$, element-based unknowns are not present.

For a generic element of $\Uh$, we use the standard HHO underlined notation 
$$\uv=\left((v_T)_{T\in\Th},(v_F)_{F\in\Fh}\right),$$
and we define the interpolator $\Ih:W^{1,1}(\Omega)\to\Uh$ such that, for all $v\in W^{1,1}(\Omega)$,
$$
\Ih v\coloneq \left(
(\lproj[T]{l}v)_{T\in\Th}, (\lproj[F]{k} v_{|F})_{F\in\Fh}
\right).
$$

To account for Dirichlet boundary conditions strongly, we introduce the subspace
$$
\Uhz\coloneq\left\{
\uv\in\Uh\st v_F\equiv 0\mbox{ for all } F\in\Fhb
\right\},
$$
where $\Fhb$ is the set collecting the mesh faces that lie on the boundary of $\Omega$.

The restrictions of $\Uh$, $\Ih$ and $\uv\in\Uh$ to a generic mesh element $T\in\Th$ are denoted by $\UT$, $\IT$, and $\uv[T]$, respectively. That is,
$$
\UT\coloneq\{\uv[T]=(v_T,(v_F)_{F\in\Fh[T]})\,:\,v_T\in\Poly{l}(T)\,,\;v_F\in \Poly{k}(F)\quad\forall F\in\Fh[T]\}
$$
and, for all $v\in W^{1,1}(T)$,
$$
\IT v\coloneq(\lproj[T]{l}v,(\lproj[F]{k} v_{|F})_{F\in\Fh[T]}).
$$
Moreover, we adopt the convention that, for all $T\in\Th$,
\begin{equation}\label{eq:vT.l<0}
  \text{$v_T\coloneq\sum_{F\in\Fh[T]}\omega_{TF} v_F$ if $l<0$},
\end{equation}
where, following~\cite[Appendix A]{Lipnikov-Manzini:2014}, the weights $\{\omega_{TF}\}_{F\in\Fh[T]}$ are defined in such a way that $\sum_{F\in\Fh[T]}\omega_{TF}(q,1)_F = (q,1)_T$ for all $q\in\Poly{1}(T)$ (this condition is required in the above reference to obtain $L^2$-superconvergence, not treated in this work).
For all $\uv\in\Uh$, we also define the broken polynomial field $v_h$ such that
\begin{equation}\label{eq:vh}
  {v_h}_{|T}\coloneq v_T\qquad\forall T\in\Th.
\end{equation}

The space of discrete unknowns and the reconstruction of the scalar variable for a DSDG are given by, respectively,
\begin{equation}\label{eq:DSGDM:XDz.PiD}  
  \XDz[h] \coloneq \Uhz\quad\text{ and }\quad \PiD[h]\uv \coloneq v_h \text{ for all }\uv\in\Uhz.
\end{equation}

\subsection{Reconstruction of the gradient}

To complete the definition of a DSGD, it remains to identify a reconstruction of the gradient, which makes the object of this section.

\subsubsection{A consistent and limit-conforming high-order gradient}

Let a mesh element $T\in\Th$ be fixed.
Taking inspiration from the integration by parts formula~\eqref{eq:ipp.poly}, we define the gradient reconstruction $\GT:\UT\to \Poly{k}(T)^d$ such that, for any $\uv[T]=(v_T,(v_F)_{F\in\Fh[T]})\in\UT$, $\GT\uv[T]$ satisfies, for all $\bphi\in \Poly{k}(T)^d$,
\begin{equation}
  \label{eq:GT}
  (\GT \uv[T],\bphi)_T
  = -(v_T,\DIV\bphi)_T
  + \sum_{F\in\Fh[T]} (v_F, \bphi\SCAL\normal_{TF})_F.
\end{equation}
By construction, it holds for all $v\in W^{1,1}(T)$,
\begin{equation}\label{GT:commutativity}
\GT \IT v = \vlproj[T]{k}(\GRAD v).
\end{equation}
Recalling the estimates \eqref{eq:approx.approx.trace} on $\vlproj[T]{k}$, this implies that $\GT\IT v$ optimally approximates $\GRAD v$ in $\Poly{k}(T)^d$ when $v$ is smooth enough.

A reconstruction of the gradient that meets the consistency requirement expressed by \rprop{GD2} can be obtained at this point letting $\grD[h]$ be such that, for all $\uv\in\Uhz$,
\begin{equation}\label{eq:grD.nonstable}
  \grD[h]\uv=\Gh\uv[h],
\end{equation}
where $\Gh:\Uh\to\Poly{k}(\Th)^d$ is the global consistent gradient reconstruction operator obtained patching the local reconstructions:
For all $\uv\in\Uh$,
\begin{equation}\label{eq:Gh}
  (\Gh\uv)_{|T}\coloneq\GT\uv[T]\qquad\forall T\in\Th.
\end{equation}
However, for general element shapes, the $L^p$-norm of this gradient reconstruction is not a norm on the space $\XDz[h]=\Uhz$, hence the coercivity requirement expressed by \rprop{GD1} cannot be met.
This initial choice of reconstructed gradient therefore has to be stabilised by accounting
for jumps between element and face unknowns.
These jumps can be controlled in turn via a discrete counterpart of the $W^{1,p}$-seminorm,
which gives us an emulated Sobolev structure on $\Uh$.

\begin{remark}[Non-conforming $\Poly{1}$ finite elements]

  If $T$ is a $d$-simplex (i.e., a triangle if $d=2$, a tetrahedron if $d=3$, etc.) and we take $k=0$ and $l=-1$, the gradient reconstruction $\GT\uv[T]$ defined by~\eqref{eq:GT} coincides with the gradient of the non-conforming $\Poly{1}$ function $\varphi$ such that $|F|^{-1}\int_F\varphi=v_F$ for all $F\in\Fh[T]$.
  In this case, the $L^p$-norm of the global gradient given by~\eqref{eq:grD.nonstable} defines a norm on the space of discrete unknowns, and stabilisation is not needed.
  To recover the usual non-conforming $\Poly{1}$ scheme (often called
the Crouzeix--Raviart scheme, although historically this name refers to the usage
of non-conforming $\Poly{1}$--$\Poly{0}$ discretisations for the velocity--pressure
unknowns in Stokes and Navier--Stokes equations~\cite{Crouzeix.Raviart:77}),~\eqref{eq:DSGDM:XDz.PiD} has to be modified setting ${\PiD[h]\uv}_{|T}\coloneq\rT[1]\uv[T]$ for all $T\in\Th$ and all $\uv[T]\in\UT[0,-1]$, where $\rT[1]$ is the high-order reconstruction of scalar function defined in the following section.
\end{remark}

\subsubsection{High-order reconstruction of scalar functions and discrete $W^{1,p}$-seminorm}

Let $\uv[T]\in\UT$.
Recalling the convention~\eqref{eq:vT.l<0}, $v_T$ defines a reconstruction of scalar functions inside $T$ of degree $\max(0,l)$.
However, taking again inspiration from the integration by parts formula \eqref{eq:ipp.poly}, this time specialised to $\vec{\phi}=\GRAD w$ with $w\in\Poly{k+1}(T)$, one can define a higher-order reconstruction $\rT:\UT\to\Poly{k+1}(T)$ such that, for all $\uv[T]\in\UT$, $\rT\uv[T]$ satisfies, for all $w\in\Poly{k+1}(T)$,
\begin{subequations}\label{eq:rT}
  \begin{align}
    \label{eq:rT.pde}
    (\GRAD\rT\uv[T],\GRAD w)_T &= -(v_T,\LAPL w)_T + \sum_{F\in\Fh[T]} (v_F,\GRAD w\SCAL\normal_{TF})_F.
  \end{align}
  Equation~\eqref{eq:rT.pde} defines $\rT\uv[T]$ up to an additive constant, which we fix by imposing
  \begin{equation}\label{eq:rT.average}
  (\rT\uv[T]-v_T,1)_T=0.
  \end{equation}
\end{subequations}
\begin{remark}[Optimal approximation properties of $\rT\circ\IT$]\label{rem:approx:rT.IT}
When $l\ge 0$, following the reasoning of~\cite[Lemma~3]{Di-Pietro.Ern.ea:14}, it can be proved that $\rT\circ\IT=\eproj{k+1}$, and optimal approximation properties in $\Poly{k+1}(T)$ follow from~\eqref{eq:approx.approx.trace} with $\alpha=1$ and $\ell=k+1$.
The case $l<0$, on the other hand, can only occur when $k=0$.
Owing to the specific choice for the reconstruction of a (constant) element value in~\eqref{eq:vT.l<0}, optimal approximation properties analogous to~\eqref{eq:approx.approx.trace} with $\alpha=1$ and $\ell=1$ can be proved also in this case.
\end{remark}

To define the discrete Sobolev seminorm on $\Uh$, for all $T\in\Th$ we introduce the difference operators $\dT:\UT\to\Poly{l}(T)$ and, for all $F\in\Fh[T]$, $\dTF:\UT\to\Poly{k}(F)$ such that, for all $\uv[T]\in\UT$,
\begin{equation}\label{eq:dT.dTF}
  \dT\uv[T]\coloneq\lproj[T]{l}(\rT\uv[T]-v_T),\qquad
  \dTF\uv[T]\coloneq\lproj[F]{k}(\rT\uv[T]-v_F)\quad\forall F\in\Fh[T].
\end{equation}
The role of these difference operators in the context of HHO methods has been highlighted in~\cite[Section~3.1.4]{Di-Pietro.Tittarelli:17}.
We also note here the following relation:
\begin{equation}\label{eq:dT.dTF'}
  (\dT\uv[T],(\dTF\uv[T])_{F\in\Fh[T]})=\IT\rT\uv[T]-\uv[T],
\end{equation}
which will be exploited in Section~\ref{sec:generality.ST} and Lemma \ref{lem:norm.interp} below.

The discrete $W^{1,p}$-seminorm is defined setting
\begin{equation*}\label{eq:norm1p}
  \norm[1,p,h]{\uv}^p\coloneq\sum_{T\in\Th}\norm[1,p,T]{\uv[T]}^p,
\end{equation*}
where, for all $T\in\Th$, the local seminorm is such that, denoting by $h_F$ the diameter of the face $F$,
\begin{equation}\label{eq:norm1p.T}
  \norm[1,p,T]{\uv[T]}^p \coloneq \norm[L^p(T)^d]{\GT\uv[T]}^p + \seminorm[p,\partial T]{\uv[T]}^p,\qquad
  \seminorm[p,\partial T]{\uv[T]}^p \coloneq \sum_{F\in\Fh[T]}h_F^{1-p}\norm[L^p(F)]{(\dTF-\dT)\uv[T]}^p.
\end{equation}
As a result of Proposition~\ref{prop:norm.equiv} below, $\norm[1,p,h]{{\cdot}}$ is a norm on the subspace $\Uhz$.

\subsubsection{A stabilised reconstruction of the gradient}

We can now describe the general form of the gradient $\grD[h]:\Uh\to L^2(\Omega)^d$, built inside each mesh element from
the consistent and limit-conforming part $\GT$ and a stabilising contribution:
\begin{equation}\label{eq:DSGDM:grD}
(\grD[h]\uv)_{|T}=\grT\uv[T]\coloneq\GT\uv[T] + \ST\uv[T]\qquad\forall\uv\in\Uh\,,\;\forall T\in\Th,
\end{equation}
where $\ST:\UT\to L^2(T)^d$ satisfies the following design conditions:
\begin{enumerate}
\item[\lprop{S1}] \emph{$L^2$-stability and boundedness.}
  For all $T\in\Th$ and all $\uv[T]\in\UT$, it holds that
  \begin{equation}\label{eq:ST:stability}
    \norm[L^2(T)^d]{\ST\uv[T]}\simeq\seminorm[2,\partial T]{\uv[T]},
  \end{equation}
  where $a\simeq b$ means $C a\le b\le C^{-1} a$ with real number $C>0$ independent of $h$ and of $T$, but possibly depending on $d$ and on discretisation parameters including $\varrho$, $k$, and $l$.
\item[\lprop{S2}] \emph{Orthogonality.} For all $\uv[T]\in\UT$ and all $\bphi\in\Poly{k}(T)^d$, it holds
  \begin{equation}\label{eq:ST:orthogonality}
    (\ST\uv[T],\bphi)_T= 0.
  \end{equation}
\item[\lprop{S3}] \emph{Image.} If $p\neq 2$, there exists $k_{\rm S}\in\Natural$ independent of $h$ and of $T$ such that the image of $\ST$ is contained in $\Poly{k_{\rm S}}(\mathcal{P}_T)^d$, the space of vector-valued broken polynomials of total degree up to $k_{\rm S}$ on a regular polytopal partition $\mathcal{P}_T$ of $T$.
  Here, regular means that, for all $P\in\mathcal{P}_T$, denoting by $r_P$ and $h_P$ the inradius and diameter of $P$, respectively, it holds that
  \begin{equation}\label{eq:S3.cond}
    \varrho h_P\le r_P,\qquad\varrho h_T\le h_P.
  \end{equation}
\end{enumerate}
\begin{remark}[$L^2$-based stabilising contribution]
  Property \rprop{S2}, which is crucial to ensure the stabilising properties of $\ST$, requires to work with an inner product space.
  In our case, a natural choice is $L^2(T)^d$.
  The role of orthogonality properties analogous to \rprop{S2} has been previously recognised
  in the context of specific stabilised method, see for example~\cite[Proposition 7]{Bonelle.Di-Pietro.ea:15}
  for the lowest-order Compatible Discretisation Operator methods,
  \cite[Theorems 13.7 and 14.5]{gdm} for the Hybrid Mimetic Mixed methods and the nodal Mimetic Mixed Methods,
  and \cite[Section 4.2]{DL14} for numerical methods for elasticity models.
\end{remark}
\begin{remark}[$L^p$-stability of $\ST$]
    Property \rprop{S3} is required to extend the stability properties expressed by \rprop{S1} to $L^p$; see the proof of point (i) in Proposition~\ref{prop:ST} for further details.
\end{remark}

The above construction of a DSGD is summarised in the following
\begin{definition}[Discontinuous Skeletal Gradient Discretisation]\label{def:DSGD}
Given a polytopal mesh $\Mh$, a Discontinuous Skeletal Gradient Discretisation (DSGD) is given
by $\GD[h]=(\XDz[h],\PiD[h],\grD[h])$ where $\XDz[h]$ and $\PiD[h]$ are defined by
\eqref{eq:DSGDM:XDz.PiD}, and $\grD[h]$ is given by \eqref{eq:DSGDM:grD} with a family of stabilisations $\{\ST\st T\in\Mh\}$
satisfying properties \rprop{S1}--\rprop{S3}.
\end{definition}

\subsection{Main results}\label{sec:main.results}

The construction detailed above yields a GD that meets properties \rprop{GD1}--\rprop{GD4} identified in Section~\ref{sec:GDM}, as summarised in the following

\begin{theorem}[Properties of DSGD]\label{thm:DSGDM}
  If $(\Mh)_{h\in\mathcal{H}}$ is a regular sequence of polytopal meshes, then the sequence of the corresponding DSGDs $(\GD[h])_{h\in\mathcal H}$
  given by Definition \ref{def:DSGD} satisfies properties \rprop{GD1}--\rprop{GD4}.
\end{theorem}

\begin{proof}
  See \ref{sec:proof:DSGDM}.
\end{proof}

Since we are dealing with arbitrary-order methods, given the error estimates in Theorem \ref{thm:error.est.example}, a relevant point consists in estimating the convergence rates of the quantities $\SD[h](\phi)$ (see \rprop{GD2}) and $\WD(\vec{\psi})$ (see \rprop{GD3}) when their arguments exhibit further regularity.
This makes the object of the following

\begin{proposition}[Estimates on $\SD$ and $\WD$]\label{prop:est.SD.WD}
  Let $\Mh$ be a polytopal mesh and $\GD[h]$ be a DSGD as in Definition \ref{def:DSGD}.
  Then, denoting by $a\lesssim b$ the inequality $a\le Cb$ with real number $C>0$ not depending on $h$, but possibly depending on $d$, $p$, $\varrho$, $k$, $l$, and $k_{\rm S}$, it holds with $l^+\coloneq\max(l,0)$,
    \begin{subequations}\label{est:terms.SD}
      \begin{align}\label{est:terms.SD:Lp}
        &\forall \phi\in W^{1,p}_0(\Omega)\cap W^{l^++1,p}(\Th),
        &\norm[L^p(\Omega)]{\PiD[h]\Ih\phi-\phi} &\lesssim        
        h^{l^++1}\seminorm[W^{l^++1,p}(\Th)]{\phi},
        \\
        \label{est:terms.SD:W1p}
        &\forall \phi\in W^{1,p}_0(\Omega)\cap W^{k+2,p}(\Th),
        &\norm[L^p(\Omega)^d]{\grD[h]\Ih\phi-\GRAD\phi} &\lesssim h^{k+1}\seminorm[W^{k+2,p}(\Th)]{\phi}.
      \end{align}
    \end{subequations}
    As a consequence,
  \begin{equation}\label{est:SD}
    \forall \phi\in W^{1,p}_0(\Omega)\cap W^{\min(k,l^+)+2,p}(\Th),\qquad
    \SD[h](\phi)\lesssim h^{\min(k,l^+)+1}\norm[W^{\min(k,l^+)+2,p}(\Th)]{\phi}.
  \end{equation}
  Moreover,
  \begin{equation}\label{est:WD}
    \forall \vec{\psi}\in \vec{W}^{p'}(\oDIV;\Omega)\cap W^{k+1,p'}(\Th)^d,\qquad
    \WD[h](\vec{\psi}) \lesssim h^{k+1}\norm[W^{k+1,p'}(\Th)^d]{\vec{\psi}}.
  \end{equation}
  Here, for an integer $s\ge 0$ and a real number $q\in[1,+\infty)$, $W^{s,q}(\Th)\coloneq\{v\in L^q(\Omega)\st v_{|T}\in W^{s,q}(T)\ \forall T\in\Th\}$
  is the broken space on $\Th$ constructed on $W^{s,q}$ and endowed with the norm
  \[
  \norm[W^{s,q}(\Th)]{v}\coloneq\left(\sum_{T\in\Th} \norm[W^{s,q}(T)]{v_{|T}}^q\right)^{\frac1q}.
  \]
\end{proposition}

\begin{proof}
  See \ref{sec:est.SD.WD}.
\end{proof}

\begin{remark}[Order of {$\SD[h]$}]
  In the case $l\ge k$, \eqref{est:SD} yields the optimal order $\mathcal O(h^{k+1})$ for interpolations
  of smooth enough functions. If $l=k-1$, one order is lost and \eqref{est:SD} gives an $\mathcal O(h^k)$
  estimate (but, as shown by \eqref{est:terms.SD}, this loss is only perceptible on the approximations
  of the functions, not of their gradients).
\end{remark}

\subsection{Local stabilising contribution based on a Raviart--Thomas--N\'ed\'elec subspace}\label{sec:stab.RT}

We construct in this section a stabilising contribution that fulfils the requirements expressed by \rprop{S1}--\rprop{S3}.

\subsubsection{An inspiring remark}\label{sec:insp}

Let, for the sake of brevity, $$\dGT\coloneq\GRAD\rT-\GT.$$
We start by observing that it holds, for all $\uv[T]\in\UT$ and all $\bphi\in\Poly{k}(T)^d$,
\begin{equation*} 
  \begin{aligned}
    (\dGT\uv[T],\bphi)_T
    &= (\GRAD\rT\uv[T],\bphi)_T - (\GT\uv[T],\bphi)_T
    \\
    &= (v_T-\rT\uv[T],\DIV\bphi)_T + \sum_{F\in\Fh[T]}(\rT\uv[T]-v_F,\bphi\SCAL\normal_{TF})_F
    \\
    &= (\lproj[T]{l}(v_T-\rT\uv[T]),\DIV\bphi)_T + \sum_{F\in\Fh[T]}(\lproj[F]{k}(\rT\uv[T]-v_F),\bphi\SCAL\normal_{TF})_F
    \\
    &= -(\dT\uv[T],\DIV\bphi)_T + \sum_{F\in\Fh[T]} (\dTF\uv[T],\bphi\SCAL\normal_{TF})_F
    \\
    &= (\GRAD\dT\uv[T],\bphi)_T + \sum_{F\in\Fh[T]} ((\dTF-\dT)\uv[T],\bphi\SCAL\normal_{TF})_F,
  \end{aligned}
\end{equation*}
where we have used the definition of $\dGT$ in the first line,
an integration by parts together with the definition~\eqref{eq:GT} of $\GT$ in the second line,
~\eqref{eq:lproj} together with the fact that $\DIV\bphi\in\Poly{k-1}(T)\subset\Poly{l}(T)$ since $l\ge k-1$ and $\bphi_{|F}\SCAL\normal_{TF}\in\Poly{k}(F)$ for all $F\in\Fh[T]$ to introduce the projectors in the third line,
the definition~\eqref{eq:dT.dTF} of $\dT$ and $\dTF$ in the fourth line,
and an integration by parts to conclude.
Rearranging the terms, we arrive at
\begin{equation}\label{eq:inspiration.residual}
  ((\dGT-\GRAD\dT)\uv[T],\bphi)_T
  = \sum_{F\in\Fh[T]}((\dTF-\dT)\uv[T],\bphi\SCAL\normal_{TF})_F.
\end{equation}
A few remarks are of order to illustrate the consequences of the above relation.
\begin{remark}[Control of the element-based difference through face-based differences]\label{rem:control.volumic}
  A first notable consequence is that the element-based difference \mbox{$(\dGT-\GRAD\dT)\uv[T]$} can be controlled in terms of the face-based differences $\{(\dTF-\dT)\uv[T]\st F\in\Fh[T]\}$:
  For all $T\in\Th$ and all $\uv[T]\in\UT$, it holds
  \begin{equation}\label{eq:control.dGT-GdT}
    \norm[L^2(T)^d]{(\dGT-\GRAD\dT)\uv[T]}\lesssim\seminorm[2,\partial T]{\uv[T]},
  \end{equation}
  where $a\lesssim b$ means $a\le Cb$ with real number $C>0$ independent of $h$ and of $T$, but possibly depending on $d$, $\varrho$, $k$, and $l$.
  To prove~\eqref{eq:control.dGT-GdT}, it suffices to observe that
  $$
  \begin{aligned}
    \norm[L^2(T)^d]{(\dGT-\GRAD\dT)\uv[T]}
    &= \sup_{\bphi\in\Poly{k}(T)^d,\norm[L^2(T)^d]{\bphi}=1} ((\dGT-\GRAD\dT)\uv[T],\bphi)_T
    \\
    &= \sup_{\bphi\in\Poly{k}(T)^d,\norm[L^2(T)^d]{\bphi}=1} \sum_{F\in\Fh[T]}((\dTF-\dT)\uv[T],\bphi\SCAL\normal_{TF})_F
    \\
    &\le \sup_{\bphi\in\Poly{k}(T)^d,\norm[L^2(T)^d]{\bphi}=1} \seminorm[2,\partial T]{\uv[T]} h_T^{\frac{1}{2}}\norm[L^2(\partial T)]{\bphi\SCAL\normal_T}
    \lesssim\seminorm[2,\partial T]{\uv[T]},
  \end{aligned}
  $$
  where we have used the fact that $(\dGT-\GRAD\dT)\uv[T]\in\Poly{k}(T)^d$ in the first line,~\eqref{eq:inspiration.residual} in the second line, the Cauchy--Schwarz inequality in the third line, and the discrete trace inequality~\eqref{eq:Lp.trace.discrete} below with $p=2$ to infer $h_T^{\frac{1}{2}}\norm[L^2(\partial T)]{\bphi\SCAL\normal_T}\lesssim\norm[L^2(T)^d]{\bphi}$ and conclude.
\end{remark}
\begin{remark}[Stabilisation based on a lifting of face-based differences]\label{rem:stabilisation.idea}
    Let now $\uv[T]\in\UT$ be fixed.
    Relation~\eqref{eq:inspiration.residual} no longer holds true in general if we replace $\bphi$ by a function $\vec{\eta}$ belonging to a space $\RangeST$ larger than $\Poly{k}(T)^d$.
    It then makes sense to define the nontrivial residual linear form $\mathcal{R}_T(\uv[T];\cdot):\RangeST\to\Real$ such that
    $$
    \mathcal{R}_T(\uv[T];\vec{\eta})
    \coloneq
    - ((\dGT-\GRAD\dT)\uv[T],\vec{\eta})_T
    + \sum_{F\in\Fh[T]}((\dTF-\dT)\uv[T],\vec{\eta}\SCAL\normal_{TF})_F.
    $$
    Assume now $\RangeST$ large enough for the $L^2(T)^d$-norm of the Riesz representation $\vec{\rm L}_T\uv[T]\in \RangeST$ of $\mathcal{R}_T(\uv[T];\cdot)$ to control $\seminorm[2,\partial T]{\uv[T]}$ (hence also $\norm[L^2(T)^d]{(\dGT-\GRAD\dT)\uv[T]}$ by~\eqref{eq:control.dGT-GdT}).
    Property \rprop{S1} is then fulfilled letting the stabilising contribution in~\eqref{eq:DSGDM:grD} be such that $\ST\uv[T]=\vec{\rm L}_T\uv[T]$ for all $\uv[T]$.
    This choice also satisfies \rprop{S2} by construction since $\Poly{k}(T)^d\subset\RangeST$ and $\mathcal{R}_T(\uv[T];\cdot)$ vanishes on $\Poly{k}(T)^d$ for any $\uv[T]\in\UT$ owing to~\eqref{eq:inspiration.residual}.
    Finally, property \rprop{S3} is satisfied provided that $\RangeST$ is a piecewise polynomial space on a regular polytopal partition of $T$.

    The above procedure can be interpreted as a lifting on $\RangeST$ of the face-based differences $\{(\dTF-\dT)\uv[T]\st F\in\Fh[T]\}$ realised by means of the operator $\vec{\rm L}_T$.
    This interpretation justifies the terminology employed in Section~\ref{sec:lifting} below.
\end{remark}

\subsubsection{A Raviart--Thomas--N\'ed\'elec subspace}

In this section we define a good candidate to play the role of the space $\RangeST$ in Remark~\ref{rem:stabilisation.idea}.
From this point on, we work on a fixed mesh element $T\in\Th$ and assume, for the sake of simplicity, that
\begin{inparaenum}[(i)]
\item the faces of $T$ are $(d{-}1)$-simplices and that
\item $T$ is star-shaped with respect to a point $\vec{x}_T$ whose ortogonal distance $d_{TF}$ from each face $F\in\Fh[T]$ satisfies
  \begin{equation}\label{eq:dTF.hT}
    d_{TF}\ge\varrho h_T,
  \end{equation}
\end{inparaenum}
where, as in Section~\ref{sec:convergent.GS}, $\varrho$ denotes the mesh regularity parameter.
These assumptions can be relaxed using a simplicial submesh of $T$ and at the price of a heavier notation.
For all $F\in\Fh[T]$, we denote by $\PTF$ the $d$-simplex of base $F$ and apex $\vec{x}_T$, and by $\Fh[TF]$ the set of $(d-1)$-simplicial faces of $\PTF$, see Figure \ref{fig-PTF}.
In what follows, we work on the face-based simplicial partition $\mathcal{P}_T\coloneq\{\PTF\st F\in\Fh[T]\}$.

\begin{figure}
\begin{center}
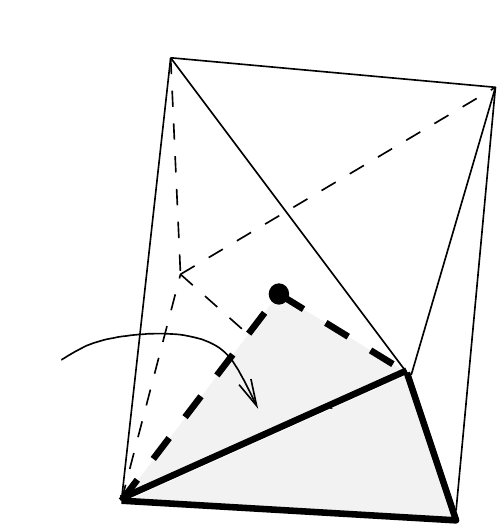
\caption{Illustration of $P_{TF}$.}
\label{fig-PTF}
\end{center}
\end{figure}

For an integer $m\ge 0$ and a face $F\in\Fh[T]$, we let $\RT[m](\PTF)\coloneq\Poly{m}(\PTF)^d+\vec{x}\mathbb{P}^m(\PTF)$ denote the Raviart--Thomas--N\'ed\'elec space~\cite{Raviart.Thomas:77,Nedelec:80} of degree $m$ on the simplex $\PTF$.
Each function $\vec{\eta}\in\RT[m](\PTF)$ is uniquely identified by the following degrees of freedom (see, e.g.,~\cite[Proposition~2.3.4]{Boffi.Brezzi.ea:13}):
$$
\begin{alignedat}{2}
  (\vec{\eta}\SCAL\normal_\sigma,q)_\sigma&\qquad&&\forall \sigma\in\Fh[TF],\forall q\in\Poly{m}(\sigma),
  \\
  (\vec{\eta},\vec{\chi})_{\PTF} &\qquad&&\forall\vec{\chi}\in\Poly{m-1}(\PTF)^d,
\end{alignedat}
$$
where, for all $\sigma\in\Fh[TF]$, the normal $\normal_\sigma$ points out of $\PTF$.
Additionally, we note the following relation, valid for all $\vec{\eta}\in\RT[m](\PTF)$:
\begin{equation}\label{eq:control.RT}
  \norm[L^2(P_{TF})^d]{\vec{\eta}}^2
  \simeq
  \norm[L^2(P_{TF})^d]{\vlproj[P_{TF}]{m-1}\vec{\eta}}^2
  + \sum_{\sigma\in\Fh[TF]}h_F\norm[L^2(\sigma)]{\vec{\eta}\SCAL\normal_\sigma}^2,
\end{equation}
where $\simeq$ means $Ca\le b\le C^{-1}a$ with real number $C>0$ independent of $h$ and of $\PTF$, but possibly depending on $\varrho$ and $m$.

The candidate to play the role of the space $\RangeST$ in Remark~\ref{rem:stabilisation.idea} is $\RT(\mathcal{P}_T)$, the broken Raviart--Thomas--N\'ed\'elec space of degree $(k+1)$ on the submesh $\mathcal{P}_T$.

\subsubsection{Lifting of face-based differences}\label{sec:lifting}

We are now ready to construct the lifting of face-based differences.
Owing to the specific choice of $\RangeST$, we can proceed face by face.
Specifically, for all $F\in\Fh[T]$, we define the lifting operator $\LTF:\UT\to\RT(\PTF)$ such that, for all $\uv[T]\in\UT$, $\LTF\uv[T]$ satisfies for all $\vec{\eta}\in\RT(\PTF)$
\begin{equation}\label{eq:RTF}
  (\LTF\uv[T],\vec{\eta})_{\PTF}
  = 
  -((\dGT-\GRAD\dT)\uv[T],\vec{\eta})_{\PTF}
  + ((\dTF-\dT)\uv[T],\vec{\eta}\SCAL\normal_{TF})_F.
\end{equation}%
In what follows, we extend $\LTF\uv[T]$ by zero outside $\PTF$.
\begin{proposition}[Stabilisation based on a Raviart--Thomas--N\'ed\'elec subspace]\label{prop:stabRTN}
  The following stabilising contribution satisfies properties \rprop{S1}--\rprop{S3}:
  \begin{equation}\label{eq:ST.RT}
    \ST\coloneq\sum_{F\in\Fh[T]}\LTF.
  \end{equation}
\end{proposition}
\begin{proof}
  (i) \emph{Proof of \rprop{S1}.}
  We abridge by $a\lesssim b$ the inequality $a\le Cb$ with real number $C$ independent of both $h$ and $T$, but possibly depending on $d$, $\varrho$, $k$, and $l$.
  We start by proving that
  \begin{equation}\label{eq:ST.RT:1}
    \seminorm[2,\partial T]{\uv[T]}\lesssim\norm[L^2(T)^d]{\ST\uv[T]}.
  \end{equation}
  Let $\vec{\eta}\in\RT(\PTF)$ be such that
  \begin{equation}\label{def:eta.stab}
  \begin{alignedat}{2}
    (\vec{\eta}\SCAL\normal_{TF},q)_F &= h_F^{-1}((\dTF-\dT)\uv[T],q)_F &\qquad&\forall q\in\Poly{k+1}(F),
    \\
    (\vec{\eta}\SCAL\normal_\sigma,q)_{\sigma} &= 0 &\qquad&\forall\sigma\in\Fh[TF]\setminus\{F\},\forall q\in\Poly{k+1}(\sigma),
    \\
    (\vec{\eta},\vec{\chi})_{\PTF} &= 0 &\qquad&\forall\vec{\chi}\in\Poly{k}(\PTF)^d.
  \end{alignedat}
  \end{equation}
  Plugging this definition into~\eqref{eq:RTF} and using the Cauchy--Schwarz inequality, we infer that
  $$
  \begin{aligned}
    h_F^{-1}\norm[L^2(F)]{(\dTF-\dT)\uv[T]}^2
    &\le\norm[L^2(\PTF)^d]{\LTF\uv[T]}\norm[L^2(\PTF)^d]{\vec{\eta}}
    \\
    &\lesssim \norm[L^2(\PTF)^d]{\LTF\uv[T]} h_F^{-\frac12}\norm[L^2(F)]{(\dTF-\dT)\uv[T]},
  \end{aligned}
  $$
  where we have used~\eqref{eq:control.RT} and \eqref{def:eta.stab} with $q=\vec{\eta}\SCAL\normal_{TF}\in \Poly{k+1}(F)$
or $q=\vec{\eta}\SCAL\normal_\sigma\in\Poly{k+1}(\sigma)$ to estimate the $L^2$-norm of $\vec{\eta}$.
  Dividing by $h_F^{-\frac12}\norm[L^2(F)]{(\dTF-\dT)\uv[T]}$, squaring, summing over $F\in\Fh[T]$, and taking the square root of the resulting inequality proves~\eqref{eq:ST.RT:1}.

  Let us now prove that 
  \begin{equation}\label{eq:ST.RT:2}
    \norm[L^2(T)^d]{\ST\uv[T]}\lesssim\seminorm[2,\partial T]{\uv[T]}.
  \end{equation}
  Letting in~\eqref{eq:RTF} $\vec{\eta}=\LTF\uv[T]$, summing over $F\in\Fh[T]$, and using multiple times the Cauchy--Schwarz inequality, we obtain
  $$
  \begin{aligned}
    \norm[L^2(T)^d]{\ST\uv[T]}^2
    &= \sum_{F\in\Fh[T]}\norm[L^2(\PTF)^d]{\LTF\uv[T]}^2
    \\
    &\le\bigg(
    \norm[L^2(T)^d]{(\dGT-\GRAD\dT)\uv[T]}^2 + \sum_{F\in\Fh[T]} h_F^{-1}\norm[L^2(F)]{(\dTF-\dT)\uv[T]}^2
    \bigg)^{\frac12}
    \\
    &\qquad
    \times\bigg(
    \norm[L^2(T)^d]{\ST\uv[T]}^2 + h_T \norm[L^2(\partial T)]{\ST\uv[T]\SCAL\normal_T}^2
    \bigg)^{\frac12}
    \\
    &\lesssim\seminorm[2,\partial T]{\uv[T]}\norm[L^2(T)^d]{\ST\uv[T]},
  \end{aligned}
  $$
  where we have used ~\eqref{eq:control.dGT-GdT} and the discrete trace inequality~\eqref{eq:Lp.trace.discrete} below with $p=2$ to conclude.
  Combining~\eqref{eq:ST.RT:1} and~\eqref{eq:ST.RT:2} gives \rprop{S1}.

  (ii) \emph{Proof of \rprop{S2}.}
  Let $\bphi\in\Poly{k}(T)^d$, set $\vec{\eta}=\bphi_{|\PTF}\in\Poly{k}(\PTF)^d\subset\RT(\PTF)$ in~\eqref{eq:RTF}, and sum over $F\in\Fh[T]$ to obtain
  $$
  \begin{aligned}
    (\ST\uv[T],\bphi)_T
    &= \sum_{F\in\Fh[T]}(\LTF\uv[T],\bphi)_{\PTF}
    \\
    &= -((\dGT-\GRAD\dT)\uv[T],\bphi)_T + \sum_{F\in\Fh[T]}((\dTF-\dT)\uv[T],\bphi\SCAL\normal_{TF})_F=0,
  \end{aligned}
  $$
  where we have used~\eqref{eq:inspiration.residual} to conclude.

  (iii) \emph{Proof of \rprop{S3}.}
  The regularity of the face-based partition $\mathcal{P}_T$ follows from~\eqref{eq:dTF.hT} in view of~\cite[Lemma~3]{Di-Pietro.Lemaire:15}.
  Property \rprop{S3} is then satisfied with $k_{\rm S}\coloneq k+2$.
\end{proof}


\begin{remark}[The lowest-order case]
  The lifting operators $\vec{L}_{TF}$ and stabilisation $\ST$ described above are some of the possible choices that satisfy \rprop{S1}--\rprop{S3}. In certain cases, simpler liftings can be designed. Consider for example $k=0$ and $l\in\{-1,0\}$.
  In this case, \begin{inparaenum}[(i)]
  \item $\dGT\uv[T]=\vec{0}$ since $\Poly{0}(T)^d=\GRAD\Poly{1}(T)$, so that $\GT[0]=\GRAD\rT[1]$;
  \item $\dT[-1]=\dT[0]=0$ by \eqref{eq:rT.average}.
  \end{inparaenum}
  Hence, \eqref{eq:RTF} reduces to 
  \begin{equation}\label{eq:RTF.0}
    (\LTF[1]\uv[T],\vec{\eta})_{\PTF}=(\dTF\uv[T],\vec{\eta}\SCAL\normal_{TF})_F.
  \end{equation}
  An appropriate lifting can then be constructed in $\Poly{0}(\PTF)^d$ instead of $\RT[1](\PTF)$ as described hereafter, and the assumption that the faces of $T$ are simplices can be removed (for any $F\in\Fh[T]$, $P_{TF}$ is then the pyramid with base $F$ and apex $\vec{x}_T$).
  Define $\vec{L}_{TF}^{\mbox{\tiny\textsc{HMM}}}:\UT[0,l]\to\Poly{0}(\PTF)^d$ such that, for all $\uv[T]\in\UT[0,l]$,
  $$
  \vec{L}_{TF}^{\mbox{\tiny\textsc{HMM}}}\uv[T] = \frac{|F|}{|\PTF|}\dTF[0]\uv[T]\normal_{TF}
  = \frac{d}{d_{TF}}\dTF[0]\uv[T]\normal_{TF}.
  $$
  Notice that this lifting is designed to satisfy \eqref{eq:RTF.0} for all $\vec{\eta}\in \Poly{0}(\PTF)^d$.
  Defining, for all $T\in\Th$, $\ST$ by \eqref{eq:ST.RT}, the discrete elements of the corresponding DSGD $\mathcal{D}$ are then
  \[
  \XDz=\{\uv=((v_T)_{T\in\Th},(v_F)_{F\in\Fh})\st v_T\in\Real\,,\;v_F\in\Real\,,\;v_F=0\quad\forall
  F\in\Fhb\},
  \]
  and, for all $\uv\in \XDz$ and all $T\in\Th$, expressing \eqref{eq:GT}, \eqref{eq:rT}
  and \eqref{eq:dT.dTF} for $k=0$, 
  \begin{align*}
    (\PiD \uv)_{|T}={}&v_T\,,\\
    (\grD \uv)_{|P_{TF}}={}&\GT[0]\uv + \frac{d}{d_{TF}}\left(v_T+\GT[0]\uv[T]\cdot(\overline{\vec{x}}_F-\vec{x}_T)-v_F\right)\normal_{TF}\quad
    \forall F\in\Fh[T]
  \end{align*}
  where $\overline{\vec{x}}_F$ is the centre of mass of $F$ and
  $\GT[0]\uv[T] \coloneq \frac{1}{|T|}\sum_{F\in\Fh[T]}|F|v_F\normal_{TF}$. 

  This gradient discretisation corresponds, for $k=l=0$, to the Hybrid-Mixed-Mimetic methods of~\cite{Droniou.Eymard.ea:10}
  (with the isomorphism $A_T=-\frac{1}{\sqrt{d}}{\rm Id}$ in \cite[Section 5.3.1]{Droniou.Eymard.ea:10}).
  The coercivity, GD-consistency, limit--conformity and compactness of families of such GDs are established in
  \cite[Chapter 13]{gdm}, and can also be proved by checking that $\ST$ satisfies \rprop{S1}--\rprop{S3}
  (use $\vec{\eta}=h_F^{-1}\dTF[0]\uv[T]\normal_{TF}$ in \eqref{eq:RTF.0} to establish $\gtrsim$ in 
  \eqref{eq:ST:stability}, and \cite[Eq. (13.10)]{gdm} to see that $\int_T\ST\uv[T]=\vec{0}$ and thus that
  \rprop{S2} holds).
\end{remark}

\subsection{Numerical examples}\label{sec:numerical.examples}

In this section we provide numerical evidence to support our theoretical results.

\subsubsection{Trigonometric solution for $p\ge 2$}\label{sec:numerical.examples:trigonometric}

We solve inside the two-dimensional unit square $\Omega=(0,1)^2$ the $p$-Laplace problem~\eqref{eq:plap:weak} corresponding to the exact solution
\begin{equation}\label{eq:trigonometric}
  u(\vec{x}) = \sin(\pi x_1)\sin(\pi x_2),
\end{equation}
with $p\in\{2,3,4\}$ and source term inferred from $u$.
\begin{figure}\centering
  \includegraphics[height=3cm]{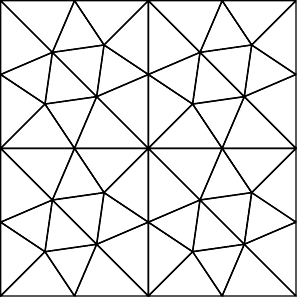}
  \hspace{0.25cm}
  \includegraphics[height=3cm]{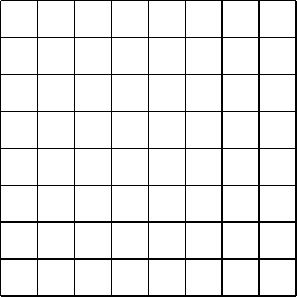}
  \hspace{0.25cm}
  \includegraphics[height=3cm]{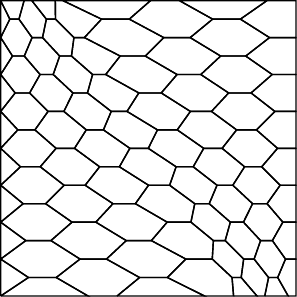}
  \hspace{0.25cm}
  \includegraphics[height=3cm]{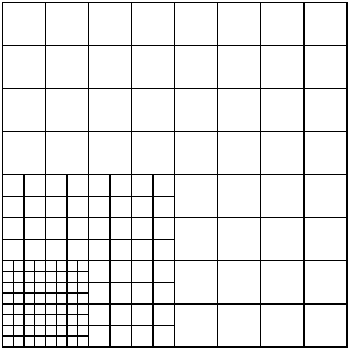}
  \hspace{0.25cm}
  \caption{Triangular, Cartesian, hexagonal and locally refined meshes for the numerical examples of Section~\ref{sec:numerical.examples}\label{fig:meshes}}
\end{figure}
We consider the matching triangular, Cartesian, (predominantly) hexagonal, and locally refined mesh families depicted in Figure \ref{fig:meshes} and polynomial degrees ranging from $0$ to $4$.
The first, second, and fourth mesh families originate from the FVCA5 benchmark~\cite{Herbin.Hubert:08}, whereas the third from~\cite{Di-Pietro.Lemaire:15}.
The local refinement in the third mesh family has no specific meaning here: its purpose is to demonstrate the seamless treatment of nonconforming interfaces.

\begin{figure}\centering
  \begin{footnotesize}
    \ref{legend:trigonometric.tria:dsgd}
  \end{footnotesize}
  \vspace{0.25cm} \\
  \begin{minipage}{0.32\textwidth}\centering
    \begin{tikzpicture}[scale=0.60]
      \begin{loglogaxis}[
          legend columns=-1,
          legend to name=legend:trigonometric.tria:dsgd,
          legend style={/tikz/every even column/.append style={column sep=0.35cm}}
        ]
        \addplot[red,mark=*,mark options={solid},thick] table[x=meshsize,y=err_p2]{plap-rt_0_mesh1.dat};
        \addplot[blue,mark=square*,mark options={solid},thick] table[x=meshsize,y=err_p2]{plap-rt_1_mesh1.dat};
        \addplot[green!50!black,mark=diamond*,mark options={solid},thick] table[x=meshsize,y=err_p2]{plap-rt_2_mesh1.dat};
        \addplot[brown,mark=triangle*,mark options={solid},thick] table[x=meshsize,y=err_p2]{plap-rt_3_mesh1.dat};
        \addplot[violet,mark=star,mark options={solid},thick] table[x=meshsize,y=err_p2]{plap-rt_4_mesh1.dat};

        \legend{$k=0$, $k=1$, $k=2$, $k=3$, $k=4$};
        
        \logLogSlopeTriangle{0.90}{0.3}{0.1}{1}{black};
        \logLogSlopeTriangle{0.90}{0.3}{0.1}{2}{black};
        \logLogSlopeTriangle{0.90}{0.3}{0.1}{3}{black};
        \logLogSlopeTriangle{0.90}{0.3}{0.1}{4}{black};
        \logLogSlopeTriangle{0.90}{0.3}{0.1}{5}{black};
      \end{loglogaxis}
    \end{tikzpicture}
    \subcaption{$p=2$, triangular}
  \end{minipage}  
  \begin{minipage}{0.32\textwidth}\centering
    \begin{tikzpicture}[scale=0.60]
      \begin{loglogaxis}
        \addplot[red,mark=*,mark options={solid},thick] table[x=meshsize,y=err_p3]{plap-rt_0_mesh1.dat};
        \addplot[blue,mark=square*,mark options={solid},thick] table[x=meshsize,y=err_p3]{plap-rt_1_mesh1.dat};
        \addplot[green!50!black,mark=diamond*,mark options={solid},thick] table[x=meshsize,y=err_p3]{plap-rt_2_mesh1.dat};
        \addplot[brown,mark=triangle*,mark options={solid},thick] table[x=meshsize,y=err_p3]{plap-rt_3_mesh1.dat};
        \addplot[violet,mark=star,mark options={solid},thick] table[x=meshsize,y=err_p3]{plap-rt_4_mesh1.dat};

        \logLogSlopeTriangle{0.875}{0.35}{0.1}{1/2}{black};
        \logLogSlopeTriangle{0.875}{0.35}{0.1}{1}{black};
        \logLogSlopeTriangle{0.875}{0.35}{0.1}{3/2}{black};
        \logLogSlopeTriangle{0.875}{0.35}{0.1}{2}{black};
        \logLogSlopeTriangle{0.875}{0.35}{0.1}{5/2}{black};
      \end{loglogaxis}        
    \end{tikzpicture}
    \subcaption{$p=3$, triangular}
  \end{minipage}
  \begin{minipage}{0.32\textwidth}\centering
    \begin{tikzpicture}[scale=0.60]
      \begin{loglogaxis}
        \addplot[red,mark=*,mark options={solid},thick] table[x=meshsize,y=err_p4]{plap-rt_0_mesh1.dat};
        \addplot[blue,mark=square*,mark options={solid},thick] table[x=meshsize,y=err_p4]{plap-rt_1_mesh1.dat};
        \addplot[green!50!black,mark=diamond*,mark options={solid},thick] table[x=meshsize,y=err_p4]{plap-rt_2_mesh1.dat};
        \addplot[brown,mark=triangle*,mark options={solid},thick] table[x=meshsize,y=err_p4]{plap-rt_3_mesh1.dat};
        \addplot[violet,mark=star,mark options={solid},thick] table[x=meshsize,y=err_p4]{plap-rt_4_mesh1.dat};

        \logLogSlopeTriangle{0.875}{0.35}{0.1}{1/3}{black};
        \logLogSlopeTriangle{0.875}{0.35}{0.1}{2/3}{black};
        \logLogSlopeTriangle{0.875}{0.35}{0.1}{1}{black};
        \logLogSlopeTriangle{0.875}{0.35}{0.1}{4/3}{black};
        \logLogSlopeTriangle{0.875}{0.35}{0.1}{5/3}{black};
      \end{loglogaxis}
    \end{tikzpicture}
    \subcaption{$p=4$, triangular}
  \end{minipage}
  \vspace{0.25cm} \\
  \begin{minipage}{0.32\textwidth}\centering
    \begin{tikzpicture}[scale=0.60]
      \begin{loglogaxis}
        \addplot[red,mark=*,mark options={solid},thick] table[x=meshsize,y=err_p2]{plap-rt_0_mesh2.dat};
        \addplot[blue,mark=square*,mark options={solid},thick] table[x=meshsize,y=err_p2]{plap-rt_1_mesh2.dat};
        \addplot[green!50!black,mark=diamond*,mark options={solid},thick] table[x=meshsize,y=err_p2]{plap-rt_2_mesh2.dat};
        \addplot[brown,mark=triangle*,mark options={solid},thick] table[x=meshsize,y=err_p2]{plap-rt_3_mesh2.dat};
        \addplot[violet,mark=star,mark options={solid},thick] table[x=meshsize,y=err_p2]{plap-rt_4_mesh2.dat};

        \logLogSlopeTriangle{0.90}{0.3}{0.1}{1}{black};
        \logLogSlopeTriangle{0.90}{0.3}{0.1}{2}{black};
        \logLogSlopeTriangle{0.90}{0.3}{0.1}{3}{black};
        \logLogSlopeTriangle{0.90}{0.3}{0.1}{4}{black};
        \logLogSlopeTriangle{0.90}{0.3}{0.1}{5}{black};        
      \end{loglogaxis}
    \end{tikzpicture}
    \subcaption{$p=2$, Cartesian}
  \end{minipage}  
  \begin{minipage}{0.32\textwidth}\centering
    \begin{tikzpicture}[scale=0.60]
      \begin{loglogaxis}
        \addplot[red,mark=*,mark options={solid},thick] table[x=meshsize,y=err_p3]{plap-rt_0_mesh2.dat};
        \addplot[blue,mark=square*,mark options={solid},thick] table[x=meshsize,y=err_p3]{plap-rt_1_mesh2.dat};
        \addplot[green!50!black,mark=diamond*,mark options={solid},thick] table[x=meshsize,y=err_p3]{plap-rt_2_mesh2.dat};
        \addplot[brown,mark=triangle*,mark options={solid},thick] table[x=meshsize,y=err_p3]{plap-rt_3_mesh2.dat};
        \addplot[violet,mark=star,mark options={solid},thick] table[x=meshsize,y=err_p3]{plap-rt_4_mesh2.dat};

        \logLogSlopeTriangle{0.875}{0.35}{0.1}{1/2}{black};
        \logLogSlopeTriangle{0.875}{0.35}{0.1}{1}{black};
        \logLogSlopeTriangle{0.875}{0.35}{0.1}{3/2}{black};
        \logLogSlopeTriangle{0.875}{0.35}{0.1}{2}{black};
        \logLogSlopeTriangle{0.875}{0.35}{0.1}{5/2}{black};
      \end{loglogaxis}
    \end{tikzpicture}
    \subcaption{$p=3$, Cartesian}
  \end{minipage}
  \begin{minipage}{0.32\textwidth}\centering
    \begin{tikzpicture}[scale=0.60]
      \begin{loglogaxis}
        \addplot[red,mark=*,mark options={solid},thick] table[x=meshsize,y=err_p4]{plap-rt_0_mesh2.dat};
        \addplot[blue,mark=square*,mark options={solid},thick] table[x=meshsize,y=err_p4]{plap-rt_1_mesh2.dat};
        \addplot[green!50!black,mark=diamond*,mark options={solid},thick] table[x=meshsize,y=err_p4]{plap-rt_2_mesh2.dat};
        \addplot[brown,mark=triangle*,mark options={solid},thick] table[x=meshsize,y=err_p4]{plap-rt_3_mesh2.dat};
        \addplot[violet,mark=star,mark options={solid},thick] table[x=meshsize,y=err_p4]{plap-rt_4_mesh2.dat};

        \logLogSlopeTriangle{0.875}{0.35}{0.1}{1/3}{black};
        \logLogSlopeTriangle{0.875}{0.35}{0.1}{2/3}{black};
        \logLogSlopeTriangle{0.875}{0.35}{0.1}{1}{black};
        \logLogSlopeTriangle{0.875}{0.35}{0.1}{4/3}{black};
        \logLogSlopeTriangle{0.875}{0.35}{0.1}{5/3}{black};       
      \end{loglogaxis}
    \end{tikzpicture}
    \subcaption{$p=4$, Cartesian}
  \end{minipage}        
  \vspace{0.25cm} \\
  \begin{minipage}{0.32\textwidth}\centering
    \begin{tikzpicture}[scale=0.60]
      \begin{loglogaxis}
        \addplot[red,mark=*,mark options={solid},thick] table[x=meshsize,y=err_p2]{plap-rt_0_pi6_tiltedhexagonal.dat};
        \addplot[blue,mark=square*,mark options={solid},thick] table[x=meshsize,y=err_p2]{plap-rt_1_pi6_tiltedhexagonal.dat};
        \addplot[green!50!black,mark=diamond*,mark options={solid},thick] table[x=meshsize,y=err_p2]{plap-rt_2_pi6_tiltedhexagonal.dat};
        \addplot[brown,mark=triangle*,mark options={solid},thick] table[x=meshsize,y=err_p2]{plap-rt_3_pi6_tiltedhexagonal.dat};
        \addplot[violet,mark=star,mark options={solid},thick] table[x=meshsize,y=err_p2]{plap-rt_4_pi6_tiltedhexagonal.dat};

        \logLogSlopeTriangle{0.90}{0.3}{0.1}{1}{black};
        \logLogSlopeTriangle{0.90}{0.3}{0.1}{2}{black};
        \logLogSlopeTriangle{0.90}{0.3}{0.1}{3}{black};
        \logLogSlopeTriangle{0.90}{0.3}{0.1}{4}{black};
        \logLogSlopeTriangle{0.90}{0.3}{0.1}{5}{black};
      \end{loglogaxis}
    \end{tikzpicture}
    \subcaption{$p=2$, hexagonal}
  \end{minipage}  
  \begin{minipage}{0.32\textwidth}\centering
    \begin{tikzpicture}[scale=0.60]
      \begin{loglogaxis}
        \addplot[red,mark=*,mark options={solid},thick] table[x=meshsize,y=err_p3]{plap-rt_0_pi6_tiltedhexagonal.dat};
        \addplot[blue,mark=square*,mark options={solid},thick] table[x=meshsize,y=err_p3]{plap-rt_1_pi6_tiltedhexagonal.dat};
        \addplot[green!50!black,mark=diamond*,mark options={solid},thick] table[x=meshsize,y=err_p3]{plap-rt_2_pi6_tiltedhexagonal.dat};
        \addplot[brown,mark=triangle*,mark options={solid},thick] table[x=meshsize,y=err_p3]{plap-rt_3_pi6_tiltedhexagonal.dat};
        \addplot[violet,mark=star,mark options={solid},thick] table[x=meshsize,y=err_p3]{plap-rt_4_pi6_tiltedhexagonal.dat};

        \logLogSlopeTriangle{0.875}{0.35}{0.1}{1/2}{black};
        \logLogSlopeTriangle{0.875}{0.35}{0.1}{1}{black};
        \logLogSlopeTriangle{0.875}{0.35}{0.1}{3/2}{black};
        \logLogSlopeTriangle{0.875}{0.35}{0.1}{2}{black};
        \logLogSlopeTriangle{0.875}{0.35}{0.1}{5/2}{black};        
      \end{loglogaxis}
    \end{tikzpicture}
    \subcaption{$p=3$, hexagonal}
  \end{minipage}
  \begin{minipage}{0.32\textwidth}\centering
    \begin{tikzpicture}[scale=0.60]
      \begin{loglogaxis}
        \addplot[red,mark=*,mark options={solid},thick] table[x=meshsize,y=err_p4]{plap-rt_0_pi6_tiltedhexagonal.dat};
        \addplot[blue,mark=square*,mark options={solid},thick] table[x=meshsize,y=err_p4]{plap-rt_1_pi6_tiltedhexagonal.dat};
        \addplot[green!50!black,mark=diamond*,mark options={solid},thick] table[x=meshsize,y=err_p4]{plap-rt_2_pi6_tiltedhexagonal.dat};
        \addplot[brown,mark=triangle*,mark options={solid},thick] table[x=meshsize,y=err_p4]{plap-rt_3_pi6_tiltedhexagonal.dat};
        \addplot[violet,mark=star,mark options={solid},thick] table[x=meshsize,y=err_p4]{plap-rt_4_pi6_tiltedhexagonal.dat};

        \logLogSlopeTriangle{0.875}{0.35}{0.1}{1/3}{black};
        \logLogSlopeTriangle{0.875}{0.35}{0.1}{2/3}{black};
        \logLogSlopeTriangle{0.875}{0.35}{0.1}{1}{black};
        \logLogSlopeTriangle{0.875}{0.35}{0.1}{4/3}{black};
        \logLogSlopeTriangle{0.875}{0.35}{0.1}{5/3}{black};
      \end{loglogaxis}
    \end{tikzpicture}
    \subcaption{$p=4$, hexagonal}
  \end{minipage}
  \vspace{0.25cm} \\
  \begin{minipage}{0.32\textwidth}\centering
    \begin{tikzpicture}[scale=0.60]
      \begin{loglogaxis}
        \addplot[red,mark=*,mark options={solid},thick] table[x=meshsize,y=err_p2]{plap-rt_0_mesh3.dat};
        \addplot[blue,mark=square*,mark options={solid},thick] table[x=meshsize,y=err_p2]{plap-rt_1_mesh3.dat};
        \addplot[green!50!black,mark=diamond*,mark options={solid},thick] table[x=meshsize,y=err_p2]{plap-rt_2_mesh3.dat};
        \addplot[brown,mark=triangle*,mark options={solid},thick] table[x=meshsize,y=err_p2]{plap-rt_3_mesh3.dat};
        \addplot[violet,mark=star,mark options={solid},thick] table[x=meshsize,y=err_p2]{plap-rt_4_mesh3.dat};

        \logLogSlopeTriangle{0.90}{0.3}{0.1}{1}{black};
        \logLogSlopeTriangle{0.90}{0.3}{0.1}{2}{black};
        \logLogSlopeTriangle{0.90}{0.3}{0.1}{3}{black};
        \logLogSlopeTriangle{0.90}{0.3}{0.1}{4}{black};
        \logLogSlopeTriangle{0.90}{0.3}{0.1}{5}{black};
      \end{loglogaxis}
    \end{tikzpicture}
    \subcaption{$p=2$, locally refined}
  \end{minipage}  
  \begin{minipage}{0.32\textwidth}\centering
    \begin{tikzpicture}[scale=0.60]
      \begin{loglogaxis}
        \addplot[red,mark=*,mark options={solid},thick] table[x=meshsize,y=err_p3]{plap-rt_0_mesh3.dat};
        \addplot[blue,mark=square*,mark options={solid},thick] table[x=meshsize,y=err_p3]{plap-rt_1_mesh3.dat};
        \addplot[green!50!black,mark=diamond*,mark options={solid},thick] table[x=meshsize,y=err_p3]{plap-rt_2_mesh3.dat};
        \addplot[brown,mark=triangle*,mark options={solid},thick] table[x=meshsize,y=err_p3]{plap-rt_3_mesh3.dat};
        \addplot[violet,mark=star,mark options={solid},thick] table[x=meshsize,y=err_p3]{plap-rt_4_mesh3.dat};

        \logLogSlopeTriangle{0.875}{0.35}{0.1}{1/2}{black};
        \logLogSlopeTriangle{0.875}{0.35}{0.1}{1}{black};
        \logLogSlopeTriangle{0.875}{0.35}{0.1}{3/2}{black};
        \logLogSlopeTriangle{0.875}{0.35}{0.1}{2}{black};
        \logLogSlopeTriangle{0.875}{0.35}{0.1}{5/2}{black};        
      \end{loglogaxis}
    \end{tikzpicture}
    \subcaption{$p=3$, locally refined}
  \end{minipage}
  \begin{minipage}{0.32\textwidth}\centering
    \begin{tikzpicture}[scale=0.60]
      \begin{loglogaxis}
        \addplot[red,mark=*,mark options={solid},thick] table[x=meshsize,y=err_p4]{plap-rt_0_mesh3.dat};
        \addplot[blue,mark=square*,mark options={solid},thick] table[x=meshsize,y=err_p4]{plap-rt_1_mesh3.dat};
        \addplot[green!50!black,mark=diamond*,mark options={solid},thick] table[x=meshsize,y=err_p4]{plap-rt_2_mesh3.dat};
        \addplot[brown,mark=triangle*,mark options={solid},thick] table[x=meshsize,y=err_p4]{plap-rt_3_mesh3.dat};
        \addplot[violet,mark=star,mark options={solid},thick] table[x=meshsize,y=err_p4]{plap-rt_4_mesh3.dat};

        \logLogSlopeTriangle{0.875}{0.35}{0.1}{1/3}{black};
        \logLogSlopeTriangle{0.875}{0.35}{0.1}{2/3}{black};
        \logLogSlopeTriangle{0.875}{0.35}{0.1}{1}{black};
        \logLogSlopeTriangle{0.875}{0.35}{0.1}{4/3}{black};
        \logLogSlopeTriangle{0.875}{0.35}{0.1}{5/3}{black};
      \end{loglogaxis}
    \end{tikzpicture}
    \subcaption{$p=4$, locally refined}
  \end{minipage}      
  \caption{$\norm[L^p(\Omega)^d]{\Gh(\Ih u-\uu[h])}$ v. $h$. Trigonometric test case, $p\in\{2,3,4\}$, DSGD.\label{fig:trigonometric:dsgd}}
\end{figure}
We report in Figure~\ref{fig:trigonometric:dsgd} the error $\norm[L^p(\Omega)^d]{\Gh(\Ih u-\uu)}$ versus the meshsize $h$, where $\Gh$ is the consistent (but in general not stable) global gradient reconstruction defined by \eqref{eq:Gh}.
  The reference slopes correspond to the convergence rates resulting from Theorem \ref{thm:error.est.example} together with Proposition \ref{prop:est.SD.WD} (more precisely, the order corresponds to the dominating term).
The theoretical orders of convergence are perfectly matched for $p=2$.
Similar considerations hold for $p=3$ and $k<2$ whereas, for $k\ge 2$, the order of convergence is limited by the regularity of the function $\vec{\psi}\mapsto|\vec{\psi}|^{p-2}\vec{\psi}$, which impacts
on the regularity of $|\GRAD u|^{p-2}\GRAD u$.
Finally, for $p=4$ the theoretical orders of convergence are matched for $k\in\{0,1,2\}$, while faster convergence than predicted by the error estimates is observed for $k\in\{3,4\}$.
This phenomenon will be further investigated in the future.
For a comparison with the HHO method of \cite{Di-Pietro.Droniou:16}, see Figure \ref{fig:trigonometric:hho} below.

\subsubsection{Exponential solution for $p<2$}

\begin{figure}\centering
  \begin{footnotesize}
    \ref{legend:exponential.tria:dsgd}
  \end{footnotesize}
  \vspace{0.25cm} \\
  \begin{minipage}{0.48\textwidth}\centering
    \begin{tikzpicture}[scale=0.70]
      \begin{loglogaxis}[
          legend columns=-1,
          legend to name=legend:exponential.tria:dsgd,
          legend style={/tikz/every even column/.append style={column sep=0.35cm}}
        ]
        \addplot[red,mark=*,mark options={solid},thick] table[x=meshsize,y=err_p]{plap-p-rt_0_mesh1.dat};
        \addplot[blue,mark=square*,mark options={solid},thick] table[x=meshsize,y=err_p]{plap-p-rt_1_mesh1.dat};
        \addplot[green!50!black,mark=diamond*,mark options={solid},thick] table[x=meshsize,y=err_p]{plap-p-rt_2_mesh1.dat};
        \addplot[brown,mark=triangle*,mark options={solid},thick] table[x=meshsize,y=err_p]{plap-p-rt_3_mesh1.dat};
        \addplot[violet,mark=star,mark options={solid},thick] table[x=meshsize,y=err_p]{plap-p-rt_4_mesh1.dat};

        \legend{$k=0$, $k=1$, $k=2$, $k=3$, $k=4$};
        
        \logLogSlopeTriangle{0.875}{0.4}{0.1}{3/4}{black};
        \logLogSlopeTriangle{0.875}{0.4}{0.1}{3/2}{black};
        \logLogSlopeTriangle{0.875}{0.4}{0.1}{9/4}{black};
        \logLogSlopeTriangle{0.875}{0.4}{0.1}{3}{black};
        \logLogSlopeTriangle{0.875}{0.4}{0.1}{15/4}{black};
      \end{loglogaxis}
    \end{tikzpicture}
    \subcaption{Triangular, DSDG}
  \end{minipage}
  \begin{minipage}{0.48\textwidth}\centering
    \begin{tikzpicture}[scale=0.70]
      \begin{loglogaxis}
        \addplot[red,mark=*,mark options={solid},thick] table[x=meshsize,y=err_p]{plap-p_0_mesh1.dat};
        \addplot[blue,mark=square*,mark options={solid},thick] table[x=meshsize,y=err_p]{plap-p_1_mesh1.dat};
        \addplot[green!50!black,mark=diamond*,mark options={solid},thick] table[x=meshsize,y=err_p]{plap-p_2_mesh1.dat};
        \addplot[brown,mark=triangle*,mark options={solid},thick] table[x=meshsize,y=err_p]{plap-p_3_mesh1.dat};
        \addplot[violet,mark=star,mark options={solid},thick] table[x=meshsize,y=err_p]{plap-p_4_mesh1.dat};

        \logLogSlopeTriangle{0.875}{0.4}{0.1}{3/4}{black};
        \logLogSlopeTriangle{0.875}{0.4}{0.1}{3/2}{black};
        \logLogSlopeTriangle{0.875}{0.4}{0.1}{9/4}{black};
        \logLogSlopeTriangle{0.875}{0.4}{0.1}{3}{black};
        \logLogSlopeTriangle{0.875}{0.4}{0.1}{15/4}{black};
      \end{loglogaxis}
    \end{tikzpicture}
    \subcaption{Triangular, HHO}
  \end{minipage}
  \vspace{0.25cm} \\
  \begin{minipage}{0.48\textwidth}\centering
    \begin{tikzpicture}[scale=0.70]
      \begin{loglogaxis}
        \addplot[red,mark=*,mark options={solid},thick] table[x=meshsize,y=err_p]{plap-p-rt_0_mesh2.dat};
        \addplot[blue,mark=square*,mark options={solid},thick] table[x=meshsize,y=err_p]{plap-p-rt_1_mesh2.dat};
        \addplot[green!50!black,mark=diamond*,mark options={solid},thick] table[x=meshsize,y=err_p]{plap-p-rt_2_mesh2.dat};
        \addplot[brown,mark=triangle*,mark options={solid},thick] table[x=meshsize,y=err_p]{plap-p-rt_3_mesh2.dat};
        \addplot[violet,mark=star,mark options={solid},thick] table[x=meshsize,y=err_p]{plap-p-rt_4_mesh2.dat};
        
        \logLogSlopeTriangle{0.875}{0.4}{0.1}{3/4}{black};
        \logLogSlopeTriangle{0.875}{0.4}{0.1}{3/2}{black};
        \logLogSlopeTriangle{0.875}{0.4}{0.1}{9/4}{black};
        \logLogSlopeTriangle{0.875}{0.4}{0.1}{3}{black};
        \logLogSlopeTriangle{0.875}{0.4}{0.1}{15/4}{black};
      \end{loglogaxis}
    \end{tikzpicture}    
    \subcaption{Cartesian, DSGD}
  \end{minipage}
  \begin{minipage}{0.48\textwidth}\centering
    \begin{tikzpicture}[scale=0.70]
      \begin{loglogaxis}
        \addplot[red,mark=*,mark options={solid},thick] table[x=meshsize,y=err_p]{plap-p_0_mesh2.dat};
        \addplot[blue,mark=square*,mark options={solid},thick] table[x=meshsize,y=err_p]{plap-p_1_mesh2.dat};
        \addplot[green!50!black,mark=diamond*,mark options={solid},thick] table[x=meshsize,y=err_p]{plap-p_2_mesh2.dat};
        \addplot[brown,mark=triangle*,mark options={solid},thick] table[x=meshsize,y=err_p]{plap-p_3_mesh2.dat};
        \addplot[violet,mark=star,mark options={solid},thick] table[x=meshsize,y=err_p]{plap-p_4_mesh2.dat};
        
        \logLogSlopeTriangle{0.875}{0.4}{0.1}{3/4}{black};
        \logLogSlopeTriangle{0.875}{0.4}{0.1}{3/2}{black};
        \logLogSlopeTriangle{0.875}{0.4}{0.1}{9/4}{black};
        \logLogSlopeTriangle{0.875}{0.4}{0.1}{3}{black};
        \logLogSlopeTriangle{0.875}{0.4}{0.1}{15/4}{black};
      \end{loglogaxis}
    \end{tikzpicture}    
    \subcaption{Cartesian, HHO}
  \end{minipage}
  \vspace{0.25cm} \\
  \begin{minipage}{0.48\textwidth}\centering
    \begin{tikzpicture}[scale=0.70]
      \begin{loglogaxis}
        \addplot[red,mark=*,mark options={solid},thick] table[x=meshsize,y=err_p]{plap-p-rt_0_pi6_tiltedhexagonal.dat};
        \addplot[blue,mark=square*,mark options={solid},thick] table[x=meshsize,y=err_p]{plap-p-rt_1_pi6_tiltedhexagonal.dat};
        \addplot[green!50!black,mark=diamond*,mark options={solid},thick] table[x=meshsize,y=err_p]{plap-p-rt_2_pi6_tiltedhexagonal.dat};
        \addplot[brown,mark=triangle*,mark options={solid},thick] table[x=meshsize,y=err_p]{plap-p-rt_3_pi6_tiltedhexagonal.dat};
        \addplot[violet,mark=star,mark options={solid},thick] table[x=meshsize,y=err_p]{plap-p-rt_4_pi6_tiltedhexagonal.dat};
        
        \logLogSlopeTriangle{0.875}{0.4}{0.1}{3/4}{black};
        \logLogSlopeTriangle{0.875}{0.4}{0.1}{3/2}{black};
        \logLogSlopeTriangle{0.875}{0.4}{0.1}{9/4}{black};
        \logLogSlopeTriangle{0.875}{0.4}{0.1}{3}{black};
        \logLogSlopeTriangle{0.875}{0.4}{0.1}{15/4}{black};
      \end{loglogaxis}
    \end{tikzpicture}    
    \subcaption{Hexagonal, DSGD}
  \end{minipage}
    \begin{minipage}{0.48\textwidth}\centering
    \begin{tikzpicture}[scale=0.70]
      \begin{loglogaxis}
        \addplot[red,mark=*,mark options={solid},thick] table[x=meshsize,y=err_p]{plap-p_0_pi6_tiltedhexagonal.dat};
        \addplot[blue,mark=square*,mark options={solid},thick] table[x=meshsize,y=err_p]{plap-p_1_pi6_tiltedhexagonal.dat};
        \addplot[green!50!black,mark=diamond*,mark options={solid},thick] table[x=meshsize,y=err_p]{plap-p_2_pi6_tiltedhexagonal.dat};
        \addplot[brown,mark=triangle*,mark options={solid},thick] table[x=meshsize,y=err_p]{plap-p_3_pi6_tiltedhexagonal.dat};
        \addplot[violet,mark=star,mark options={solid},thick] table[x=meshsize,y=err_p]{plap-p_4_pi6_tiltedhexagonal.dat};
        
        \logLogSlopeTriangle{0.875}{0.4}{0.1}{3/4}{black};
        \logLogSlopeTriangle{0.875}{0.4}{0.1}{3/2}{black};
        \logLogSlopeTriangle{0.875}{0.4}{0.1}{9/4}{black};
        \logLogSlopeTriangle{0.875}{0.4}{0.1}{3}{black};
        \logLogSlopeTriangle{0.875}{0.4}{0.1}{15/4}{black};
      \end{loglogaxis}
    \end{tikzpicture}    
    \subcaption{Hexagonal, HHO}
  \end{minipage}
  \vspace{0.25cm} \\
  \begin{minipage}{0.48\textwidth}\centering
    \begin{tikzpicture}[scale=0.70]
      \begin{loglogaxis}
        \addplot[red,mark=*,mark options={solid},thick] table[x=meshsize,y=err_p]{plap-p-rt_0_mesh3.dat};
        \addplot[blue,mark=square*,mark options={solid},thick] table[x=meshsize,y=err_p]{plap-p-rt_1_mesh3.dat};
        \addplot[green!50!black,mark=diamond*,mark options={solid},thick] table[x=meshsize,y=err_p]{plap-p-rt_2_mesh3.dat};
        \addplot[brown,mark=triangle*,mark options={solid},thick] table[x=meshsize,y=err_p]{plap-p-rt_3_mesh3.dat};
        \addplot[violet,mark=star,mark options={solid},thick] table[x=meshsize,y=err_p]{plap-p-rt_4_mesh3.dat};
        
        \logLogSlopeTriangle{0.875}{0.4}{0.1}{3/4}{black};
        \logLogSlopeTriangle{0.875}{0.4}{0.1}{3/2}{black};
        \logLogSlopeTriangle{0.875}{0.4}{0.1}{9/4}{black};
        \logLogSlopeTriangle{0.875}{0.4}{0.1}{3}{black};
        \logLogSlopeTriangle{0.875}{0.4}{0.1}{15/4}{black};
      \end{loglogaxis}
    \end{tikzpicture}    
    \subcaption{Locally refined, DSGD}
  \end{minipage}
    \begin{minipage}{0.48\textwidth}\centering
    \begin{tikzpicture}[scale=0.70]
      \begin{loglogaxis}
        \addplot[red,mark=*,mark options={solid},thick] table[x=meshsize,y=err_p]{plap-p_0_mesh3.dat};
        \addplot[blue,mark=square*,mark options={solid},thick] table[x=meshsize,y=err_p]{plap-p_1_mesh3.dat};
        \addplot[green!50!black,mark=diamond*,mark options={solid},thick] table[x=meshsize,y=err_p]{plap-p_2_mesh3.dat};
        \addplot[brown,mark=triangle*,mark options={solid},thick] table[x=meshsize,y=err_p]{plap-p_3_mesh3.dat};
        \addplot[violet,mark=star,mark options={solid},thick] table[x=meshsize,y=err_p]{plap-p_4_mesh3.dat};
        
        \logLogSlopeTriangle{0.875}{0.4}{0.1}{3/4}{black};
        \logLogSlopeTriangle{0.875}{0.4}{0.1}{3/2}{black};
        \logLogSlopeTriangle{0.875}{0.4}{0.1}{9/4}{black};
        \logLogSlopeTriangle{0.875}{0.4}{0.1}{3}{black};
        \logLogSlopeTriangle{0.875}{0.4}{0.1}{15/4}{black};
      \end{loglogaxis}
    \end{tikzpicture}    
    \subcaption{Locally refined, HHO}
  \end{minipage}
  \caption{Error $\norm[L^p(\Omega)^d]{\Gh[h](\Ih u-\uu)}$ v. $h$, exponential test case, $p=\frac74$.\label{fig:exponential:dsgd}}
\end{figure}
As pointed out in~\cite{Di-Pietro.Droniou:16*1}, the trigonometric solution~\eqref{eq:trigonometric} does not have the required regularity to assess the convergence order of the DSGD method when $1<p<2$.
For this reason, we consider instead the exponential solution
$$
u(\vec{x}) = \exp(x_1 + \pi x_2),
$$
and solve the $p$-Laplace problem with $p=\frac74$, Dirichlet boundary conditions on $\partial\Omega$, and right-hand side $f$ inferred from the expression of $u$.
The error $\norm[L^p(\Omega)^d]{\Gh(\Ih u-\uu)}$ versus the meshsize $h$ is plotted in Figure~\ref{fig:exponential:dsgd} for the mesh families illustrated in Figure~\ref{fig:meshes} and polynomial degrees ranging from 0 to 4.
For the sake of completeness, a comparison with the HHO method~\eqref{eq:plap:hho} is also included.
We observe that the DSGD and HHO methods give similar results in terms of the selected error measure (which accounts only for the consistent part of the gradient, common to both methods).
When including the stabilisation seminorm in the error measure, computations not shown here for the sake of brevity hint to a slightly better accuracy for the HHO method.


\section{Alternative gradient and links with other methods}\label{sec:links}

In this section we discuss an alternative to the gradient reconstruction $\grT$ defined by~\eqref{eq:DSGDM:grD}, and links with other methods.

\subsection{A consistent gradient based on $\rT$}\label{sec:alternate.grad}

An alternative to using $\GT$ in \eqref{eq:DSGDM:grD} would be to use $\GRAD\rT$. This would lead
to a local gradient defined by
\begin{equation}\label{def:grad.rT}
(\grDt[h]\uv)_{|T}=\grTt\uv[T]:=\GRAD\rT \uv[T] + \STt\uv[T]\qquad\forall\uv\in\Uh\,,\;\forall T\in\Th.
\end{equation}
The stabilisation term $\STt$ is still required to satisfy the $L^2$-stability and boundedness property \rprop{S1} and 
the image property \rprop{S3}. As shown in \eqref{eq:norm.equiv} below, the property \rprop{S2}
is required on $\ST\uv[T]$ to ensure that it is orthogonal to the consistent part $\GT\uv[T]$
of $\grT\uv[T]$. For $\grTt\uv[T]$, the consistent part is $\GRAD\rT\uv[T]\in\GRAD\Poly{k+1}(T)$, and the orthogonality
property on $\STt\uv[T]$ can therefore be relaxed into
\begin{enumerate}[\bf ($\widetilde{\bf S2}$)]
\item For all $\uv[T]\in\UT$ and all $\vec{\phi}\in \GRAD \Poly{k+1}(T)$,
$(\STt\uv[T],\vec{\phi})_T=0$,
\end{enumerate}
where, with respect to \rprop{S2}, the space for $\vec{\phi}$ is $\GRAD\Poly{k+1}(T)$ instead of $\Poly{k}(T)^d$.

Performing integration-by-parts and introducing projection operators in \eqref{eq:rT.pde} leads to 
\begin{equation}\label{eq:insp.2}
-(\GRAD\dT[l]\uv[T],\GRAD w)_T=\sum_{F\in\Fh[T]}((\dTF[k]-\dT[l])\uv[T],\GRAD w
\SCAL\normal_{TF})_T\,,\; \forall w\in\Poly{k+1}(T).
\end{equation}
Comparing with \eqref{eq:inspiration.residual}, we see that the volumetric term $\dGT\uv[T]$ has disappeared, and
the only remaining volumetric term $\GRAD\dT[l]\uv[T]$ belongs to $\GRAD\Poly{l}(T)$.
This is a gain with respect to the situation in Section \ref{sec:stab.RT}, in which
the degree of the volumic term was constrained by $\dGT\uv[T]$ to be $k$, whatever the
choice of $l$.
As a consequence, the construction of $\STt$ can be done in a (possibly) smaller space than $\RT(\PTF)$, as detailed in what follows.

First, considering $w=-\dT[l]\uv[T]$ in \eqref{eq:insp.2} and using the trace
inequality \eqref{eq:Lp.trace.discrete} with $p=2$ on $\GRAD \dT[l]\uv[T]$ yields 
(compare with Remark \ref{rem:control.volumic})
\begin{equation}\label{control.volumic.alt}
\norm[L^2(T)^d]{\GRAD\dT[l]\uv[T]}\lesssim\seminorm[2,\partial T]{\uv[T]}.
\end{equation}
Then, following the ideas of Section \ref{sec:lifting}, we construct the lifting
$\LTFt:\UT\to \RT[\max(l,k)](P_{TF})$ such that, for all $\vec{\eta}\in \RT[\max(l,k)](P_{TF})$, 
\begin{equation}\label{eq:RTFt}
  (\LTFt\uv[T],\vec{\eta})_{\PTF}
  = 
  (\GRAD\dT\uv[T],\vec{\eta})_{\PTF}
  + ((\dTF-\dT)\uv[T],\vec{\eta}\SCAL\normal_{TF})_F.
\end{equation}
Since $\GRAD\dT\uv[T]\in \GRAD\Poly{l}(T)\subset \Poly{l-1}(T)^d$
and $(\dTF-\dT)\uv[T]\in \Poly{\max(l,k)}(F)$, following the proof of \rprop{S1}
in Proposition \ref{prop:stabRTN} shows that, with this choice of $\LTFt$, we only need \eqref{def:eta.stab} to hold
for $q\in \Poly{\max(l,k)}(F)$ and $\vec{\chi}\in \Poly{l-1}(T)^d$.
The space $\RT[\max(l,k)](P_{TF})$ enables these choices of $q$ and $\vec{\chi}$,
and the $L^2$-stability of $\STt:=\sum_{F\in\Fh[T]}\LTFt$ therefore follows.
When $l\in\{k-1,k\}$, this stabilisation term is constructed on a piecewise $\RT[k]$ space, instead of a piecewise $\RT$ space for $\ST$ in Section \ref{sec:lifting}.

Although the choice \eqref{def:grad.rT} leads to coercive, consistent, limit-conforming, and
compact families of gradient discretisations, it can turn out to be far from optimal for general problems.
More precisely, its limit-conformity properties are much worse than those of \eqref{eq:DSGDM:grD}.
Indeed, the full orthogonality property \rprop{S2} is essential to establish the
$\mathcal O(h^{k+1})$ estimate \eqref{est:WD} on $\WD[h]$ (see Remark \ref{rem:GT.vs.grad.rT}).
In general, we cannot establish more than an $\mathcal O(h)$ estimate on $\WD[h]$, irrespective of $k$, if
the gradient is reconstructed via \eqref{def:grad.rT}. For anisotropic linear problems, a modification
of $\rT$ embedding a dependence on the diffusion coefficient can be constructed to recover optimal rates of convergence \cite{EDP15}; for fully non-linear models, though, the only option to recover a truly high-order method
seems to be using the gradient $\GT\uv[T]$ in the full polynomial space $\Poly{k}(T)^d$.


\begin{figure}\centering
  \begin{footnotesize}
    \ref{legend:trigonometric.tria:hho}
  \end{footnotesize}
  \vspace{0.25cm} \\
  \begin{minipage}{0.32\textwidth}\centering
    \begin{tikzpicture}[scale=0.60]
      \begin{loglogaxis}[
          legend columns=4,
          legend to name=legend:trigonometric.tria:hho,
          legend style={/tikz/every even column/.append style={column sep=0.35cm}}
        ]
        \addplot[red,mark=*,thick] table[x=meshsize,y=err_p2]{plap_0_mesh1.dat};
        \addplot[red,mark=*,mark options={solid},dashed,thick] table[x=meshsize,y=err_p2]{plap-pt_0_mesh1.dat};          
        \addplot[blue,mark=square*,thick] table[x=meshsize,y=err_p2]{plap_1_mesh1.dat};
        \addplot[blue,mark=square*,mark options={solid},dashed,thick] table[x=meshsize,y=err_p2]{plap-pt_1_mesh1.dat};
        \addplot[green!50!black,mark=diamond*,thick] table[x=meshsize,y=err_p2]{plap_2_mesh1.dat};
        \addplot[green!50!black,mark=diamond*,mark options={solid},dashed,thick] table[x=meshsize,y=err_p2]{plap-pt_2_mesh1.dat};
        \addplot[brown,mark=triangle*,thick] table[x=meshsize,y=err_p2]{plap_3_mesh1.dat};
        \addplot[brown,mark=triangle*,mark options={solid},dashed,thick] table[x=meshsize,y=err_p2]{plap-pt_3_mesh1.dat};
        \addplot[violet,mark=star,thick] table[x=meshsize,y=err_p2]{plap_4_mesh1.dat};
        \addplot[violet,mark=star,mark options={solid},dashed,thick] table[x=meshsize,y=err_p2]{plap-pt_4_mesh1.dat};

        \legend{$k=0$ ($\GT$),$k=0$ ($\GRAD\rT$),%
          $k=1$ ($\GT$),$k=1$ ($\GRAD\rT$),%
          $k=2$ ($\GT$),$k=2$ ($\GRAD\rT$),%
          $k=3$ ($\GT$),$k=3$ ($\GRAD\rT$),%
          $k=4$ ($\GT$),$k=4$ ($\GRAD\rT$)};

        \logLogSlopeTriangle{0.90}{0.3}{0.1}{1}{black};
        \logLogSlopeTriangle{0.90}{0.3}{0.1}{2}{black};
        \logLogSlopeTriangle{0.90}{0.3}{0.1}{3}{black};
        \logLogSlopeTriangle{0.90}{0.3}{0.1}{4}{black};
        \logLogSlopeTriangle{0.90}{0.3}{0.1}{5}{black};
      \end{loglogaxis}
    \end{tikzpicture}
    \subcaption{$p=2$, triangular}
  \end{minipage}  
  \begin{minipage}{0.32\textwidth}\centering
    \begin{tikzpicture}[scale=0.60]
      \begin{loglogaxis}
        \addplot[red,mark=*,thick] table[x=meshsize,y=err_p3]{plap_0_mesh1.dat};
        \addplot[red,mark=*,mark options={solid},dashed,thick] table[x=meshsize,y=err_p3]{plap-pt_0_mesh1.dat};          
        \addplot[blue,mark=square*,thick] table[x=meshsize,y=err_p3]{plap_1_mesh1.dat};
        \addplot[blue,mark=square*,mark options={solid},dashed,thick] table[x=meshsize,y=err_p3]{plap-pt_1_mesh1.dat};
        \addplot[green!50!black,mark=diamond*,thick] table[x=meshsize,y=err_p3]{plap_2_mesh1.dat};
        \addplot[green!50!black,mark=diamond*,mark options={solid},dashed,thick] table[x=meshsize,y=err_p3]{plap-pt_2_mesh1.dat};
        \addplot[brown,mark=triangle*,thick] table[x=meshsize,y=err_p3]{plap_3_mesh1.dat};
        \addplot[brown,mark=triangle*,mark options={solid},dashed,thick] table[x=meshsize,y=err_p3]{plap-pt_3_mesh1.dat};
        \addplot[violet,mark=star,thick] table[x=meshsize,y=err_p3]{plap_4_mesh1.dat};
        \addplot[violet,mark=star,mark options={solid},dashed,thick] table[x=meshsize,y=err_p3]{plap-pt_4_mesh1.dat};

        \logLogSlopeTriangle{0.875}{0.35}{0.1}{1/2}{black};
        \logLogSlopeTriangle{0.875}{0.35}{0.1}{1}{black};
        \logLogSlopeTriangle{0.875}{0.35}{0.1}{3/2}{black};
        \logLogSlopeTriangle{0.875}{0.35}{0.1}{2}{black};
        \logLogSlopeTriangle{0.875}{0.35}{0.1}{5/2}{black};
      \end{loglogaxis}        
    \end{tikzpicture}
    \subcaption{$p=3$, triangular}
  \end{minipage}
  \begin{minipage}{0.32\textwidth}\centering
    \begin{tikzpicture}[scale=0.60]
      \begin{loglogaxis}
        \addplot[red,mark=*,thick] table[x=meshsize,y=err_p4]{plap_0_mesh1.dat};
        \addplot[red,mark=*,mark options={solid},dashed,thick] table[x=meshsize,y=err_p4]{plap-pt_0_mesh1.dat};          
        \addplot[blue,mark=square*,thick] table[x=meshsize,y=err_p4]{plap_1_mesh1.dat};
        \addplot[blue,mark=square*,mark options={solid},dashed,thick] table[x=meshsize,y=err_p4]{plap-pt_1_mesh1.dat};
        \addplot[green!50!black,mark=diamond*,thick] table[x=meshsize,y=err_p4]{plap_2_mesh1.dat};
        \addplot[green!50!black,mark=diamond*,mark options={solid},dashed,thick] table[x=meshsize,y=err_p4]{plap-pt_2_mesh1.dat};
        \addplot[brown,mark=triangle*,thick] table[x=meshsize,y=err_p4]{plap_3_mesh1.dat};
        \addplot[brown,mark=triangle*,mark options={solid},dashed,thick] table[x=meshsize,y=err_p4]{plap-pt_3_mesh1.dat};
        \addplot[violet,mark=star,thick] table[x=meshsize,y=err_p4]{plap_4_mesh1.dat};
        \addplot[violet,mark=star,mark options={solid},dashed,thick] table[x=meshsize,y=err_p4]{plap-pt_4_mesh1.dat};

        \logLogSlopeTriangle{0.875}{0.35}{0.1}{1/3}{black};
        \logLogSlopeTriangle{0.875}{0.35}{0.1}{2/3}{black};
        \logLogSlopeTriangle{0.875}{0.35}{0.1}{1}{black};
        \logLogSlopeTriangle{0.875}{0.35}{0.1}{4/3}{black};
        \logLogSlopeTriangle{0.875}{0.35}{0.1}{5/3}{black};
      \end{loglogaxis}
    \end{tikzpicture}
    \subcaption{$p=4$, triangular}
  \end{minipage}
  \vspace{0.25cm} \\
  \begin{minipage}{0.32\textwidth}\centering
    \begin{tikzpicture}[scale=0.60]
      \begin{loglogaxis}
        \addplot[red,mark=*,thick] table[x=meshsize,y=err_p2]{plap_0_mesh2.dat};
        \addplot[red,mark=*,mark options={solid},dashed,thick] table[x=meshsize,y=err_p2]{plap-pt_0_mesh2.dat};          
        \addplot[blue,mark=square*,thick] table[x=meshsize,y=err_p2]{plap_1_mesh2.dat};
        \addplot[blue,mark=square*,mark options={solid},dashed,thick] table[x=meshsize,y=err_p2]{plap-pt_1_mesh2.dat};
        \addplot[green!50!black,mark=diamond*,thick] table[x=meshsize,y=err_p2]{plap_2_mesh2.dat};
        \addplot[green!50!black,mark=diamond*,mark options={solid},dashed,thick] table[x=meshsize,y=err_p2]{plap-pt_2_mesh2.dat};
        \addplot[brown,mark=triangle*,thick] table[x=meshsize,y=err_p2]{plap_3_mesh2.dat};
        \addplot[brown,mark=triangle*,mark options={solid},dashed,thick] table[x=meshsize,y=err_p2]{plap-pt_3_mesh2.dat};
        \addplot[violet,mark=star,thick] table[x=meshsize,y=err_p2]{plap_4_mesh2.dat};
        \addplot[violet,mark=star,mark options={solid},dashed,thick] table[x=meshsize,y=err_p2]{plap-pt_4_mesh2.dat};

        \logLogSlopeTriangle{0.90}{0.3}{0.1}{1}{black};
        \logLogSlopeTriangle{0.90}{0.3}{0.1}{2}{black};
        \logLogSlopeTriangle{0.90}{0.3}{0.1}{3}{black};
        \logLogSlopeTriangle{0.90}{0.3}{0.1}{4}{black};
        \logLogSlopeTriangle{0.90}{0.3}{0.1}{5}{black};        
      \end{loglogaxis}
    \end{tikzpicture}
    \subcaption{$p=2$, Cartesian}
  \end{minipage}  
  \begin{minipage}{0.32\textwidth}\centering
    \begin{tikzpicture}[scale=0.60]
      \begin{loglogaxis}
        \addplot[red,mark=*,thick] table[x=meshsize,y=err_p3]{plap_0_mesh2.dat};
        \addplot[red,mark=*,mark options={solid},dashed,thick] table[x=meshsize,y=err_p3]{plap-pt_0_mesh2.dat};          
        \addplot[blue,mark=square*,thick] table[x=meshsize,y=err_p3]{plap_1_mesh2.dat};
        \addplot[blue,mark=square*,mark options={solid},dashed,thick] table[x=meshsize,y=err_p3]{plap-pt_1_mesh2.dat};
        \addplot[green!50!black,mark=diamond*,thick] table[x=meshsize,y=err_p3]{plap_2_mesh2.dat};
        \addplot[green!50!black,mark=diamond*,mark options={solid},dashed,thick] table[x=meshsize,y=err_p3]{plap-pt_2_mesh2.dat};
        \addplot[brown,mark=triangle*,thick] table[x=meshsize,y=err_p3]{plap_3_mesh2.dat};
        \addplot[brown,mark=triangle*,mark options={solid},dashed,thick] table[x=meshsize,y=err_p3]{plap-pt_3_mesh2.dat};
        \addplot[violet,mark=star,thick] table[x=meshsize,y=err_p3]{plap_4_mesh2.dat};
        \addplot[violet,mark=star,mark options={solid},dashed,thick] table[x=meshsize,y=err_p3]{plap-pt_4_mesh2.dat};

        \logLogSlopeTriangle{0.875}{0.35}{0.1}{1/2}{black};
        \logLogSlopeTriangle{0.875}{0.35}{0.1}{1}{black};
        \logLogSlopeTriangle{0.875}{0.35}{0.1}{3/2}{black};
        \logLogSlopeTriangle{0.875}{0.35}{0.1}{2}{black};
        \logLogSlopeTriangle{0.875}{0.35}{0.1}{5/2}{black};
      \end{loglogaxis}
    \end{tikzpicture}
    \subcaption{$p=3$, Cartesian}
  \end{minipage}
  \begin{minipage}{0.32\textwidth}\centering
    \begin{tikzpicture}[scale=0.60]
      \begin{loglogaxis}
        \addplot[red,mark=*,thick] table[x=meshsize,y=err_p4]{plap_0_mesh2.dat};
        \addplot[red,mark=*,mark options={solid},dashed,thick] table[x=meshsize,y=err_p4]{plap-pt_0_mesh2.dat};          
        \addplot[blue,mark=square*,thick] table[x=meshsize,y=err_p4]{plap_1_mesh2.dat};
        \addplot[blue,mark=square*,mark options={solid},dashed,thick] table[x=meshsize,y=err_p4]{plap-pt_1_mesh2.dat};
        \addplot[green!50!black,mark=diamond*,thick] table[x=meshsize,y=err_p4]{plap_2_mesh2.dat};
        \addplot[green!50!black,mark=diamond*,mark options={solid},dashed,thick] table[x=meshsize,y=err_p4]{plap-pt_2_mesh2.dat};
        \addplot[brown,mark=triangle*,thick] table[x=meshsize,y=err_p4]{plap_3_mesh2.dat};
        \addplot[brown,mark=triangle*,mark options={solid},dashed,thick] table[x=meshsize,y=err_p4]{plap-pt_3_mesh2.dat};
        \addplot[violet,mark=star,thick] table[x=meshsize,y=err_p4]{plap_4_mesh2.dat};
        \addplot[violet,mark=star,mark options={solid},dashed,thick] table[x=meshsize,y=err_p4]{plap-pt_4_mesh2.dat};

        \logLogSlopeTriangle{0.875}{0.35}{0.1}{1/3}{black};
        \logLogSlopeTriangle{0.875}{0.35}{0.1}{2/3}{black};
        \logLogSlopeTriangle{0.875}{0.35}{0.1}{1}{black};
        \logLogSlopeTriangle{0.875}{0.35}{0.1}{4/3}{black};
        \logLogSlopeTriangle{0.875}{0.35}{0.1}{5/3}{black};       
      \end{loglogaxis}
    \end{tikzpicture}
    \subcaption{$p=4$, Cartesian}
  \end{minipage}        
  \vspace{0.25cm} \\
  \begin{minipage}{0.32\textwidth}\centering
    \begin{tikzpicture}[scale=0.60]
      \begin{loglogaxis}
        \addplot[red,mark=*,thick] table[x=meshsize,y=err_p2]{plap_0_pi6_tiltedhexagonal.dat};
        \addplot[red,mark=*,mark options={solid},dashed,thick] table[x=meshsize,y=err_p2]{plap-pt_0_pi6_tiltedhexagonal.dat};          
        \addplot[blue,mark=square*,thick] table[x=meshsize,y=err_p2]{plap_1_pi6_tiltedhexagonal.dat};
        \addplot[blue,mark=square*,mark options={solid},dashed,thick] table[x=meshsize,y=err_p2]{plap-pt_1_pi6_tiltedhexagonal.dat};
        \addplot[green!50!black,mark=diamond*,thick] table[x=meshsize,y=err_p2]{plap_2_pi6_tiltedhexagonal.dat};
        \addplot[green!50!black,mark=diamond*,mark options={solid},dashed,thick] table[x=meshsize,y=err_p2]{plap-pt_2_pi6_tiltedhexagonal.dat};
        \addplot[brown,mark=triangle*,thick] table[x=meshsize,y=err_p2]{plap_3_pi6_tiltedhexagonal.dat};
        \addplot[brown,mark=triangle*,mark options={solid},dashed,thick] table[x=meshsize,y=err_p2]{plap-pt_3_pi6_tiltedhexagonal.dat};
        \addplot[violet,mark=star,thick] table[x=meshsize,y=err_p2]{plap_4_pi6_tiltedhexagonal.dat};
        \addplot[violet,mark=star,mark options={solid},dashed,thick] table[x=meshsize,y=err_p2]{plap-pt_4_pi6_tiltedhexagonal.dat};

        \logLogSlopeTriangle{0.90}{0.3}{0.1}{1}{black};
        \logLogSlopeTriangle{0.90}{0.3}{0.1}{2}{black};
        \logLogSlopeTriangle{0.90}{0.3}{0.1}{3}{black};
        \logLogSlopeTriangle{0.90}{0.3}{0.1}{4}{black};
        \logLogSlopeTriangle{0.90}{0.3}{0.1}{5}{black};
      \end{loglogaxis}
    \end{tikzpicture}
    \subcaption{$p=2$, hexagonal}
  \end{minipage}  
  \begin{minipage}{0.32\textwidth}\centering
    \begin{tikzpicture}[scale=0.60]
      \begin{loglogaxis}
        \addplot[red,mark=*,thick] table[x=meshsize,y=err_p3]{plap_0_pi6_tiltedhexagonal.dat};
        \addplot[red,mark=*,mark options={solid},dashed,thick] table[x=meshsize,y=err_p3]{plap-pt_0_pi6_tiltedhexagonal.dat};          
        \addplot[blue,mark=square*,thick] table[x=meshsize,y=err_p3]{plap_1_pi6_tiltedhexagonal.dat};
        \addplot[blue,mark=square*,mark options={solid},dashed,thick] table[x=meshsize,y=err_p3]{plap-pt_1_pi6_tiltedhexagonal.dat};
        \addplot[green!50!black,mark=diamond*,thick] table[x=meshsize,y=err_p3]{plap_2_pi6_tiltedhexagonal.dat};
        \addplot[green!50!black,mark=diamond*,mark options={solid},dashed,thick] table[x=meshsize,y=err_p3]{plap-pt_2_pi6_tiltedhexagonal.dat};
        \addplot[brown,mark=triangle*,thick] table[x=meshsize,y=err_p3]{plap_3_pi6_tiltedhexagonal.dat};
        \addplot[brown,mark=triangle*,mark options={solid},dashed,thick] table[x=meshsize,y=err_p3]{plap-pt_3_pi6_tiltedhexagonal.dat};
        \addplot[violet,mark=star,thick] table[x=meshsize,y=err_p3]{plap_4_pi6_tiltedhexagonal.dat};
        \addplot[violet,mark=star,mark options={solid},dashed,thick] table[x=meshsize,y=err_p3]{plap-pt_4_pi6_tiltedhexagonal.dat};

        \logLogSlopeTriangle{0.875}{0.35}{0.1}{1/2}{black};
        \logLogSlopeTriangle{0.875}{0.35}{0.1}{1}{black};
        \logLogSlopeTriangle{0.875}{0.35}{0.1}{3/2}{black};
        \logLogSlopeTriangle{0.875}{0.35}{0.1}{2}{black};
        \logLogSlopeTriangle{0.875}{0.35}{0.1}{5/2}{black};        
      \end{loglogaxis}
    \end{tikzpicture}
    \subcaption{$p=3$, hexagonal}
  \end{minipage}
  \begin{minipage}{0.32\textwidth}\centering
    \begin{tikzpicture}[scale=0.60]
      \begin{loglogaxis}
        \addplot[red,mark=*,thick] table[x=meshsize,y=err_p4]{plap_0_pi6_tiltedhexagonal.dat};
        \addplot[red,mark=*,mark options={solid},dashed,thick] table[x=meshsize,y=err_p4]{plap-pt_0_pi6_tiltedhexagonal.dat};          
        \addplot[blue,mark=square*,thick] table[x=meshsize,y=err_p4]{plap_1_pi6_tiltedhexagonal.dat};
        \addplot[blue,mark=square*,mark options={solid},dashed,thick] table[x=meshsize,y=err_p4]{plap-pt_1_pi6_tiltedhexagonal.dat};
        \addplot[green!50!black,mark=diamond*,thick] table[x=meshsize,y=err_p4]{plap_2_pi6_tiltedhexagonal.dat};
        \addplot[green!50!black,mark=diamond*,mark options={solid},dashed,thick] table[x=meshsize,y=err_p4]{plap-pt_2_pi6_tiltedhexagonal.dat};
        \addplot[brown,mark=triangle*,thick] table[x=meshsize,y=err_p4]{plap_3_pi6_tiltedhexagonal.dat};
        \addplot[brown,mark=triangle*,mark options={solid},dashed,thick] table[x=meshsize,y=err_p4]{plap-pt_3_pi6_tiltedhexagonal.dat};
        \addplot[violet,mark=star,thick] table[x=meshsize,y=err_p4]{plap_4_pi6_tiltedhexagonal.dat};
        \addplot[violet,mark=star,mark options={solid},dashed,thick] table[x=meshsize,y=err_p4]{plap-pt_4_pi6_tiltedhexagonal.dat};     

        \logLogSlopeTriangle{0.875}{0.35}{0.1}{1/3}{black};
        \logLogSlopeTriangle{0.875}{0.35}{0.1}{2/3}{black};
        \logLogSlopeTriangle{0.875}{0.35}{0.1}{1}{black};
        \logLogSlopeTriangle{0.875}{0.35}{0.1}{4/3}{black};
        \logLogSlopeTriangle{0.875}{0.35}{0.1}{5/3}{black};
      \end{loglogaxis}
    \end{tikzpicture}
    \subcaption{$p=4$, hexagonal}
  \end{minipage}
  \vspace{0.25cm} \\
  \begin{minipage}{0.32\textwidth}\centering
    \begin{tikzpicture}[scale=0.60]
      \begin{loglogaxis}
        \addplot[red,mark=*,thick] table[x=meshsize,y=err_p2]{plap_0_mesh3.dat};
        \addplot[red,mark=*,mark options={solid},dashed,thick] table[x=meshsize,y=err_p2]{plap-pt_0_mesh3.dat};          
        \addplot[blue,mark=square*,thick] table[x=meshsize,y=err_p2]{plap_1_mesh3.dat};
        \addplot[blue,mark=square*,mark options={solid},dashed,thick] table[x=meshsize,y=err_p2]{plap-pt_1_mesh3.dat};
        \addplot[green!50!black,mark=diamond*,thick] table[x=meshsize,y=err_p2]{plap_2_mesh3.dat};
        \addplot[green!50!black,mark=diamond*,mark options={solid},dashed,thick] table[x=meshsize,y=err_p2]{plap-pt_2_mesh3.dat};
        \addplot[brown,mark=triangle*,thick] table[x=meshsize,y=err_p2]{plap_3_mesh3.dat};
        \addplot[brown,mark=triangle*,mark options={solid},dashed,thick] table[x=meshsize,y=err_p2]{plap-pt_3_mesh3.dat};
        \addplot[violet,mark=star,thick] table[x=meshsize,y=err_p2]{plap_4_mesh3.dat};
        \addplot[violet,mark=star,mark options={solid},dashed,thick] table[x=meshsize,y=err_p2]{plap-pt_4_mesh3.dat};

        \logLogSlopeTriangle{0.90}{0.3}{0.1}{1}{black};
        \logLogSlopeTriangle{0.90}{0.3}{0.1}{2}{black};
        \logLogSlopeTriangle{0.90}{0.3}{0.1}{3}{black};
        \logLogSlopeTriangle{0.90}{0.3}{0.1}{4}{black};
        \logLogSlopeTriangle{0.90}{0.3}{0.1}{5}{black};
      \end{loglogaxis}
    \end{tikzpicture}
    \subcaption{$p=2$, locally refined}
  \end{minipage}  
  \begin{minipage}{0.32\textwidth}\centering
    \begin{tikzpicture}[scale=0.60]
      \begin{loglogaxis}
        \addplot[red,mark=*,thick] table[x=meshsize,y=err_p3]{plap_0_mesh3.dat};
        \addplot[red,mark=*,mark options={solid},dashed,thick] table[x=meshsize,y=err_p3]{plap-pt_0_mesh3.dat};          
        \addplot[blue,mark=square*,thick] table[x=meshsize,y=err_p3]{plap_1_mesh3.dat};
        \addplot[blue,mark=square*,mark options={solid},dashed,thick] table[x=meshsize,y=err_p3]{plap-pt_1_mesh3.dat};
        \addplot[green!50!black,mark=diamond*,thick] table[x=meshsize,y=err_p3]{plap_2_mesh3.dat};
        \addplot[green!50!black,mark=diamond*,mark options={solid},dashed,thick] table[x=meshsize,y=err_p3]{plap-pt_2_mesh3.dat};
        \addplot[brown,mark=triangle*,thick] table[x=meshsize,y=err_p3]{plap_3_mesh3.dat};
        \addplot[brown,mark=triangle*,mark options={solid},dashed,thick] table[x=meshsize,y=err_p3]{plap-pt_3_mesh3.dat};
        \addplot[violet,mark=star,thick] table[x=meshsize,y=err_p3]{plap_4_mesh3.dat};
        \addplot[violet,mark=star,mark options={solid},dashed,thick] table[x=meshsize,y=err_p3]{plap-pt_4_mesh3.dat};

        \logLogSlopeTriangle{0.875}{0.35}{0.1}{1/2}{black};
        \logLogSlopeTriangle{0.875}{0.35}{0.1}{1}{black};
        \logLogSlopeTriangle{0.875}{0.35}{0.1}{3/2}{black};
        \logLogSlopeTriangle{0.875}{0.35}{0.1}{2}{black};
        \logLogSlopeTriangle{0.875}{0.35}{0.1}{5/2}{black};        
      \end{loglogaxis}
    \end{tikzpicture}
    \subcaption{$p=3$, locally refined}
  \end{minipage}
  \begin{minipage}{0.32\textwidth}\centering
    \begin{tikzpicture}[scale=0.60]
      \begin{loglogaxis}
        \addplot[red,mark=*,thick] table[x=meshsize,y=err_p4]{plap_0_mesh3.dat};
        \addplot[red,mark=*,mark options={solid},dashed,thick] table[x=meshsize,y=err_p4]{plap-pt_0_mesh3.dat};          
        \addplot[blue,mark=square*,thick] table[x=meshsize,y=err_p4]{plap_1_mesh3.dat};
        \addplot[blue,mark=square*,mark options={solid},dashed,thick] table[x=meshsize,y=err_p4]{plap-pt_1_mesh3.dat};
        \addplot[green!50!black,mark=diamond*,thick] table[x=meshsize,y=err_p4]{plap_2_mesh3.dat};
        \addplot[green!50!black,mark=diamond*,mark options={solid},dashed,thick] table[x=meshsize,y=err_p4]{plap-pt_2_mesh3.dat};
        \addplot[brown,mark=triangle*,thick] table[x=meshsize,y=err_p4]{plap_3_mesh3.dat};
        \addplot[brown,mark=triangle*,mark options={solid},dashed,thick] table[x=meshsize,y=err_p4]{plap-pt_3_mesh3.dat};
        \addplot[violet,mark=star,thick] table[x=meshsize,y=err_p4]{plap_4_mesh3.dat};
        \addplot[violet,mark=star,mark options={solid},dashed,thick] table[x=meshsize,y=err_p4]{plap-pt_4_mesh3.dat};

        \logLogSlopeTriangle{0.875}{0.35}{0.1}{1/3}{black};
        \logLogSlopeTriangle{0.875}{0.35}{0.1}{2/3}{black};
        \logLogSlopeTriangle{0.875}{0.35}{0.1}{1}{black};
        \logLogSlopeTriangle{0.875}{0.35}{0.1}{4/3}{black};
        \logLogSlopeTriangle{0.875}{0.35}{0.1}{5/3}{black};
      \end{loglogaxis}
    \end{tikzpicture}
    \subcaption{$p=4$, locally refined}
  \end{minipage}      
  \caption{$\norm[L^p(\Omega)^d]{\Gh(\Ih u-\uu[h])}$ v. $h$. Trigonometric test case, $p\in\{2,3,4\}$, HHO.\label{fig:trigonometric:hho}}
\end{figure}
To illustrate numerically the loss of convergence experienced when using the gradient reconstruction~\eqref{def:grad.rT} in the context of fully non-linear problems, we solve the problem described in Section~\ref{sec:numerical.examples:trigonometric} using two numerical methods: the HHO method of~\cite{Di-Pietro.Droniou:16} (see~\eqref{eq:plap:hho} below), and the method obtained from the latter replacing $\GT$ by $\nabla\rT$.
We report in Figure \ref{fig:trigonometric:hho} the error $\norm[1,p,h]{\Ih u-\uu}$ versus the meshsize $h$, with reference slopes corresponding to the estimates of convergence rates derived in~\cite[Theorem~3.2]{Di-Pietro.Droniou:16*1}.
The leftmost column, corresponding to the Poisson problem with $p=2$, shows that both $\GT$ and $\nabla\rT$ can be used when $\WD[h]$ is applied to $\vec{\psi}=\GRAD u$.
For $p\in\{3,4\}$, on the other hand, a significant loss in the convergence rate is observed for $k>1$ (the dashed line corresponding to $\GRAD\rT$ departs from the solid line corresponding to $\GT$).
Notice that, for $k\ge3$ and $p=3$, the order of convergence is again limited by the regularity of the function $\vec{\psi}\mapsto|\vec{\psi}|^{p-2}\vec{\psi}$.
The results presented here replace the ones of~\cite[Figure 3]{Di-Pietro.Droniou:16*1}, which were affected by a bug in one of the libraries used in our code.
From these new tests, the error estimates of~\cite[Theorem~3.2]{Di-Pietro.Droniou:16*1} appear to be sharp also for $p>2$.

\subsection{Hybrid High-Order methods}

The HHO method proposed in~\cite{Di-Pietro.Ern.ea:14} for problem~\eqref{lin.ell.w} with $\matr{\Lambda}=\matr{I}_d$ reads
\begin{equation}\label{eq:hho}
\text{Find $\uu\in\Uhz$ such that, for all $\uv\in\Uhz$, $a_h^{\rm hho}(\uu,\uv)=(f,v_h)$,}
\end{equation}
where the broken polynomial function $v_h$ is defined by \eqref{eq:vh}, and the bilinear form $a_h^{\rm hho}:\Uh\times\Uh\to\Real$ is assembled from the elementary contributions
\begin{equation}\label{aT.hho}
a_T^{\rm hho}(\uu[T],\uv[T])\coloneq (\GRAD\rT\uu[T],\GRAD\rT\uv[T])_T
+ \sum_{F\in\Fh[T]}h_F^{-1}((\dTF-\dT)\uu[T],(\dTF-\dT)\uv[T])_F.
\end{equation}
Here, the consistent gradient is $\GRAD\rT$, as in \eqref{def:grad.rT},
and the stabilisation is not incorporated in the gradient reconstruction, but rather added as a separate term in the bilinear form.

For the non-linear problem~\eqref{eq:plap:weak}, on the other hand, the HHO method considered in~\cite{Di-Pietro.Droniou:16,Di-Pietro.Droniou:16*1} reads
$$
\text{Find $\uu\in\Uhz$ such that, for all $\uv\in\Uhz$, $A_h^{\rm hho}(\uu,\uv)=(f,v_h)$,}
$$
with function $A_h^{\rm hho}:\Uh\times\Uh\to\Real$ assembled from the elementary contributions
\begin{multline}\label{eq:plap:hho}
  A_T^{\rm hho}(\uu[T],\uv[T])\coloneq
  \\
  \int_T\vec{\sigma}(\GT\uu[T])\SCAL\GT\uv[T]
  + \sum_{F\in\Fh[T]}h_F^{p-1}\int_F|(\dTF-\dT)\uu[T]|^{p-2}(\dTF-\dT)\uu[T](\dTF-\dT)\uv[T].
\end{multline}
Also in this case, stability is achieved by a separate term, and only the consistent (but not stable) gradient
reconstruction appears in the consistency term.
Notice that, unlike the linear case, the gradient reconstructed in the full polynomial space $\Poly{k}(T)^d$ is present here; see comments at the end of Section \ref{sec:alternate.grad}.

\subsection{High-order non-conforming Mimetic Finite Difference}\label{sec:MFD}

The ncMFD method of \cite{Lipnikov-Manzini:2014} hinges on degrees of freedom (DOFs) that
are the polynomial moments of degree up to $(k-1)$ inside the mesh elements and the polynomial moments of degree up to $k$ on the mesh faces.
$X_h$ is the space of vectors gathering such DOFs.
Given a function $v\in H^1(\Omega)$, we denote by $\vsI$ its interpolation in $X_h$, that is,
the vector collecting its moments in the elements and on the faces.
If $\vmfd\in X_h$, we denote by $v_T\in\Poly{k-1}(T)$, $T\in\Th$, and $v_F\in\Poly{k}(F)$, $F\in\Fh$, the polynomials reconstructed from the moments represented by $\vmfd$.
This defines an isomorphism
\begin{equation}\label{iso.ncMFD.U}
\vmfd\in X_h\mapsto \uv[h]=((v_T)_{T\in\Th},(v_F)_{F\in\Fh})\in\Uh[k,k-1].
\end{equation}
Denoting by $\vmfd_{|T}$ the sub-vector made of the DOFs of $\vmfd\in X_h$ in the mesh element $T\in\Th$ and on the mesh faces in $\Fh[T]$, the ncMFD method for problem \eqref{lin.ell.w} with $\matr{\Lambda}=\matr{I}_d$ reads
\begin{equation}\label{def:ncMFD}
\mbox{Find $\umfd\in X_h$ such that, for all $\vmfd\in X_h$, }\sum_{T\in\Th}\umfd_{|T}^t \mathsf{M}_T\vmfd_{|T}=
\sum_{T\in\Th}L_f(\vmfd_{|T}),
\end{equation}
where $L_f(\vmfd_T)$ is a discretisation of $\int_T f v$, and
the matrix $\mathsf{M}_T$ is positive semi-definite with suitable consistency and stability
properties. Setting $\Ndk={\rm dim}(\Poly{k-1}(T))$ and selecting a basis
$(q_i)_{i=0,\ldots,\Ndk-1}$ of $\Poly{k-1}(T)$ with $q_0=1$,
the required consistency and stability properties on $\mathsf{M}_T$ lead to
the following decomposition (see \cite[Eq. (35)]{Lipnikov-Manzini:2014}):
\[
\mathsf{M}_T=\mathsf{M}_T^0+\mathsf{M}_T^1=
\widehat{\mathsf{R}}_T(\widehat{\mathsf{R}}_T^t\widehat{\mathsf{N}})^{-1}\widehat{\mathsf{R}}_T^t+\mathsf{M}_T^1
\]
where $\mathsf{A}^t$ is the transpose of $\mathsf{A}$,
$\ker(\mathsf{M}_T^1)=\{(\vsI)_{|T}\,:\, v\in \Poly{k-1}(T)\}$, $\widehat{\mathsf{N}}$ has
columns $\widehat{\mathsf{N}}_i=(q_i)^{\INTP}$ for $i=1,\ldots,\Ndk-1$, and
$\widehat{\mathsf{R}}_T$ is the matrix with columns $(\widehat{\mathsf{R}}_{T,i})_{i=1,\ldots,\Ndk}$ defined by
\begin{equation}\label{MFD:def.R}
\forall \vmfd\in X_h\,,\;(\widehat{\mathsf{R}}_{T,i})^t \vmfd_{|T}=
-(v_T,\Delta q_i)+\sum_{F\in\Fh[T]}(v_F,\GRAD q_i\SCAL\normal_{TF})_F.
\end{equation}
The stabilising matrix $\mathsf{M}_T^1$ does not have any impact on the consistency of the method;
its sole role is to stabilise the matrix $\mathsf{M}$ so that its kernel is
$\Real 1^{\INTP}$ (which is expected: a matrix $\mathsf{M}_T$ representing the bilinear form
$\int_T \GRAD u\cdot\GRAD v$
  should vanish on interpolants of constant functions).
  The matrix $\mathsf{M}_T^0$, on the other hand, contains all the consistency
  properties of the method.
  The analysis of the stabilisation part is made in Section \ref{sec:generality.ST}, alongside the analysis of the stabilisation in the HHO and ncVEM methods.

Let us analyse here the consistent part $\mathsf{M}_T^0$ of $\mathsf{M}_T$.
Take $\vmfd\in X_h$, and let $\uv[h]\in \Uh[k,k-1]$
corresponding to $\vmfd$ through the isomorphism \eqref{iso.ncMFD.U}.
Comparing \eqref{eq:rT.pde} and \eqref{MFD:def.R} shows that
\[
(\GRAD\rT\uv[T],\GRAD q_i)_T=(\widehat{\mathsf{R}}_{T,i})^t\vmfd_{|T}=(\widehat{\mathsf{R}}_T^t\vmfd_{|T})_i
\,,\;\forall i=1,\ldots,\Ndk-1.
\]
Hence, with obvious notations,
\begin{align}
\vmfd_{|T}^t \mathsf{M}^0_T\wmfd_{|T}={}&(\widehat{\mathsf{R}}_T^t \vmfd_{|T})^t
(\widehat{\mathsf{R}}_T^t\widehat{\mathsf{N}})^{-1}(\widehat{\mathsf{R}}_T^t \wmfd_{|T})\nonumber\\
={}& [(\GRAD\rT\uv[T],\GRAD q_i)_T]_{i=1,\ldots,\Ndk-1}^t
(\widehat{\mathsf{R}}_T^t\widehat{\mathsf{N}})^{-1}
[(\GRAD\rT\uw[T],\GRAD q_i)_T]_{i=1,\ldots,\Ndk-1}.
\label{ncMFD.gram}\end{align}
Applying \eqref{MFD:def.R} to $\vmfd_{|T}=(q_j)^{\INTP}$ and integrating by parts shows that
$(\widehat{\mathsf{R}}_T^t\widehat{\mathsf{N}})_{ij}=(\GRAD q_i,\GRAD q_j)_T$, that is,
$\widehat{\mathsf{R}}_T^t\widehat{\mathsf{N}}$ is the Gram matrix of $(\GRAD q_i)_{i=1,\ldots,\Ndk-1}$ in $L^2(T)^d$.
Equation \eqref{ncMFD.gram} can therefore be re-written as
\[
\vmfd_{|T}^t \mathsf{M}_T^0\wmfd_{|T}= (\GRAD\rT\uv[T],\GRAD \rT\uw[T])_T.
\]
Thus, recalling the isomorphism \eqref{iso.ncMFD.U} between $X_h$ and $\Uh[k,k-1]$,
the ncMFD method \eqref{def:ncMFD} is equivalent to
\begin{equation*}
\mbox{Find $\uu\in \Uh[k,k-1]$ such that, for all $\uv\in \Uh[k,k-1]$, }
\sum_{T\in\Th}(\GRAD\rT\uu[T],\GRAD \rT\uv[T])_T+s_T^{\rm mfd}(\umfd,\vmfd)
=\sum_{T\in\Th}L_f(\vmfd_{|T}),
\end{equation*}
where $s_T^{\rm mfd}$ is the stabilising bilinear form assembled from the
local matrices $(\mathsf{M}_T^1)_{T\in\Th}$. Here, the consistent part is
constructed from $\GRAD \rT$, as in \eqref{def:grad.rT}, but the stabilisation
is external to the gradient.
In the ncMFD method, the loading term is discretised (for $k\ge 2$) by
\begin{equation}\label{ncMFD:L}
  L_f(\vmfd_{|T})=\int_T \lproj{k-1}(f)v_T=\int_T fv_T.
\end{equation}
A modification of \eqref{ncMFD:L} is necessary for the low orders $k=0,1$ -- see \cite[Section 2.7]{Lipnikov-Manzini:2014}.


\subsection{Non-conforming Virtual Element Method}\label{sec:ncVEM}
For a given mesh element $T\in\Th$, we define the local non-conforming
virtual element space as follows~\cite{Ayuso-de-Dios.Lipnikov.ea:16}:
\begin{equation}\label{eq:VhkT}
  \Vhk(T) \coloneq 
  \left\{
    \vh\in H^1(T)\,:\,\normal_{TF}\SCAL\GRAD {\vh}_{|_{F}}\in\Poly{k}(F)\quad\forall F\in\Fh[T],
    \quad\Delta\vh\in\Poly{k-1}(T)
  \right\}.
\end{equation}
This space is finite dimensional and its functions are described by
their face moments against the polynomials of degree up to $k$ on each
face $F\in\Fh[T]$ and the cell moments against polynomials of degree
up to $k-1$ inside cell $T$.
The unisolvency of these degrees of freedom has been proved
in~\cite{Ayuso-de-Dios.Lipnikov.ea:16}.

The \emph{global non-conforming virtual element space of degree $k$}
is given by:
\begin{equation}
\begin{aligned}
  \Vhkz(\Omega) \coloneq \bigg\{
  \vh\in L^2(\Omega)\st{}&
  \int_{F}\jump{\vh}q=0\quad\forall F\in\Fh\,,\;\forall q\in\Poly{k}(F)\\
  &\mbox{ and }{\vh}_{|T}\in\Vhk(T)\quad\forall T\in\Th
  \bigg\},
\end{aligned}
\end{equation}
where $\jump{\vh}$ denotes the \emph{jump operator} with the usual
definition at interfaces (the sign is not relevant), and extended to
boundary faces setting $\jump{\vh}\coloneq\vh$.

The polynomials of degree up to $(k+1)$ are a subspace of $\Vhk(T)$.
The jump conditions on the elements of $\Vhkz(\Omega)$ ensure that
$\Ih[k,k-1]:\Vhkz(\Omega)\to \Uhz[k,k-1]$ is well defined (there is only
one interpolant on each face, and the polynomial moments up to degree $k$ of functions in $\Vhkz(\Omega)$ vanish on each $F\in\Fhb$). Moreover, the unisolvent property of
the set of DOFs shows that $\Ih[k,k-1]$ is an isomorphism between
$\Vhkz(\Omega)$ and $\Uhz[k,k-1]$.

\medskip
The virtual element discretisation of problem~\eqref{lin.ell.w} with
$\vec{\Lambda}={\rm Id}$ reads as:
\begin{equation}
  \mbox{Find $\uh\in\Vhkz(\Omega)$ such that, for all $\vh\in\Vhkz(\Omega)$, }
  a_h(\uh,\vh)=(f_h,\vh)
\end{equation}
where the bilinear form $a_{h}(\uh,\vh)$ approximates the left-hand
side of~\eqref{lin.ell.w} and ${f_{h}}_{|T}=\lproj{k-1}f$.

Mimicking the additivity of integrals, we assume that the virtual
element bilinear form is the summation of local elemental terms
\begin{equation}
  a_{h}(\uh,\vh) = \sum_{T\in\Th}a_{h,T}(\uh,\vh).
\end{equation}
Two different formulations of the local bilinear form $a_{h,T}$ can be found in
the literature, both including a consistency and a
stability term:
\begin{description}
\item[--] \emph{first formulation}~\cite{Ayuso-de-Dios.Lipnikov.ea:16}:
  for every $\uh$, $\vh\in\Vhk(T)$:
  \begin{equation}\label{eq:ah.vem:second}
    a_{h,T}(\uh,\vh) \coloneq ( \GRAD\eproj[T]{k+1}\uh, \GRAD\eproj[T]{k+1}\vh )_T
    + \mathcal{S}\big( ( {\rm Id}-\eproj[T]{k+1} )\uh, ( {\rm Id}-\eproj[T]{k+1} )\vh \big);
  \end{equation}
\item[--] \emph{second formulation}~\cite{Cangiani.Manzini.Sutton:16,Cangiani.Gyrya.Manzini:16}:
  for every $\uh$, $\vh\in\Vhk(T)$:
  \begin{equation}\label{eq:ah.vem:first}
    a_{h,T}(\uh,\vh) \coloneq ( \vlproj{k}\GRAD\uh, \vlproj{k}\GRAD\vh )_T
    + \mathcal{S}\big( ( {\rm Id}-\eproj[T]{k+1} )\uh, ( {\rm Id}-\eproj[T]{k+1} )\vh \big).
   \end{equation}
\end{description}
In these definitions, the first term on the right is the consistency
term designed to provide the exactness of the integration whenever at
the least one of the entries $\mathfrak{u}_h$ or $\mathfrak{v}_h$ is a
polynomial of degree up to $(k+1)$.
The second term is a stabilisation, and $\mathcal{S}$ can be any
symmetric and positive definite bilinear form for which there exist
two positive constants $c_*$ and $c^*$ such that
\begin{equation}\label{ncVEM:stab}
  c_*\norm[L^2(T)^d]{\GRAD \vh}^2\leq\mathcal{S}(\vh,\vh)\leq c^*\norm[L^2(T)^d]{\GRAD \vh}^2
  \quad\forall\vh\in\Vhk(T)\textrm{~such~that~}\eproj{k+1}\vh=0.
\end{equation}

\begin{remark}
  The connection between formulation~\eqref{eq:ah.vem:second} and the
  ncMFD method of~\cite{Lipnikov-Manzini:2014} reviewed in
  Section~\ref{sec:MFD} has been established
  in~\cite{Ayuso-de-Dios.Lipnikov.ea:16}.
\end{remark}

\subsection{HHO, ncMFD and ncVEM are gradient discretisation methods}\label{sec:generality.ST}

In the previous sections, we showed that the consistent gradient in the HHO, ncMFD and ncVEM
methods (first formulation) is, for the Poisson problem, identical to the consistent gradient in \eqref{def:grad.rT}.
We show here that the stabilisations used in these methods can actually be represented
by well-chosen stabilisation terms $\ST$ satisfying \rprop{S1}--\rprop{S3}, and thus that these methods are gradient discretisation methods. 

Following the discussion in the previous sections, 
for the Poisson problem the HHO, ncMFD and ncVEM (first formulation) can be written,
upon an isomorphism of the space of discrete unknowns and with the proper choice of $l\in \{k-1,k,k+1\}$:
Find $\uu\in\Uhz$ such that, for all $\uv\in\Uhz$,
\begin{equation}\label{eq:hho.mfd.vem}
\sum_{T\in\Th}(\GRAD\rT\uu[T],\GRAD\rT\uv[T])_T
+\sum_{T\in\Th}\mathrm{Stab}_T(\uu[T]-\IT\rT\uu[T],\uv[T]-\IT\rT\uv[T])
=\sum_{T\in\Th} L_{T,f}(\uv[T]),
\end{equation}
where $L_{T,f}:\UT\to \Real$ is a specific linear form and
$\mathrm{Stab}_T$ is a symmetric bilinear form on $\UT$ that is coercive and stable
on ${\rm Im}({\rm Id}-\IT\rT)$, that is,
\begin{equation}\label{VEM.coer}
  \forall \uv[T]\in\UT\,,\;
  \mathrm{Stab}_T(\uv[T]-\IT\rT\uv[T],\uv[T]-\IT\rT\uv[T])\simeq \seminorm[2,\partial T]{\uv[T]}^2,
\end{equation}
where $a\simeq b$ means $C a\le b\le C^{-1} a$ with real number $C>0$ independent of $h$ and of $T$.

\begin{remark}[Stabilisation term]
For the HHO scheme \eqref{aT.hho}, by \eqref{eq:dT.dTF'} the vector 
$((\dTF[k]-\dT[l])\uv[T])_{F\in\Fh[T]}$ is the difference of the face- and cell-unknowns of $\uv[T]-\IT\rT\uv[T]$. Hence, the definition
\[
  {\rm Stab}_T(\uu[T]-\IT\rT\uu[T],\uv[T]-\IT\rT\uv[T])
  \coloneq
  \sum_{F\in\Fh[T]}h_F^{-1}((\dTF[k]-\dT[l])\uu[T],(\dTF[k]-\dT[l])\uv[T])_F
\]
is valid and ensures that the stabilisation terms in \eqref{aT.hho} and \eqref{eq:hho.mfd.vem}
coincide. This also shows that ${\rm Stab}_T(\uv[T]-\IT\rT\uv[T],\uv[T]-\IT\rT\uv[T])=\seminorm[2,\partial T]{\uv[T]}^2$, and thus that \eqref{VEM.coer} holds.

For the ncVEM version (any formulation), comparing the stabilisations
in \eqref{eq:ah.vem:first}--\eqref{eq:ah.vem:second} and in \eqref{eq:hho.mfd.vem} leads to defining,
for all $\uh,\vh\in \Vhk(T)$ and setting $\uu[T]=\IT\uh[h|T]$ and $\uv[T]=\IT\vh[h|T]$,
\[
	{\rm Stab}_T(\uu[T]-\IT\rT\uu[T],\uv[T]-\IT\rT\uv[T])
	\coloneq
	\mathcal{S}\big( ( {\rm Id}-\eproj{k+1} )\uh, ( {\rm Id}-\eproj{k+1} )\vh \big).
\]
Estimate \eqref{VEM.coer} then follows from \eqref{ncVEM:stab} and from \eqref{eq:ncVEM.norm.2}
in Lemma \ref{lem:ncVEM.norm}.
\end{remark}

Due to the orthogonality condition \rprop{S2},
the GS \eqref{lin.ell.gs} for Problem \eqref{lin.ell} with $\matr{\Lambda}=\matr{I}_d$,
based on the gradient reconstructions \eqref{def:grad.rT} (with $\STt=\ST$) and some
function reconstruction $\widetilde{\Pi}_{\GD[h]}$, is given by
\begin{equation}\label{eq:gs.hho.mfd.vem}
\sum_{T\in\Th}(\GRAD\rT\uu[T],\GRAD\rT\uv[T])_T
+\sum_{T\in\Th}(\ST\uu[T],\ST\uv[T])_T=\int_\Omega f \widetilde{\Pi}_{\GD[h]} \uv.
\end{equation}
For all methods except the ncMFD with $k=1$, accounting for~\eqref{eq:vT.l<0} when $k=0$ and $l=-1$, we have $L_{f,T}(\uv[T])=(f,v_T)_T$
and thus the choice $\widetilde{\Pi}_{\GD[h]}=\Pi_{\GD[h]}$ defined in \eqref{eq:DSGDM:XDz.PiD}
ensures that the right-hand sides of \eqref{eq:hho.mfd.vem} and \eqref{eq:gs.hho.mfd.vem} coincide.
For the ncMFD with $k=1$, a slightly different discretisation of the right-hand side has to be considered in order to ensure optimal $L^2$-error estimates under elliptic regularity; see \cite[Section 2.7]{Lipnikov-Manzini:2014} for further details.

To prove that \eqref{eq:hho.mfd.vem} can be written as \eqref{eq:gs.hho.mfd.vem},
it remains to show that for any stabilisation $\mathrm{Stab}_T$ as above, there exists
$\ST$ satisfying \rprop{S1}--\rprop{S3} and
such that
\begin{equation}\label{stab.equiv}
\forall \uu[T],\uv[T]\in \UT\,,\;(\ST\uu[T],\ST\uv[T])_T=\mathrm{Stab}_T(\uu[T]-\IT\rT\uu[T],\uv[T]-\IT\rT\uv[T]).
\end{equation}
Let us fix an initial $\ST[0]$ satisfying the design properties (for example,
the stabilisation defined by \eqref{eq:ST.RT}). The property \rprop{S1} on $\ST[0]$
and Lemma \ref{lem:norm.interp} in Appendix \ref{sec:equiv.seminorms} below show that $\ST[0]:\UT\to {\rm Im}(\ST[0])$ and $({\rm Id}-\IT\rT):\UT\to{\rm Im}({\rm Id}-\IT\rT)$ have the same kernel.
As shown by \eqref{VEM.coer} and Lemma \ref{lem:norm.interp}, ${\rm Stab}_T$ is an inner product on ${\rm Im}({\rm Id}-\IT\rT)$. Applying thus
\cite[Lemma A.3]{Droniou.Eymard.ea:10} produces an inner product $\langle\cdot,\cdot\rangle_T$ on ${\rm  Im}(\ST[0])$
such that 
\begin{equation}\label{equiv.stab.1}
\forall \uu[T],\uv[T]\in \UT\,,\;
\langle\ST[0]\uu[T],\ST[0]\uv[T]\rangle_T
=
\mathrm{Stab}_T(\uu[T]-\IT\rT\uu[T],\uv[T]-\IT\rT\uv[T]).
\end{equation}
Using then \cite[Lemma 5.2]{Droniou_et_al:13} with the inner products $\langle\cdot,\cdot\rangle_T$
and $(\cdot,\cdot)_T$ on ${\rm Im}(\ST[0])$, we find an isomorphism $\mathcal L_T$ of ${\rm Im}(\ST[0])$
such that
\begin{equation}\label{equiv.stab.2}
\forall \uu[T],\uv[T]\in \UT\,,\;
\langle\ST[0]\uu[T],\ST[0]\uv[T]\rangle_T=(\mathcal L_T\ST[0]\uu[T],\mathcal L_T\ST[0]\uv[T])_T.
\end{equation}
Combining \eqref{equiv.stab.1} and \eqref{equiv.stab.2} shows that \eqref{stab.equiv}
holds with
$$
\ST=\mathcal L_T\ST[0].
$$ 
The proof that $\ST$ defined above satisfies the design properties is easy.
The $L^2$-stability and boundedness \rprop{S1} is a direct consequence of
\eqref{stab.equiv} with $\uv[T]=\uu[T]$ and of \eqref{VEM.coer}.
The image of $\ST[0]$ is, by assumption, $L^2(T)^d$-orthogonal to $\Poly{k}(T)^d$ and contained in some $\Poly{k_S}(\mathcal{P}_T)$.
Since $\mathcal L_T$ is an isomorphism of ${\rm Im}(\ST[0])$, we have ${\rm Im}(\ST)={\rm Im}(\ST[0])$
and Properties \rprop{S2} and \rprop{S3} on $\ST$ therefore follow.

\begin{remark}[Formulations based on $\GT$]
With minor modifications, it is easy to construct a stabilisation $\ST$ that satisfies \textbf{($\widetilde{\bf S2}$)} instead of
\rprop{S2}.
This can be done at a reduced cost, as discussed in Section \ref{sec:alternate.grad}.
The reasoning above can also be easily adapted to the first formulation of ncVEM, provided
that the terms $\GRAD\rT$ in \eqref{eq:hho.mfd.vem} are replaced with $\GT$.
\end{remark}


\appendix

\section{Proofs of the results on DSGDs}\label{sec:proofs}

This section contains the proofs of Theorem~\ref{thm:DSGDM} and Proposition~\ref{prop:est.SD.WD} preceded by some preliminary results: the study of the properties of stabilising contributions satisfying \rprop{S1}--\rprop{S3} and uniform equivalences of discrete $W^{1,p}$-seminorms.
We also include lemmas used in Section \ref{sec:links} to show that the HHO method, ncMFD method and ncVEM are GDMs.

\subsection{Properties of the stabilising contribution}

\begin{proposition}[Properties of $\ST$]\label{prop:ST}
  Let $\{\ST\st T\in\Th\}$ be a family of stabilising contributions satisfying assumptions \rprop{S1}--\rprop{S3}.
  Then, the following properties hold:
  \begin{enumerate}[(i)]
  \item \emph{$L^p$-stability and boundedness.} For all $T\in\Th$ and all $\uv[T]\in\UT$,
    \begin{equation}\label{eq:ST:stability:p}
      \norm[L^p(T)^d]{\ST\uv[T]}\simeq\seminorm[p,\partial T]{\uv[T]},
    \end{equation}
    with hidden constant as in~\eqref{eq:ST:stability} and additionally depending on $p$ and $k_{\rm S}$.
  \item \emph{Consistency.} For all $T\in\Th$ and all $v\in W^{k+2,p}(T)$,
  \begin{equation}\label{eq:ST:consistency}
    \norm[L^p(T)^d]{\ST\IT v}\lesssim h_T^{k+1}\seminorm[W^{k+2,p}(T)]{v},
  \end{equation}
  where $a\lesssim b$ means $a\le Cb$ with real number $C>0$ independent of both $h$ and $T$, but possibly depending on $d$, $\varrho$, $k$, $l$, $p$ and $k_{\rm S}$.
  \end{enumerate}
  As a consequence of~\eqref{eq:ST:consistency}, if $v\in \Poly{k+1}(T)$, then $\ST\IT v=0$.
\end{proposition}
In the proof, we will need the following direct and reverse Lebesgue embeddings, proved in~\cite[Lemma~5.1]{Di-Pietro.Droniou:16}:
  Let $X$ denote a measurable subset of $\Real^d$ with inradius $r_X$ and diameter $h_X$, and let two reals $r,s\in[1,+\infty]$ and an integer $\ell\in\Natural$ be fixed.
  Then, for all $q\in\Poly{\ell}(X)$, it holds that
  \begin{equation}\label{eq:lebesgue.embeddings}
    \norm[L^r(X)]{q}\simeq |X|^{\frac1r-\frac1s}\norm[L^s(X)]{q},
  \end{equation}
  where $a\simeq b$ means $C a\le b\le C^{-1} a$ with real number $C>0$ only depending on $d$, a lower bound of the ratio $\frac{r_X}{h_X}$, $r$, $s$, and $\ell$.

We will also need the following $L^p$-trace inequality (see, e.g.,~\cite[Eq. (A.10)]{Di-Pietro.Droniou:16}):
For all $T\in\Th$ and all $v\in W^{1,p}(T)$,
\begin{equation}\label{eq:Lp.trace}
  h_T^{\frac1p}\norm[L^p(\partial T)]{v}\lesssim
  \norm[L^p(T)]{v} + h_T\norm[L^p(T)^d]{\GRAD v},
\end{equation}
where $a\lesssim b$ means $a\le Cb$ with real number $C>0$ independent of $h$ and of $T$, but possibly depending on $d$, $p$, and $\varrho$.
When $v\in\Poly{\ell}(T)$ for some integer $\ell\ge 0$, combining~\eqref{eq:Lp.trace} with the following inverse inequality (see, e.g.,~\cite[Remark~A.2]{Di-Pietro.Droniou:16}):
\begin{equation}\label{eq:Lp.inverse}
  \norm[L^p(T)^d]{\GRAD v}\lesssim h_T^{-1}\norm[L^p(T)]{v},
\end{equation} yields
\begin{equation}\label{eq:Lp.trace.discrete}
  h_T^{\frac1p}\norm[L^p(\partial T)]{v}\lesssim
  \norm[L^p(T)]{v},
\end{equation}
where the hidden multiplicative constant in~\eqref{eq:Lp.inverse} and~\eqref{eq:Lp.trace.discrete} can additionally depend on $\ell$.

We are now ready to prove Proposition~\ref{prop:ST}.
\begin{proof}[Proof of Proposition~\ref{prop:ST}]
  (i) \emph{$L^p$-stability and boundedness.}
  When $p=2$,~\eqref{eq:ST:stability:p} coincides with~\eqref{eq:ST:stability}.
  Let us now consider the case $p\neq 2$ and recall the following inequalities valid for all integers $n\ge 1$ and all reals $q\in[1,+\infty)$, $a_i\ge 0$ ($1\le i\le n$):
  \begin{equation}\label{eq:magic}
    \sum_{i=1}^n a_i^q\le\left(\sum_{i=1}^na_i\right)^q\le n^{q-1}\sum_{i=1}^n a_i^q,
    \qquad
    n^{\frac{1-q}{q}}\sum_{i=1}^n a_i^{\frac1q}\le\left(\sum_{i=1}^na_i\right)^{\frac1q}\le \sum_{i=1}^n a_i^{\frac1q}.
  \end{equation}

  We also notice that, owing to~\eqref{eq:S3.cond}, it holds for all $T\in\Th$ that
  \begin{equation}\label{eq:ST:stability:p:1}
    |P|\simeq |T|\quad\forall P\in\mathcal{P}_T,\qquad
    \card(\mathcal{P}_T)\simeq 1.
  \end{equation}
  To prove $|P|\simeq |T|$, it suffices to observe that $|P|\simeq h_P^d\simeq h_T^d\simeq |T|$, where we have used, respectively, the first and second conditions in~\eqref{eq:S3.cond} and the mesh regularity to conclude.
  The bound on $\card(\mathcal{P}_T)$ follows by writing $|T|=\sum_{P\in\mathcal{P}_T} |P|\simeq \sum_{P\in\mathcal{P}_T}|T|=\card(\mathcal{P}_T)|T|$.
  
  Let now $T\in\Th$ and $\uv[T]\in\UT$ be fixed.
  Since $\ST\uv[T]\in \Poly{k_S}(P)$ for all $P\in\mathcal P_T$, we have that
  $$
  \begin{aligned}
    \norm[L^p(T)^d]{\ST\uv[T]}^p
    &=\sum_{P\in\mathcal{P}_T}\norm[L^p(P)^d]{\ST\uv[T]}^p
    \\
    &\simeq\sum_{P\in\mathcal{P}_T}|P|^{p\left(\frac1p-\frac12\right)}\norm[L^2(P)^d]{\ST\uv[T]}^p
    &\qquad&\text{Eqs.~\eqref{eq:S3.cond} and~\eqref{eq:lebesgue.embeddings}}
    \\
    &\simeq|T|^{p\left(\frac1p-\frac12\right)}\sum_{P\in\mathcal{P}_T}\norm[L^2(P)^d]{\ST\uv[T]}^p
    &\qquad&\text{Eq.~\eqref{eq:ST:stability:p:1}}.
  \end{aligned}
  $$
  In the second line, condition~\eqref{eq:S3.cond} is invoked to use $\varrho$ as a lower bound for $\frac{r_P}{h_P}$ when applying~\eqref{eq:lebesgue.embeddings}.
  Using the first pair of inequalities in~\eqref{eq:magic} with $q=\frac{p}2$ if $p\ge 2$, the second pair of inequalities in~\eqref{eq:magic} with $q=\frac2p$ if $p<2$ and, in both cases, $a_i=\norm[L^2(P)^d]{\ST\uv[T]}^2$ and $n=\card(\mathcal{P}_T)\simeq 1$, we infer
  $$
  \norm[L^p(T)^d]{\ST\uv[T]}^p
  \simeq|T|^{p\left(\frac1p-\frac12\right)}\left(\sum_{P\in\mathcal{P}_T}\norm[L^2(P)^d]{\ST\uv[T]}^2\right)^{\frac{p}{2}}
  = |T|^{p\left(\frac1p-\frac12\right)}\norm[L^2(T)^d]{\ST\uv[T]}^p.
  $$
  Taking the $p$th root of the above relation, we arrive at
  \begin{equation}\label{eq:ST:stability:p:2}
    \norm[L^p(T)^d]{\ST\uv[T]}\simeq |T|^{\frac1p-\frac12}\norm[L^2(T)^d]{\ST\uv[T]}.
  \end{equation}
  
  Proceeding similarly using~\eqref{eq:lebesgue.embeddings} repeatedly on the faces of $T$ and using $|F|h_F\simeq |T|$, we can prove that
  \begin{equation}\label{eq:ST:stability:p:3}
    \seminorm[2,\partial T]{\uv[T]}\simeq |T|^{\frac12-\frac1p}\seminorm[p,\partial T]{\uv[T]}.
  \end{equation}  
  Combining~\eqref{eq:ST:stability:p:2} and~\eqref{eq:ST:stability:p:3} with~\eqref{eq:ST:stability},~\eqref{eq:ST:stability:p} follows.

  (ii) \emph{Consistency.}
  Let a mesh element $T\in\Th$ be fixed and set, for the sake of brevity, $\huv\coloneq\IT v$.
  Using the uniform equivalence~\eqref{eq:ST:stability:p} proved in the first point, and recalling the definition~\eqref{eq:norm1p.T} of the $\seminorm[p,\partial T]{{\cdot}}$-seminorm, it is inferred that
  $$
  \norm[L^p(T)^d]{\ST\huv}^p
  \lesssim\seminorm[p,\partial T]{\huv}^p
  = \sum_{F\in\Fh[T]} h_F^{1-p}\norm[L^p(F)]{(\dTF-\dT)\huv}^p.
  $$
  For all $F\in\Fh[T]$, using the triangle inequality followed by the discrete trace inequality~\eqref{eq:Lp.trace.discrete}, we infer that
  $$
  \norm[L^p(F)]{(\dTF-\dT)\huv}^p\lesssim\norm[L^p(F)]{\dTF\huv}^p + \norm[L^p(F)]{\dT\huv}^p
  \lesssim \norm[L^p(F)]{\dTF\huv}^p + h_T^{-1}\norm[L^p(T)]{\dT\huv}^p.
  $$
  Using the above inequality together with the uniform bound on the number of faces of $T$ (see~\cite[Lemma~1.41]{Di-Pietro.Ern:12}) for the second term, and expanding the difference operators according to their definitions~\eqref{eq:dT.dTF}, we arrive at
  $$
  \norm[L^p(T)^d]{\ST\huv}^p
  \lesssim h_T^{-p}\norm[L^p(T)]{\lproj[T]{l}(\rT\huv - v)}^p
  + \sum_{F\in\Fh[T]} h_F^{1-p}\norm[L^p(F)]{\lproj[F]{k}(\rT\huv - v)}^p.
  $$
  We notice that the projectors in the above bound can be removed invoking the $L^p(T)$-boundedness of $\lproj[T]{l}$ for the first term and the $L^p(F)$-boundedness of $\lproj[F]{k}$ for the second (see~\cite[Lemma 3.2]{Di-Pietro.Droniou:16}).
  Combining this observation with the optimal approximation properties of $\rT\circ\IT$ discussed in Remark~\ref{rem:approx:rT.IT}, we then conclude that
  $$
  \norm[L^p(T)^d]{\ST\huv}^p
  \lesssim h_T^{-p}\norm[L^p(T)]{\rT\huv - v}^p
  + \sum_{F\in\Fh[T]} h_F^{1-p}\norm[L^p(F)]{\rT\huv - v}^p
  \lesssim
  h_T^{p(k+1)}\seminorm[W^{k+2,p}(T)]{v}^p.\qedhere
  $$
\end{proof}

\subsection{Uniform equivalence of discrete $W^{1,p}$-seminorms}\label{sec:equiv.seminorms}

The second preliminary result is the uniform equivalence of various $W^{1,p}$-seminorms on the global space of discrete unknowns $\Uh$.

\begin{proposition}[Uniform equivalence of discrete $W^{1,p}$-seminorms]\label{prop:norm.equiv}
  Define the discrete semi\-norm $\tnorm[1,p,h]{{\cdot}}$ such that, for all $\uv\in\Uh$,
  \begin{equation}\label{eq:tnorm}
	\begin{aligned}
    &\tnorm[1,p,h]{\uv}^p\coloneq\sum_{T\in\Th}\tnorm[1,p,T]{\uv[T]}^p,\\
    &\tnorm[1,p,T]{\uv[T]}^p\coloneq\norm[L^p(T)^d]{\GRAD v_T}^p
    + \sum_{F\in\Fh[T]} h_F^{1-p}\norm[L^p(F)]{v_F-v_T}^p\quad\forall T\in\Th.
	\end{aligned}
  \end{equation}
  Then, $\tnorm[1,p,h]{{\cdot}}$ is a norm on the subspace $\Uhz$ and, denoting by $\{\ST\st T\in\Th\}$ a family of stabilising contributions that satisfy \rprop{S1}--\rprop{S3} and defining $\grD[h]$ by \eqref{eq:DSGDM:grD}, it holds for all $\uv[h]\in\Uh$ and all $T\in\Th$,
  \begin{equation}\label{eq:norm.equiv.local}
    \tnorm[1,p,T]{\uv[T]}
    \simeq\norm[1,p,T]{\uv[T]}
    \simeq\norm[L^p(T)^d]{\grD[h]\uv},
  \end{equation}
  where $a\simeq b$ means $C a\le b\le C^{-1} a$ with real number $C>0$ independent of $T$ and $h$, but possibly depending on $d$, $p$, $\varrho$, $k$, $l$, and $k_{\rm S}$.
  As a consequence, for all $\uv[h]\in\Uh$,
  \begin{equation}\label{eq:norm.equiv}
    \tnorm[1,p,h]{\uv}
    \simeq\norm[1,p,h]{\uv}
    \simeq\norm[L^p(\Omega)^d]{\grD[h]\uv}.
  \end{equation}
\end{proposition}
\begin{proof}
  The fact that $\tnorm[1,p,h]{{\cdot}}$ is a norm on $\Uhz$ can be proved in a similar manner as for the case $p=2$ and $l=k$ considered in~\cite[Proposition~5]{Di-Pietro.Ern:15}.
  The local seminorm equivalence $\tnorm[1,p,T]{\uv}\simeq\norm[1,p,T]{\uv}$ valid for all $T\in\Th$, on the other hand, is proved in~\cite[Lemma~5.2]{Di-Pietro.Droniou:16} (see also references therein) for the case $l=k$, and the same reasoning extends to $l=k-1$ and $l=k+1$.
 
  Let now a mesh element $T\in\Th$ be fixed.
  We have that
  \begin{equation}\label{eq:norm.equiv:1}
    \begin{aligned}
      \norm[L^p(T)^d]{\grD[h]\uv}=
      \norm[L^p(T)^d]{\grT\uv[T]}
      &\simeq |T|^{\frac1p-\frac12}\norm[L^2(T)^d]{\grT\uv[T]}
      \\
      &= |T|^{\frac1p-\frac12}\left( \norm[L^2(T)^d]{\GT\uv[T]}^2 + \norm[L^2(T)^d]{\ST\uv[T]}^2\right)^{\frac12}
      \\
      &\simeq |T|^{\frac1p-\frac12} \left(\norm[L^2(T)^d]{\GT\uv[T]} + \seminorm[2,\partial T]{\uv[T]}\right)
      \\
      &\simeq \left( \norm[L^p(T)^d]{\GT\uv[T]} + \seminorm[p,\partial T]{\uv[T]}\right)
      \simeq\norm[1,p,T]{\uv[T]},
    \end{aligned}
  \end{equation}
  where we have used a reasoning similar to the one leading to \eqref{eq:ST:stability:p:2} in the first line (recall that $\grT\uv[T]$ is piecewise polynomial on $T$ owing to \rprop{S3}), the orthogonality property \rprop{S2} in the second line, the stability and boundedness property \rprop{S1} in the third line, and \eqref{eq:ST:stability:p:3} together with the discrete Lebesgue embeddings \eqref{eq:lebesgue.embeddings} in the last line.
This concludes the proof of \eqref{eq:norm.equiv.local}. The global version
\eqref{eq:norm.equiv} follows by raising \eqref{eq:norm.equiv.local} to the power $p$ and summing over $T\in\Th$.
\end{proof}

The following lemma, which justifies the importance of the seminorm $\seminorm[p,\partial T]{\cdot}$,
was used in Section \ref{sec:generality.ST} to prove that HHO, ncMFD and ncVEM are GDM.

\begin{lemma}\label{lem:norm.interp}
For any $T\in\Th$ and any $\uv[T]\in \UT$, it holds that
\begin{equation}\label{equiv.norm.residuals}
\tnorm[1,p,T]{\uv[T]-\IT\rT\uv[T]}\simeq \seminorm[p,\partial T]{\uv[T]},
\end{equation}
where $a\simeq b$ means $C^{-1} a\le b\le C a$ with real number $C>0$ depending only on $d$, $\varrho$, $p$, $k$, and $l$.
As a consequence, $\uv[T]-\IT\rT\uv[T]=0$ if and only if $\seminorm[p,\partial T]{\uv[T]}=0$.
\end{lemma}

\begin{proof} Here, $a\lesssim b$ means that $a\le Cb$ for some real number $C>0$ as in the statement.
Using direct and inverse Lebesgue inequalities as in the proof of Proposition \ref{prop:ST},
we deduce from \eqref{control.volumic.alt} that $\norm[L^p(T)^d]{\GRAD\dT[l]\uv[T]}\lesssim\seminorm[p,\partial T]{\uv[T]}$.
Hence, using the relation \eqref{eq:dT.dTF'} together with the definitions \eqref{eq:tnorm} of $\tnorm[1,p,T]{{\cdot}}$ and \eqref{eq:norm1p.T} of $\seminorm[p,\partial T]{\uv[T]}$, we obtain
\[
\tnorm[1,p,T]{\uv[T]-\IT\rT\uv[T]}^p
=\norm[L^p(T)^d]{\GRAD \dT[l]\uv[T]}^p
+\sum_{F\in\Fh[T]}h_F^{1-p}\norm[L^p(F)]{(\dTF[k]-\dT[l])\uv[T]}^p
\simeq\seminorm[p,\partial T]{\uv[T]}^p,
\]
which is \eqref{equiv.norm.residuals}.
If $\uv[T]-\IT\rT\uv[T]=0$, the relation above shows that $\seminorm[p,\partial T]{\uv[T]}=0$.
Conversely, if $\seminorm[p,\partial T]{\uv[T]}=0$ then, letting
$\uw[T]\coloneq\uv[T]-\IT\rT\uv[T]$, we have that $\tnorm[1,p,T]{\uw[T]}=0$, which implies in turn that $w_T$ is constant equal to $c$ and that $w_F=w_T=c$ for all $F\in\Fh[T]$. 
Then, $\rT\uw[T]=\rT\IT c=\eproj[T]{k+1}c=c$ (see Remark
\ref{rem:approx:rT.IT} and additionally observe, for $k=0$ and $l=-1$, that
$w_T=c$ owing to the choice of the weights $\omega_{TF}$).
We then have that
$$
c=\rT\uw[T]
=\rT(\uv[T]-\IT\rT\uv[T])
= \rT\uv[T]-\rT\IT\rT\uv[T]
= \rT\uv[T]-\rT\uv[T]
=0,
$$
where we have used the definition of $\uw[T]$ in the second equality, the
linearity of $\rT$ in the third, the fact that $\rT\IT\rT=\rT$ in the fourth
(since $\rT\IT=\eproj[T]{k+1}$ preserves polynomials up to degree $\le k+1$).
Thus, the constant $c$ is 0, which shows that $\uv[T]-\IT\rT\uv[T]=0$.
\end{proof}

The last lemma of this section was used in Section \ref{sec:generality.ST} to analyse the ncVEM stabilisation in the context of the GDM.

\begin{lemma}\label{lem:ncVEM.norm}
For all $T\in\Th$ and all $\vh\in \Vhk(T)$ (with $\Vhk(T)$ defined by \eqref{eq:VhkT}), it holds that
\begin{equation}\label{eq:ncVEM.norm}
\norm[L^2(T)^d]{\GRAD\vh} \simeq \tnorm[1,2,T]{\IT[k,k-1]\vh},
\end{equation}
where $\tnorm[1,2,T]{\cdot}$ is defined in \eqref{eq:tnorm} and $a\simeq b$ means $C^{-1} a\le b\le Ca$ with real number $C>0$ independent on $h$, but possibly depending on $d$, $k$, $l$ and $\varrho$.
As a consequence,
\begin{equation}\label{eq:ncVEM.norm.2}
\text{for all $\vh\in\Vhk(T)$ such that $\eproj[T]{k}\vh=0$, $\norm[L^2(T)^d]{\GRAD\vh}\simeq\seminorm[1,2,T]{\IT[k,k-1]\vh}$.}
\end{equation}
\end{lemma}

\begin{proof} Here, $a\lesssim b$ means $a\le Cb$ with $C$ as in the statement.
Since $\GRAD\vh\SCAL\normal_{TF}\in\Poly{k}(F)$ for all $F\in\Fh[T]$ and $\Delta \vh\in\Poly{k-1}(T)$, integrating by parts and setting $\uv[T]\coloneq\IT[k,k-1]\vh$,
\begin{align*}
\norm[L^2(T)^d]{\GRAD\vh}^2={}&-(\vh,\Delta \vh)_T + \sum_{F\in\Fh[T]}(\vh,\GRAD\vh\SCAL\normal_{TF})_F\\
={}&(-v_T,\Delta \vh)_T+\sum_{F\in\Fh[T]}(v_F,\GRAD\vh\SCAL\normal_{TF})_F\\
={}&(\GRAD v_T,\GRAD \vh)_T+\sum_{F\in\Fh[T]}(v_F-v_T,\GRAD\vh\SCAL\normal_{TF})_F.
\end{align*}
The Cauchy--Schwarz inequality and the trace inequality \eqref{eq:Lp.trace.discrete} applied on $\GRAD\vh\SCAL\normal_{TF}$ then yield $\norm[L^2(T)^d]{\GRAD\vh}\lesssim \tnorm[1,2,T]{\uv[T]}$,
which is half of \eqref{eq:ncVEM.norm}.

To prove the second half for $k\ge 1$, recall first that $v_T=\lproj[T]{k-1}\vh$. A
triangle inequality and \eqref{eq:approx} with $\ell=k-1$, $p=2$, $\alpha=0$ and $r=s=1$ thus show that
\begin{equation}\label{eq:norm.ncVEM.a}
\norm[L^2(T)^d]{\GRAD v_T}\lesssim \norm[L^2(T)^d]{\GRAD \vh}.
\end{equation}
Then, write
\[
\norm[L^2(F)]{v_F-v_T}=\norm[L^2(F)]{\lproj[F]{k}(\vh-\lproj[T]{k-1}\vh)}
\le \norm[L^2(F)]{\vh-\lproj[T]{k-1}\vh}\le h_T^{\frac{1}{2}}\norm[L^2(T)^d]{\GRAD\vh},
\]
where \eqref{eq:approx.trace} was used with $\ell=k-1$, $p=2$, $\alpha=0$, $r=0$ and $s=1$. Combining this
estimate with \eqref{eq:norm.ncVEM.a} concludes the proof of \eqref{eq:ncVEM.norm}.
The above argument can be extended along similar lines to the case $k=0$, the only variation being linked to the fact that, in this case, $v_T$ is not an $L^{2}$-orthogonal projection of $\mathfrak{v}_h$ (see \eqref{eq:vT.l<0}).

To prove \eqref{eq:ncVEM.norm.2}, notice that if $\eproj[T]{k+1}\vh=0$ then
$\rT\uv[T]=\rT\IT\vh=\eproj[T]{k+1}\vh=0$. Hence, \eqref{eq:ncVEM.norm} yields
\[
\norm[L^2(T)^d]{\GRAD\vh}\simeq \tnorm[1,2,T]{\uv[T]-\IT\rT\uv[T]},
\]
and the proof is complete by invoking \eqref{equiv.norm.residuals}.
\end{proof}

\subsection{Main results}

We are now ready to prove the main results stated in Section~\ref{sec:main.results}.

\subsubsection{Properties of Discontinuous Skeletal Gradient Discretisations}\label{sec:proof:DSGDM}

\begin{proof}[Proof of Theorem \ref{thm:DSGDM}]
  We use a polytopal toolbox in the spirit of~\cite{Droniou-Eymard-Herbin:2015} and~\cite[Section~7.2]{gdm}.
  Let
  \[
  X_{\Th,0}\coloneq\left\{\uw=\left((w_T)_{T\in\Th},(w_F)_{F\in\Fh}\right)\st w_T\in\Real\,,\;w_F\in\Real\,,\;w_F=0\ \ \forall F\in\Fhb\right\}
  =\Uhz[0,0]
\]
and $\Pi_{\Th}:X_{\Th,0}\to L^p(\Omega)$, $\GRAD_{\Th}:X_{\Th,0}\to L^p(\Omega)^d$ be defined by,
for all $\uw\in X_{\Th,0}$ and all $T\in\Th$,
\[
(\Pi_{\Th}\uw)_{|T}=w_T\quad\mbox{ and }\quad (\GRAD_{\Th}\uw)_{|T}=\frac{1}{|T|}\sum_{F\in\Fh[T]}|F|w_F\mathbf{n}_{TF}.
\]
By \cite[Corollary 7.12]{gdm}, the coercivity, limit-conformity, consistency, and compactness of $(\GD[h])_{h\in\mathcal H}$ follow if we find
a mapping $\control:\Uhz\to X_{\Th,0}$ (``control'' of $\GD[h]$) such that, recalling the definition \eqref{eq:tnorm}
of $\tnorm[1,p,h]{{\cdot}}$ and setting
\begin{subequations}\label{eq:control}
  \begin{align}
    \norm[{\GD[h],\Th}]{\control}&\coloneq\max_{\uv\in \Uhz\setminus\{\underline{0}_h\}}\frac{\tnorm[1,p,h]{\control(\uv)}}{\norm[L^p(\Omega)^d]{\grD[h]\uv}}\,,\label{eq:control:1} \\
    \omega^{\Pi}(\GD[h],\Th,\control)&\coloneq\max_{\uv\in\Uhz\setminus\{\underline{0}_h\}}\frac{\norm[L^p(\Omega)]{\PiD[h]\uv - \Pi_{\Th}\control(\uv)}}{\norm[L^p(\Omega)^d]{\grD[h]\uv}}\,,\\
    \omega^{\GRAD}(\GD[h],\Th,\control)&\coloneq\max_{\uv\in\Uhz\setminus\{\underline{0}_h\}}\frac{\left(\sum_{T\in\Th}|T|^{1-p}\left|\int_T (\grD[h]\uv - \GRAD_{\Th}\control(\uv))\right|^p\right)^{\frac{1}{p}}}{\norm[L^p(\Omega)^d]{\grD[h]\uv}}\,,\;
  \end{align}
\end{subequations}
we have
\begin{equation}\label{prop.control}
\norm[{\GD[h],\Th}]{\control}\lesssim 1\mbox{ and, as $h\to 0$}\,,\;
\omega^\Pi(\GD[h],\Th,\control)\to 0\mbox{ and }
\omega^\GRAD(\GD[h],\Th,\control)\to 0,
\end{equation}
where $a\lesssim b$ means $a\le C b$ with real number $C>0$ independent of $h$, but possibly depending on $d$, $\varrho$, $k$, $l$, and $k_{\rm S}$.

Let $\control$ be defined the following way.
For all $\uv=\left((v_T)_{T\in\Th},(v_F)_{F\in\Fh}\right)\in\Uh$, we let $\control(\uv)\coloneq\uvm=\left((v_T^0)_{T\in\Th},(v_F^0)_{F\in\Fh}\right)\in X_{\Th,0}$ be such that $v_T^0=\lproj[T]{0}v_T$ for all $T\in\Th$ and $v_F^0=\lproj[F]{0}v_F$ for all $F\in\Fh$. Properties \eqref{prop.control} follow if we establish that, for all $\uv\in\Uhz$,
\begin{subequations}\label{eq:PTB}
  \begin{align}
    \tnorm[1,p,h]{\uvm}
    &\lesssim \norm[L^p(\Omega)^d]{\grD[h]\uv},\label{eq:PTB:1}
    \\
    \norm[L^p(T)]{\PiD[h]\uv-\Pi_{\Th}\uvm}&\lesssim h\norm[L^p(\Omega)^d]{\grD[h]\uv},\label{eq:PTB:2} 
  \end{align}
  and, for all $T\in\Th$ and all $\uv[T]\in\UT$,
  \begin{equation}
    (\grT\uv[T],\vec{\eta})_T= \sum_{F\in\Fh[T]} (v_F^0,\vec{\eta}\SCAL\normal_{TF})_F
    \qquad\forall\vec{\eta}\in\Poly{0}(T)^d.\label{eq:PTB:3}
  \end{equation}
\end{subequations}
Indeed, \eqref{eq:PTB:1} gives a bound on $\norm[{\GD[h],\Th}]{\control}$,
\eqref{eq:PTB:2} gives an $\mathcal O(h)$ estimate on $\omega^\Pi(\GD[h],\Th,\control)$, and
\eqref{eq:PTB:3} shows that $\omega^\GRAD(\GD[h],\Th,\control)=0$.

(i) \emph{Proof of~\eqref{eq:PTB:1}.}
By the definition~\eqref{eq:tnorm} of the $\tnorm[1,p,h]{{\cdot}}$-seminorm along with that of $\uvm$, we have that
\begin{equation}\label{eq:DSGDM:0}
  \tnorm[1,p,h]{\uvm}^p=\sum_{T\in\Th}\sum_{F\in\Fh[T]}h_F^{1-p}\norm[L^p(F)]{v_F^0 - v_T^0}^p.
\end{equation}
Let now a mesh element $T\in\Th$ be fixed and observe that, for all $F\in\Fh[T]$,
\begin{align}
  \norm[L^p(F)]{v_F^0-v_T^0}
  = \norm[L^p(F)]{\lproj[F]{0}(v_F-v_T^0)}
  &\le\norm[L^p(F)]{v_F-v_T^0}\nonumber
  \\
  &\le \norm[L^p(F)]{v_F-v_T}
  + \norm[L^p(F)]{v_T-v_T^0}\nonumber
  \\
  &\lesssim \norm[L^p(F)]{v_F-v_T} + h_T^{-\frac1p}\norm[L^p(T)]{v_T-v_T^0}\nonumber
  \\
  &\lesssim \norm[L^p(F)]{v_F-v_T} + h_T^{1-\frac1p}\norm[L^p(T)^d]{\GRAD v_T},
\label{proof:coer:1}
\end{align}
where we have used the $L^p$-boundedness of the $L^2$-orthogonal projector (see~\cite[Lemma~3.2]{Di-Pietro.Droniou:16}) in the first line, the triangle inequality in the second line, the discrete $L^p$-trace inequality~\eqref{eq:Lp.trace.discrete} in the third line, and a local Poincar\'{e}--Wirtinger inequality which can be inferred from~\eqref{eq:approx} with $\alpha=\ell=r=0$ and $s=1$ to conclude.

Taking the $p$th power of \eqref{proof:coer:1}, multiplying by $h_F^{1-p}\simeq h_T^{1-p}$ and summing over $F\in \Fh[T]$ and $T\in\Th$ leads to $\tnorm[1,p,h]{\uvm}\lesssim \tnorm[1,p,h]{\uv}$.
Estimate \eqref{eq:PTB:1} follows by using \eqref{eq:norm.equiv}.

  (ii) \emph{Proof of~\eqref{eq:PTB:2}.} Using a local Poincar\'{e}--Wirtinger inequality as above we infer, for all $T\in\Th$, $\norm[L^p(T)]{v_T-v_T^0}\lesssim h_T\norm[L^p(T)^d]{\GRAD v_T}$. Taking the $p$th power of this inequality, summing over $T\in\Th$, and using $h_T\le h$ and the uniform norm equivalence~\eqref{eq:norm.equiv} to bound the right-hand side,~\eqref{eq:PTB:2} follows.

  (iii) \emph{Proof of~\eqref{eq:PTB:3}.} Let $\vec{\eta}\in\Poly{0}(T)^d\subset\Poly{k}(T)^d$.
  Since $\DIV\vec{\eta}=0$, using the orthogonality property \rprop{S2} followed by the definition \eqref{eq:GT} of $\GT$ we infer that
$(\grT\uv[T],\vec{\eta})_T= \sum_{F\in\Fh[T]}(v_F,\vec{\eta}\SCAL\normal_{TF})_F$.
Equation~\eqref{eq:PTB:3} then follows by noticing that, $\vec{\eta}\SCAL\normal_{TF}$ being constant
on $F$, $(v_F,\vec{\eta}\SCAL\normal_{TF})_F=(\lproj[F]{0}v_F,\vec{\eta}\SCAL\normal_{TF})_F
=(v_F^0,\vec{\eta}\SCAL\normal_{TF})_F$.

\medskip

The GD-consistency follows from Proposition~\ref{prop:est.SD.WD} (proved below), and from
\cite[Lemma 2.17]{gdm} which shows that the consistency holds provided that $\SD[h](\phi)\to 0$
for all $\phi$ in a dense subset of $W^{1,p}_0(\Omega)$.
\end{proof}

\begin{remark}[Condition~\eqref{eq:PTB:1}]
  In~\cite{gdm}, a slightly different norm is considered in the argument of the maximum in~\eqref{eq:control:1}. The original expression is obtained replacing $\tnorm[1,p,h]{\uv^0}^p$ by
  $$
  \sum_{T\in\Th}\sum_{F\in\Fh[T]} d_{TF}^{1-p} \norm[L^p(F)]{v_F^0 - v_T^0}^p,
  $$
  where the only difference with respect to~\eqref{eq:DSGDM:0} is that the role of the local length scale is played by $d_{TF}$, the orthogonal distance between a point $\vec{x}_T$ inside $T$ and the face $F$, instead of $h_F$.
  If, for all $T\in\Th$, we choose the point $\vec{x}_T$ such that condition~\eqref{eq:dTF.hT} is verified, it can easily be proved that $d_{TF}\simeq h_F$, and the two norms are uniformly equivalent.
\end{remark}

\subsubsection{Estimates on $\SD$ and $\WD$}\label{sec:est.SD.WD}

\begin{proof}[Proof of Proposition~\ref{prop:est.SD.WD}]
(i) \emph{Estimates on the addends in $\SD[h]$.} Take $\phi\in W^{1,p}_0(\Omega)\cap W^{l+1,p}(\Th)$ and
let $\uv[h]=\Ih\phi\in \Uh$. For all $T\in\Th$, if $l\ge 0$ then $(\PiD[h]\uv[h])_{|T}=v_T=\lproj[T]{l}\phi$ on $T$
so the approximation estimate \eqref{eq:approx} applied with $\alpha=0$, $\ell=l$, $s=\ell+1$ and $r=0$ yields
\begin{equation}\label{est:SD.1:l>=0}
\norm[L^p(T)]{\PiD[h]\uv[h]-\phi}\lesssim h_T^{l+1}\seminorm[W^{l+1,p}(T)]{\phi}.
\end{equation}
On the other hand, if $l=-1$, the specific choice~\eqref{eq:vT.l<0} of $v_T$ yields
  \begin{equation}\label{est:SD.1:l<0}
    \norm[L^p(T)]{\PiD[h]\uv[h]-\phi}\lesssim h_T\seminorm[W^{1,p}(T)]{\phi}.
  \end{equation}
  Combining~\eqref{est:SD.1:l>=0} and~\eqref{est:SD.1:l<0}, taking the $p$th power, summing over $T\in\Th$, and taking the $p$th root of the resulting inequality gives~\eqref{est:terms.SD:Lp}.

For a fixed mesh element $T\in\Th$, use the definition \eqref{eq:DSGDM:grD} of $\grD[h]$, the 
commutativity property \eqref{GT:commutativity} of $\GT$, the approximation property
\eqref{eq:approx} of $\lproj[T]{k}$ with $s=k+1$ and $r=0$, and the consistency
\eqref{eq:ST:consistency} of $\ST$ to obtain
$$
\begin{aligned}
  \norm[L^p(T)^d]{\grD[h]\uv[h]-\GRAD\phi}\le{}&
  \norm[L^p(T)^d]{\GT\uv[T]-\GRAD\phi}+\norm[L^p(T)^d]{\ST\uv[T]}\\
  \lesssim{}& h_T^{k+1}\seminorm[W^{k+1,p}(T)^d]{\GRAD\phi}+h_T^{k+1}\seminorm[W^{k+2,p}(T)]{\phi}
  \lesssim h_T^{k+1}\seminorm[W^{k+2,p}(T)]{\phi}.
\end{aligned}
$$
The estimate~\eqref{est:terms.SD:W1p} follows taking the $p$th power, summing over $T\in\Th$, and taking the $p$th root of the resulting inequality.

If $l\in\{k,k+1\}$ or $l=-1$ (in which case $k=0$), the estimate \eqref{est:SD} on $\SD[h](\phi)$ is an immediate consequence of~\eqref{est:terms.SD}.
Consider now $l=k-1$ and $k\ge 1$. An easy modification of the proof above shows that, for all $\phi\in W^{1,p}_0(\Omega)\cap
W^{k+1,p}(\Th)$, $\norm[L^p(\Omega)^d]{\grD[h]\uv[h]-\GRAD\phi}\lesssim h^{k}\norm[W^{k+1,p}(\Th)]{\phi}$.
Then \eqref{est:SD} follows from this modified version of \eqref{est:terms.SD:W1p} and from \eqref{est:terms.SD:Lp}.

(ii) \emph{Estimate on $\WD[h](\vec{\psi})$.} For all $\uv[h]\in \XDz[h]$,
\begin{equation}\label{est:WD.0}
  \begin{aligned}
    \int_\Omega \grD[h]\uv[h](\vec{x})\SCAL\vec{\psi}(\vec{x})\ud\vec{x}
    &=\sum_{T\in\Th}(\grD[h]\uv[h],\vec{\psi}-\vlproj[T]{k}\vec{\psi})_T
    + \sum_{T\in\Th}(\grD[h]\uv[h],\vlproj[T]{k}\vec{\psi})_T
    \\
    &\eqcolon\sum_{T\in\Th}\mathcal A_T+\sum_{T\in\Th}\mathcal B_T.
  \end{aligned}
\end{equation}
The approximation property \eqref{eq:approx} of $\lproj[T]{k}$ with $s=k+1$, $r=0$ and $p'$ instead of $p$ yields
\begin{equation}\label{est:WD.A}
  \begin{aligned}
    \sum_{T\in\Th}|\mathcal A_T|&\lesssim \sum_{T\in\Th}h_T^{k+1}\norm[L^p(T)^d]{\grD[h]\uv[h]}
    \norm[W^{k+1,p'}(T)^d]{\vec{\psi}}
    \\
    &\le h^{k+1}\left(\sum_{T\in\Th}\norm[L^p(T)^d]{\grD[h]\uv[h]}^p\right)^{1/p}
    \left(\sum_{T\in\Th} \norm[W^{k+1,p'}(T)^d]{\vec{\psi}}^{p'}\right)^{1/p'}
    \\
    &=
    h^{k+1}\norm[L^p(\Omega)^d]{\grD[h]\uv[h]}\norm[W^{k+1,p'}(\Th)^d]{\vec{\psi}}.
  \end{aligned}
\end{equation}
By definitions \eqref{eq:DSGDM:grD} and \eqref{eq:GT} of $\grD[h]$ and $\GT$,
and by the orthogonality property \rprop{S2} of $\ST$,
\begin{equation}\label{est:WD.BT}
  \begin{aligned}
    \mathcal B_T =(\GT\uv[T],\vlproj[T]{k}\vec{\psi})_T
    ={}&-(v_T,\GRAD\SCAL \vlproj[T]{k}\vec{\psi})_T+\sum_{F\in\Fh[T]}(v_F,\vlproj[T]{k}\vec{\psi}\SCAL\normal_{TF})_F\\
    ={}&-(v_T,\GRAD\SCAL (\vlproj[T]{k}\vec{\psi}-\vec{\psi}))_T+
    \sum_{F\in\Fh[T]}(v_F,(\vlproj[T]{k}\vec{\psi}-\vec{\psi})\SCAL\normal_{TF})_F\\
    {}&-(v_T,\GRAD\SCAL \vec{\psi})_T+
        \sum_{F\in\Fh[T]}(v_F,\vec{\psi}\SCAL\normal_{TF})_F\\
        ={}&(\GRAD v_T,(\vlproj[T]{k}\vec{\psi}-\vec{\psi}))_T+
        \sum_{F\in\Fh[T]}(v_F-v_T,(\vlproj[T]{k}\vec{\psi}-\vec{\psi})\SCAL\normal_{TF})_F\\
        {}&-(\PiD[h]\uv[h],\GRAD\SCAL \vec{\psi})_T+
            \sum_{F\in\Fh[T]}(v_F,\vec{\psi}\SCAL\normal_{TF})_F\\
            \eqcolon{}&\mathcal B_{T,1}-(\PiD[h]\uv[h],\GRAD\SCAL \vec{\psi})_T+
            \sum_{F\in\Fh[T]}(v_F,\vec{\psi}\SCAL\normal_{TF})_F,
  \end{aligned}
\end{equation}
where we used an integration-by-parts and the definition \eqref{eq:DSGDM:XDz.PiD} of $\PiD[h]$ in the penultimate line. For any interface $F$ with $T_1,T_2$ as neighbouring mesh elements, since $\vec{\psi}\in \vec{W}^{p'}(\oDIV;\Omega)$
we have $\vec{\psi}\SCAL\normal_{T_1F}+\vec{\psi}\SCAL\normal_{T_2F}=0$ on $F$. Moreover, $v_F=0$ whenever $F$ is a boundary
face. Hence 
\[
\sum_{T\in\Th} \sum_{F\in\Fh[T]}(v_F,\vec{\psi}\SCAL\normal_{TF})_F=
\sum_{F\in\Fhi} (v_F,\vec{\psi}\SCAL\normal_{T_1F}+\vec{\psi}\SCAL\normal_{T_2F})_F
+\sum_{F\in\Fhb} (v_F,\vec{\psi}\SCAL\normal_{TF})_F=0.
\]
Summing \eqref{est:WD.BT} over $T\in\Th$ and using the previous relation leads to
\begin{equation}
\sum_{T\in\Th}\mathcal B_T =\sum_{T\in\Th}\mathcal B_{T,1}
-\int_\Omega \PiD[h]\uv[h](\vec{x})\GRAD\SCAL \vec{\psi}(\vec{x})\ud\vec{x}.
\label{est:WD.BT.2}
\end{equation}
Recalling the definition~\eqref{eq:lproj} of $\vlproj[T]{k}$, it is readily inferred that the first term in $\mathcal B_{T,1}$ is zero since $\GRAD v_T\in\GRAD\Poly{l}(T)\subset\Poly{k}(T)^d$. Moreover, using again the approximation properties \eqref{eq:approx.approx.trace} of $\lproj[T]{k}$ with $r=0$, $s=k+1$
and $p'$ instead of $p$, we can write
\begin{align*}
  |\mathcal B_{T,1}|\lesssim{}&
  \sum_{F\in\Fh[T]} \norm[L^p(F)]{v_F-v_T}\norm[L^{p'}(F)^d]{\vlproj[T]{k}\vec{\psi}-\vec{\psi}}\\
  \lesssim{}&
  \sum_{F\in\Fh[T]} h_T^{k+1-\frac{1}{p'}}\norm[L^p(F)]{v_F-v_T}\seminorm[W^{k+1,p'}(T)^d]{\vec{\psi}}\\
  \lesssim{}&
  h^{k+1}\seminorm[W^{k+1,p'}(T)^d]{\vec{\psi}}\left( 
  \sum_{F\in\Fh[T]} h_F^{\frac{1}{p}-1}\norm[L^p(F)]{v_F-v_T}
  \right),
\end{align*}
where we used $h_F\le h_T$ in the last line. Sum over $T\in\Th$ and invoke H\"older's inequality,
the property ${\rm Card}(\Fh[T])\lesssim 1$, and the norm equivalence \eqref{eq:norm.equiv} to deduce
\begin{align}
	\sum_{T\in\Th}|\mathcal B_{T,1}|
	\lesssim{}& h^{k+1}\seminorm[W^{k+1,p'}(\Th)^d]{\vec{\psi}}
	\left(
        \sum_{T\in\Th}\sum_{F\in\Fh[T]}h_F^{1-p}\norm[L^p(F)]{v_F-v_T}^p\right)^{1/p}\nonumber\\
 	\lesssim{}& h^{k+1}\norm[W^{k+1,p'}(\Th)^d]{\vec{\psi}}\norm[L^p(\Omega)^d]{\grD[h]\uv[h]}.
\label{est:WD.B}
\end{align}
Finally, using \eqref{est:WD.A}, \eqref{est:WD.BT.2} and \eqref{est:WD.B} in \eqref{est:WD.0}, we get
$$
\left|\int_\Omega \grD[h]\uv[h](\vec{x})\SCAL\vec{\psi}(\vec{x})\ud\vec{x}
+\int_\Omega \PiD[h]\uv[h](\vec{x})\DIV\vec{\psi}(\vec{x})\ud\vec{x}\right|
\lesssim h^{k+1}\norm[W^{k+1,p'}(\Th)^d]{\vec{\psi}}\norm[L^p(\Omega)^d]{\grD[h]\uv[h]}.
$$
The estimate \eqref{est:WD} follows immediately.
\end{proof}

\begin{remark}[Choice of the gradient reconstruction]\label{rem:GT.vs.grad.rT}
  An inspection of the above proof shows that using \eqref{def:grad.rT} in place of~\eqref{eq:DSGDM:grD} 
  can lead to significant losses in the order of convergence for $\WD[h](\vec{\psi})$ (while the convergence expressed by~\eqref{WD.to.0} still holds true).
  As a matter of fact, with this choice one would have to replace throughout the proof $\vlproj[T]{k}\vec{\psi}$ by the $L^2$-orthogonal projection of $\vec{\psi}$ on $\GRAD\Poly{k+1}(T)$.
  The latter quantity has optimal approximation properties only if either $k=0$ (since $\GRAD\Poly{1}(T)=\Poly{0}(T)^d$) or there exists $w\in L^{p'}(\Omega)$ such that $\vec{\psi}_{|T}=\GRAD w_{|T}$ for all $T\in\Th$.
  Recalling Theorem \ref{thm:error.est.example} with $p=2$ and $\vec{\sigma}(\vec{x},u,\GRAD u)=\GRAD u$, we see that for the Poisson equation, $\WD[h]$ is applied to $\vec{\psi}=\GRAD u$. In this case,
  the gradient reconstruction \eqref{def:grad.rT} leads to optimal convergence rates.
  However, there is a real loss of estimate for more general problems for which error estimates are written in terms of $\WD[h](\matr{\Lambda}\GRAD u)$ (for anisotropic linear diffusion, see \cite[Theorem 2.29]{gdm}) or $\WD[h](|\GRAD u|^{p-2}\GRAD u)$
  (for the $p$-Laplace equation, see Theorem \ref{thm:error.est.example}).
\end{remark}


\section*{\refname}

\bibliographystyle{elsarticle-harv}
\bibliography{gdmpho}

\end{document}